\documentclass{amsart}
\usepackage{amsmath}
\usepackage{newtxtext}
\RequirePackage[usenames,dvipsnames]{color}
\usepackage{mathrsfs}
\usepackage{amssymb}
\usepackage[hidelinks]{hyperref}
\usepackage{graphicx}
\usepackage{enumitem} 
\usepackage{txfonts} 
\usepackage{comment}

\hypersetup{
colorlinks=false,
linktoc=all}

\usepackage[all,pdftex, cmtip]{xy}
\newdir{ >}{{}*!/-10pt/@{>}}

\usepackage{tikz-cd}
\usetikzlibrary{decorations.pathmorphing}

\DeclareMathAlphabet{\mathbbe}{U}{bbold}{m}{n}

\setlist{}
\setenumerate{leftmargin=*,labelindent=0\parindent}
\setitemize{leftmargin=*,labelindent=\parindent}

\newtheorem{thm}{Theorem}[section]
\newtheorem{lem}[thm]{Lemma}
\newtheorem{prop}[thm]{Proposition}
\newtheorem{cor}[thm]{Corollary}

\theoremstyle{definition}
\newtheorem{defn}[thm]{Definition}
\newtheorem{ex}[thm]{Example}

\newtheorem{obs}[thm]{Observation}
\newtheorem{conv}[thm]{Convention}

\theoremstyle{remark}
\newtheorem{rmk}[thm]{Remark}

\makeatletter
\@addtoreset{section}{part}
\makeatother  

\makeatletter
\let\c@equation\c@thm
\makeatother
\numberwithin{equation}{section}

\newcommand{\id}{\mathrm{id}}
\newcommand{\op}{\mathrm{op}}
\newcommand{\Ho}{\mathrm{Ho}}

\newcommand{\kto}{\rightsquigarrow}
\newcommand{\Ch}{\cat{Ch}}
\newcommand{\Tot}{\mathrm{Tot}}

\newcommand{\cat}[1]{\textup{\textsf{#1}}}
\newcommand{\fun}[1]{\textup{#1}}
\newcommand{\mr}[1]{\mathrm{#1}}

\newcommand{\cA}{\mathcal{A}}
\newcommand{\cB}{\mathcal{B}}
\newcommand{\cC}{\mathcal{C}}
\newcommand{\cD}{\mathcal{D}}

\newcommand{\RR}{\mathbb{R}}

\newcommand{\DDelta}{\mathbbe{\Delta}}

\newcommand{\defeq}{\mathrel{:=}}
\newcommand{\eqdef}{\mathrel{=:}}

\makeatletter
\def\makeslashed#1#2#3#4#5{#1{\mathpalette{\sla@{#2}{#3}{#4}}{#5}}}

\def\@mathlower#1#2#3{\setbox0=\hbox{$\m@th#2#3$}\lower#1\ht0\box0}
\def\mathlower#1#2{\mathpalette{\@mathlower{#1}}{#2}}
\makeatother

\newcommand{\simto}{\xrightarrow{\smash{\mathlower{0.6}{\simeq}}}}

\newcommand{\CR}{\mathrm{cr}}
\newcommand{\Fun}{\textup{Fun}}

\ifx\color@rgb\@undefined\else
  \definecolor{maroon}{rgb}{0.5,0,0}
\fi
\ifx\color@rgb\@undefined\else
  \definecolor{violet}{rgb}{0.3,0,0.7}
\fi

\begin{document}

\title{Directional derivatives and higher order chain rules for abelian functor calculus}
\author[Bauer]{Kristine Bauer}
\author[Johnson]{Brenda Johnson}
\author[Osborne]{Christina Osborne}
\author[Riehl]{Emily Riehl}
\author[Tebbe]{Amelia Tebbe}


\address{Department of Mathematics and Statistics\\University of Calgary\\2500 University Dr.~NW\\Calgary, Alberta, Canada T2N 1N4}
\email{bauerk@ucalgary.ca}

\address{Department of Mathematics\\Union College\\Schenectady, NY 12308}
\email{johnsonb@union.edu}

\address{Department of Mathematics\\University of Virginia\\141 Cabell Drive\\Charlottesville, VA 22904}
\email{cdo5bv@virginia.edu}

\address{Department of Mathematics\\Johns Hopkins University \\ 3400 N Charles Street \\ Baltimore, MD 21218}
\email{eriehl@math.jhu.edu}

\address{Department of Mathematics\\University of Illinois Urbana-Champaign\\1409 W.~Green Street\\Urbana, IL 61801}
\email{tebbe2@illinois.edu}

\date{\today}

\begin{abstract} 
 In this paper, we consider \emph{abelian functor calculus}, the calculus of functors of abelian categories established by the second author and McCarthy.  
 We carefully construct a category of abelian categories and suitably homotopically defined functors, and show that this category, equipped with the directional derivative,  is a 
  {\it cartesian differential category} in the sense of Blute, Cockett, and Seely. This provides an abstract framework that makes certain analogies between classical and functor calculus explicit.  Inspired by Huang, Marcantognini, and Young's chain rule for higher order directional derivatives of functions, we define a higher order directional derivative for functors of abelian categories. We show that our higher order directional derivative is related to the iterated partial directional derivatives of the second author and McCarthy by a Fa\`a di Bruno style formula.  We obtain a higher order chain rule for our directional derivatives  using a feature of the cartesian differential category structure, and with this  provide a formulation for the $n$th layers of the Taylor tower of a composition of functors $F\circ G$ in terms of the derivatives and directional derivatives of $F$ and $G$, reminiscent of similar formulations for functors of spaces or spectra by Arone and Ching.  Throughout, we provide explicit chain homotopy equivalences that tighten previously established quasi-isomorphisms for properties of abelian functor calculus.

\end{abstract}

\maketitle

\setcounter{tocdepth}{1}
\tableofcontents

\section{Introduction}\label{sec:function-calculus}

\subsection{Chain Rules}
When studying the calculus of functions,  the chain rule  for the first derivative of a composition of functions is given by the familiar formula
\[ (f\circ g)'= (f'\circ g)\cdot g'.\]
By repeatedly applying this formula, one can derive the classical Fa\`{a} di Bruno formula for the $n$th derivative of the composite $f\circ g$:
\[(f\circ g)^{(n)}(x)=\sum {\frac{n!}{k_1!k_2!\ldots k_n!}}f^{(k)}(g(x))\left ({\frac{g'(x)}{1!}}\right )^{k_1}\ldots \left ({\frac{g^{(n)}(x)}{n!}}\right )^{k_n}\]
where the sum is taken over all non-negative integer solutions to 
\[
k_1+2k_2+\ldots +nk_n=n
\]
and $k=k_1+\ldots+k_n$.  
An alternative higher order chain rule can be obtained by using the {\it directional} derivative of $f$ at $x$ along $v$:
\[ \nabla f(v;x) = \lim_{t\to 0} \frac{1}{t}\left[ f(x+tv)-f(x)\right].  \]
Huang, Marcantognini, and Young \cite{HMY} define \emph{higher order directional derivatives} $\Delta_nf$, functions of $(n+1)$-variables with $\Delta_1f = \nabla f$, with which they obtain a concise formulation for the $n$th derivative of a composition
\begin{equation}\label{eq:HMY1} (f\circ g)^{(n)}(x)=\Delta_nf(g^{(n)}(x),\ldots, g'(x); g(x)),\end{equation}
when $g$ is a function of a single variable, or
\begin{equation}\label{eq:HMY2}\Delta_n(f\circ g)= \Delta_nf(\Delta_ng,\ldots, \Delta_1g; g),\end{equation} 
in the general case.

For homotopy functors (such as functors from spectra to spectra or spaces to spaces that preserve weak equivalences), Goodwillie's calculus of functors builds a \emph{Taylor tower} of ``polynomial'' functors and natural transformations:
\[ F \to \cdots \to P_nF \to P_{n-1}F \to \cdots \to P_0F \]
where each $P_nF$ is an $n$-excisive functor that approximates $F$ in a range of homotopy groups depending on the functor $F$ and the connectivity of the input \cite{G3}.  The homotopy fiber of the map $P_nF\to P_{n-1}F$ is the $n$th \emph{layer} of the tower, $D_nF$.  For homotopy functors from spectra to spectra, the layers are of the form $D_nF\simeq \partial_n F\wedge_{h\Sigma_n} X^{\wedge n}$, where $\partial_nF$ is a spectrum with an action of the the $n$th symmetric group, $\Sigma_n$.  This spectrum is called the $n$th derivative of $F$.

Several versions of  chain rules for the derivatives of functor calculus have been developed.  Notably,  Arone and  Ching \cite{AC} derived a chain rule for the derivatives $\partial_nF$ using the fact that for functors of spaces or spectra, the symmetric sequence $\{\partial_nF\}$ is a module over the operad formed by the derivatives of the identity functor of spaces.  This generalized earlier work of  Klein and  Rognes \cite{KR} that established a chain rule for first derivatives.  For functors of spectra, a chain rule for the derivatives is given by Ching \cite{C}. For functors of spaces,  Yeakel has developed an alternative method for deriving these chain rules that does not require passage to spectra \cite{yeakel-thesis}.  

\subsection{Chain Rules in Abelian Functor Calculus}

The chain rule results of the preceding paragraph use Goodwillie's original formulation of the Taylor tower.  The second author and  McCarthy \cite{JM:Deriving} defined  Taylor towers for not necessarily additive functors between abelian categories using  the classical cross effect functors of Eilenberg and Mac Lane \cite{EM}.  This approach has been generalized and applies to a wide variety of contexts \cite{BJM}, \cite{WIT1}.  In \cite{JM:Deriving}, Johnson and McCarthy defined a notion of directional derivative $\nabla F(V;X)$ for a functor $F$ valued in an abelian category and proved  the analog of Huang, Marcantognini, and Young's chain rule in degree one \cite[Proposition 5.6]{JM:Deriving}:
\[\nabla (F\circ G) \simeq \nabla F(\nabla G; G),\]
 under the hypothesis that the functor $G$ preserves the zero object.

The present work was motivated by the goal of proving an analog of the higher order directional derivative chain rule of Huang, Marcantognini, and Young for the abelian functor calculus of Johnson and McCarthy.  Achieving this goal required first dealing with the question of how to define a higher order directional derivative in this context.  To get a sense of the choices involved, consider the functor $\nabla F(V;X)$.  This is a functor of two variables, so the first choice to be made  in defining a second order directional derivative is  whether one should differentiate with respect to one of the variables (a partial derivative) 
or both variables simultaneously (a total derivative).

 For the first possibility, one notes that $\nabla F(V;X)$ is already a linear functor with respect to its first variable, so that differentiating with respect to that variable yields nothing new.  So, for the partial derivative approach, we can restrict our attention to taking the directional derivative of $\nabla F(V;X)$ with respect to $X$ to obtain a second order partial directional derivative $\nabla^2F$ that is a functor of $3$ variables.  Iterating this process yields functors $\nabla^nF$ of $(n+1)$-variables defined in \cite{JM:Deriving}, which we refer to as \emph{iterated partial directional derivatives}.
 
  If we follow the total derivative approach by treating $\nabla F$ as a functor whose domain is a product category $\cB\times \cB$, then we obtain as a second order total directional derivative a functor $\nabla^{\times 2}F = \nabla (\nabla F)$ whose source category is $\cB^{4}$.  But, this leads to some redundancy, as we end up differentiating in some directions repeatedly.  The ``correct'' approach, at least if one wants to define the analog of the Huang-Marcantognini-Young higher order directional derivatives, lies somewhere in between.  
  
To define our second order directional derivative $\Delta_2F$ of $F$, we restrict $\nabla^{\times 2}F$ along a diagonal functor $L_2 \colon \cB^3 \to \cB^4$ that takes the triple of objects $(V_2, V_1, X)$ to the ordered pair of ordered pairs $((V_2, V_1), (V_1, X))$.  With this definition we are able to prove that for a pair of composable functors $F$ and $G$
\[
\Delta_2(F\circ G)(V_2, V_1; X)\simeq \Delta_2F(\Delta_2G(V_2, V_1; X), \Delta_1 G(V_1; X); G(X)).
\]
Inductively, we  define the \emph{$n$th order directional derivative} $\Delta_nF$ and prove the desired analog of the Huang-Marcantognini-Young chain rule equation \eqref{eq:HMY2}:

{
\renewcommand{\thethm}{\ref{t:higherchainrule}}
\begin{thm}  For a composable pair $F\colon\cB\kto \cA$ and $G\colon\cC\kto \cB$ of functors of abelian categories, there is a chain homotopy equivalence between the $n$th directional derivatives
\[ \Delta_n (F\circ G)(V_n, \ldots , V_1; X) \simeq \Delta_n F(\Delta_nG(V_n, \ldots , V_1; X), \ldots , \Delta_1G(V_1; X); G(X)).\]
\end{thm}
\addtocounter{thm}{-1}
}

Of course, this ``in-between" definition of the higher order directional derivatives raises a natural question -- how is it  related to the more obvious choices for defining higher order directional derivatives?  In the case of $\nabla^{\times n}F$, the functor $\Delta_nF$ is obtained by restricting along a diagonal functor $L_n$ as in the degree $2$ case.  The case of $\nabla^nF$ proves more interesting.  We show that $\Delta_nF$ and $\nabla^nF$ satisfy a Fa\`{a} di Bruno-type relationship:
{\renewcommand{\thethm}{\ref{thm:restatedgoal}}
\begin{thm} For a functor $F$ between abelian categories, there is a chain homotopy equivalence
\[\Delta_nF(V_n,\ldots, V_1;X) \simeq \bigoplus\frac{n!}{k_1!k_2!\cdots k_n!}\left(\frac{1}{1!}\right)^{k_1}\cdots \left(\frac{1}{n!}\right)^{k_n} \nabla^{k_1+\cdots+k_n}F(V_n^{k_n},\ldots, V_1^{k_1}; X)
\]
 where the sum is over non-negative integer solutions to the equation $k_1+2k_2+\cdots +nk_n=n$.
\end{thm}
\addtocounter{thm}{-1}
}
A more concise formulation of this result, given in \S\ref{sec:faa}, makes use of the fact that the Fa\`{a} di Bruno coefficients count the number of partitions of a set of $n$ elements.  
 Using these results, we define ordinary (non-directional) derivatives for functors of modules over a commutative ring $R$ and prove a chain rule for these derivatives which is analogous to the Huang-Marcantognini-Young rule from equation \eqref{eq:HMY1}:
 
 {\renewcommand{\thethm}{\ref{cor:faachainrule}}
 \begin{cor}
For a composable pair of functors $F\colon{\Ch\cB} \to {\Ch\cA}$ and $G\colon\mathcal{M}\mathrm{od}_R\to \Ch{\cB}$ of functors of abelian categories, the $n$th derivative \[\frac{d^n}{dR^n}(F\circ G)(X)\] is chain homotopy equivalent to
\[  \bigoplus\frac{n!}{k_1!k_2!\cdots k_n!}\left(\frac{1}{1!}\right)^{k_1}\cdots \left(\frac{1}{n!}\right)^{k_n} \nabla^{k_1+\cdots+k_n}F\left(\frac{d^{k_n}}{dR^{k_n}}G(X),\ldots, \frac{d^{k_1}}{dR^{k_1}}G(X); G(X)\right)\]
where the sum is over non-negative integer solutions to the equation $k_1+2k_2+\cdots +nk_n=n$.
\end{cor}
\addtocounter{thm}{-1} }

 \subsection {A categorical context for abelian functor calculus}  \label{s:context}
 Although our original motivation for this project was to prove higher order chain rules, we found that  a significant part of the challenge in  doing so involved placing these chain rules and the abelian functor calculus itself in an appropriate categorical context.  This was developed in two stages, using the concepts of Kleisli categories and cartesian differential categories.   
 
For a functor $F\colon\cB\rightarrow \cA$ between two abelian categories, the abelian functor calculus constructs a degree $n$ approximation $P_nF\colon\cB\rightarrow\Ch\cA$, where $\Ch\cA$ is the category of chain complexes in $\cA$ concentrated in non-negative degrees.   
For a composable pair of functors $F\colon\cB\rightarrow \cA$ and $G\colon\cC\rightarrow \cB$, this results in a pair of functors $P_nF\colon\cB\rightarrow \Ch\cA$ and $P_nG\colon\cC\rightarrow \Ch\cB$ which no longer appear to be composable.  In \cite{JM:Deriving}, the authors form the composition $P_nF\circ P_nG$ by ``prolonging" $P_nF$ to a functor $\Ch\cB \to \Ch(\Ch\cA)$, composing the prolongation with $P_nG$ and taking the total complex of the resulting double complex.  
We observe in section 3 that this process is precisely the composition in the Kleisli category associated to a particular pseudomonad $\Ch$ acting on the $2$-category of abelian categories, and use $\cat{AbCat}_{\Ch}$ to denote that Kleisli category.  
 
 In functor calculus, the use of terms such as ``calculus'' and ``derivatives'' has been justified by pointing out strong formal resemblances to  the classical constructions from analysis and undergraduate calculus.  Although these analogies are compelling, one wonders if there is  a deeper justification.  
 In the course of investigating the higher order chain rule for directional derivatives, we discovered that the properties of the directional derivative for abelian functor calculus correspond exactly to the axioms defining the cartesian differential categories of Blute, Cockett, and Seely \cite{BCS:Cartesian}.  This concept captures the notion of differentiation in a wide variety of contexts (the  category of real vector spaces and smooth maps equipped with the usual differential operator is a standard example of a cartesian differential category)     and provides an explanation for some of the formal resemblances.
 
 We prove:

{\renewcommand{\thethm}{\ref{cor:CDC}}
\begin{cor} The homotopy category
$\Ho\cat{AbCat}_{\Ch}$ is a cartesian differential category.
\end{cor}
 \addtocounter{thm}{-1} }
 
In this statement, $\cat{HoAbCat}_{\Ch}$ is  the  homotopy category of $\cat{AbCat}_{\Ch}$ obtained by inverting pointwise chain homotopy equivalences.  The directional derivative $\nabla$ gives this category its cartesian differential category structure.     Fitting calculus into this context led to additional insight stemming from the fact that  a cartesian differential category is also an example of a tangent category, as defined by Cockett and Cruttwell \cite{CC}.  Tangent categories are characterized by the existence of an endofunctor $T$ satisfying  essential properties of tangent bundles for manifolds. Our first proof of the chain rule of Theorem \ref{t:higherchainrule} did not make use of the tangent category structure on  $\cat{HoAbCat}_{\Ch}$, but the proof presented here, which does use the endofunctor $T$, provides a conceptual simplification. As the first-order chain rule can be seen to encode the functoriality of the derivative, the $n$th order chain rule can be extracted from the functoriality of the iterated endofunctor $T^n$, which is an easily derived consequence of our Corollary \ref{cor:CDC}.

Proving that $\cat{HoAbCat}_{\Ch}$ is a cartesian differential category required reworking many of the constructions and results for abelian functor calculus established in \cite{JM:Deriving}.  Driving these   challenges was the fact that chain rules in the abelian context do not come in the form of isomorphisms. For example, applying the definition of $\nabla$ (see Definition \ref{defn:dir-derivative} or \ref{defn:dir-derivative-JM}) to $F\circ G$ produces a  chain complex that is not isomorphic to $\nabla F(\nabla G, G)$ in general.  Instead, when $G$ is reduced ($G(0)\cong 0$), \cite{JM:Deriving} prove that there is a quasi-isomorphism between $\nabla (F\circ G)$ and $\nabla F(\nabla G, G)$.  

The axioms of a cartesian differential category require  that $\nabla(F\circ G)$ and $\nabla F(\nabla G, G)$ be isomorphic.    This suggests that  we attempt to realize  $\cat{AbCat}_{\Ch}$ as a cartesian differential category by inverting quasi-isomorphisms.  But  because the functors we consider between abelian categories are far from being exact, when we invert quasi-isomorphisms in $\cat{AbCat}_{\Ch}$, the composition of morphisms is no longer well-defined.  However, the somewhat delicately defined composition in our Kleisli category preserves chain homotopy equivalences, so we define  $\cat{HoAbCat}_{\Ch}$ to be the category that inverts these instead.  Because of this we needed to upgrade many  definitions and results  of \cite{JM:Deriving} from quasi-isomorphisms to chain homotopy equivalences.  

Revising results of \cite{JM:Deriving} to work up to chain homotopy equivalence rather than  quasi-isomorphism was a relatively straightforward process, with the notable exception of the chain rule itself.  That required proving the following proposition, which refines Lemma 5.7 of \cite{JM:Deriving}.

{\renewcommand{\thethm}{\ref{lem:5.7}}
\begin{prop} For any composable pair of functors $F \colon \cB \to \Ch\cA$ and $G \colon \cC \to \Ch\cB$ with $G$ reduced, there is a chain homotopy equivalence
\[ D_1(F \circ G) \simeq D_1F \circ D_1G.\]
\end{prop}
\addtocounter{thm}{-1} 
}

To prove this, we prove a technical general result that might be of independent interest (or might already be in the literature):

{\renewcommand{\thethm}{\ref{thm:bicomplex-htpy-equiv}}
\begin{thm}   Let $\iota\colon A_{\bullet, \bullet}\rightarrow B_{\bullet,\bullet}$ be a morphism of first-quadrant bicomplexes that admits a row-wise strong deformation retraction. 
Then $\iota$ induces a chain homotopy equivalence of total complexes $\Tot(A_{\bullet, \bullet})\rightarrow \Tot(B_{\bullet, \bullet})$.  \end{thm}
\addtocounter{thm}{-1} }

 We include a proof of this result in Appendix A. The proof defines an explicit chain homotopy on the total complex  that does not come from a chain homotopy of bicomplexes, as defined in \cite[5.7.3]{Weibel}, because the row-wise retraction and chain homotopies are not assumed to be natural with respect to the vertical differentials. The precise conditions are stated in Definition \ref{defn:row-wise-retraction}. As a corollary, it follows that:
 
{\renewcommand{\thethm}{\ref{cor:chain-htpy-equiv}}
 \begin{cor} Let $A_{\bullet, \bullet}$ be a first-quadrant bicomplex so that every row except the zeroth row $A_{0,\bullet}$ is contractible. Then the natural inclusion $A_{0,\bullet}\hookrightarrow\Tot(A)_\bullet$ is a chain homotopy equivalence.
\end{cor}
\addtocounter{thm}{-1} }

   The proof of Proposition 5.7 appears in Appendix B.  Combining Proposition 5.7 with some other results, we prove

{\renewcommand{\thethm}{\ref{thm:diffcat}}
\begin{thm}
For any pair of functors 
$F \colon \cB \to \Ch\cA$ and $G \colon \cC \to \Ch\cB$, there is a chain homotopy equivalence
\[\nabla (F\circ G) (V; X)\simeq \nabla F(\nabla G(V; X); G(X)).\]
\end{thm}
\addtocounter{thm}{-1}}
This improves on Proposition 5.6  of \cite{JM:Deriving} in two ways -- by replacing the quasi-isomorphism with a chain homotopy equivalence, and by removing the condition that $G$ be a reduced functor.  

 \subsection{Organization of the paper}
 
 In section 2, we define and review properties of the cross effect functors.  These are the building blocks for the  Taylor towers of abelian functor calculus, including the linearization and directional derivative functors.    Section 3 is used to define the Kleisli  and homotopy categories in which the constructions and main results of the paper take place.  In section 4, we show how the Taylor towers of abelian functor calculus are built and establish their fundamental properties.  This treatment differs from that of \cite{JM:Deriving} in two important ways.  As discussed in section \ref{s:context}, we need the fundamental properties of the tower to hold up to chain homotopy equivalence instead of quasi-isomorphism.  In addition, we have streamlined the means by which terms in the abelian functor calculus tower are defined.  We do so  by identifying a comonad $C_n$ on the category of all functors between a fixed pair of abelian categories, and defining $P_nF$ directly as a resolution of $F$ by this comonad.  In \cite{JM:Deriving}, this approach was only used for the category of  reduced functors. An extra step involving a mapping cone was used to extend this construction to all functors. 
 
 In section 5, we start working out the main ingredients for our chain rules -- we define the linearization functor $D_1$ and derive its essential properties.  We use these to define the directional derivative for abelian functor calculus in section 6, and prove that it gives $\cat{HoAbCat}_{\Ch}$ the structure of a cartesian differential category.  In section 7, we define the iterated partial directional derivatives and the higher order directional derivatives and prove that they are related by the Fa\`a di Bruno formula of Theorem \ref{thm:restatedgoal}.  We prove the analog of the higher order chain rule of \cite{HMY} in section 8, using the tangent category structure of $\cat{HoAbCat}_{\Ch}$.  Finally, in section 9, we discuss how the directional derivatives are related to a notion of ordinary derivatives for abelian functor calculus and use the main result of section 7 to prove the Fa\`a di Bruno-style formula for these derivatives in Corollary \ref{cor:faachainrule}.

 \subsection*{Conventions}
 
This paper introduces a number of constructions on functors (e.g., the higher order directional derivatives) and investigates their behavior with respect to composition, product constructions, and so forth. A convenient context to describe these operations involves a large category whose objects are abelian categories  and whose morphisms are functors between them; the precise construction of this category, which is somewhat delicate, is given in \S\ref{sec:kleisli}. There likely exist more than a set's worth of functors between any fixed pair of non-small abelian categories, so if our constructions on functors are to be interpreted globally, they need to take place in some extension of the usual ZFC axioms for set theory: e.g., by assuming there exists a hierarchy of inaccessible cardinals. See \cite{shulman} for a friendly discussion of the myriad possible choices.

\subsection*{Acknowledgments}

The authors thank the Banff International Research Station for hosting the second Women in Topology workshop, which brought us together for a week in which many of the results in this paper were proven, and the Pacific Institute for the Mathematical Sciences, which provided travel support for us to work together in Calgary.  We also thank Maria Basterra and Kathryn Hess for their roles in organizing the Women in Topology workshop.   The second author is grateful for support from the Union College Faculty Research Fund.  The fourth author is  grateful for support from the National Science Foundation through  DMS-1551129.

Tslil Clingman suggested that chain complexes might define a monad on abelian categories and Robin Cockett pointed out that non-reduced functors may fail to preserve chain complexes. This is why the monad of \S\ref{sec:kleisli} is pseudo and not strict. Geoffrey Cruttwell explained the benefits of thinking of a cartesian differential category as a tangent category, a perspective which led to a simplified proof of Theorem \ref{t:higherchainrule}.  We thank Randy McCarthy for helpful conversations in starting this project, and for the ideas that inspired it.

\section{Cross effects for functors}\label{sec:cross-effects}

Classically, the term {\it cross effect} was used to describe the combined effects of two or more forces, or to describe the difference between the quantities $f(x+y)$ and $f(x)+f(y)$.  For \emph{reduced} functions satisfying $f(0)=0$, the condition that $f(x+y)=f(x)+f(y)$ for all $x$ and $y$ is equivalent to linearity.  Thus, the cross effect function, defined by 
\[ \CR_2f(x,y)=f(x+y)-f(x)-f(y),\] 
measures the failure of a reduced function $f$ to be linear; this failure is called the {\it deviation} by Eilenberg and Mac Lane \cite{EM}.  In this section, we will study an analogous notion for functors from a pointed category to an abelian category.  The cross effects were first extended to functors of additive categories by Eilenberg and Mac Lane \cite{EM}, and these became the fundamental building blocks of abelian functor calculus as developed by Johnson and McCarthy \cite{JM:Deriving}.  We recall the definition of the cross effect and summarize the properties that we will need for cross effect functors.

For the duration of this section, let $\cB$ be a category with a \emph{basepoint}, i.e.,~an initial object $\star$ which is also terminal, and finite coproducts, denoted by $\vee$.  Let $\cA$ be an abelian category with zero object $0$ and biproducts denoted $\oplus$.

\begin{defn}[\cite{EM}]\label{def:CR}  The \emph{$n$th cross effect} of a functor $F \colon \cB \to \cA$ 
is the $n$-variable functor $\CR_nF\colon \cB^{n} \to \cA$ defined recursively by
\[ F(X)\cong F(\star)\oplus \CR_1F(X)\]
\[ \CR_1F(X_1\vee X_2) \cong \CR_1F(X_1)\oplus \CR_1 F(X_2) \oplus \CR_2 F(X_1, X_2)\]
and in general,
\begin{multline*} \CR_{n-1}F(X_1\vee X_2, X_3,  \ldots , X_n)\cong \CR_{n-1}F(X_1, X_3, \ldots , X_n) \oplus \CR_{n-1}F(X_2, X_3, \ldots , X_n) \\ \oplus \CR_nF(X_1, X_2, \ldots, X_n).\end{multline*}
\end{defn}

Despite the asymmetry in this definition, the $n$th cross effect is symmetric in its $n$-variables \cite[Proposition 1.2]{JM:Deriving}.

\begin{rmk} In a category with a basepoint and finite coproducts, each coproduct inclusion $X \hookrightarrow X \vee Y$ is a split monomorphism; the relevance of the retraction is that split monomorphisms are preserved by any functor. In an abelian category, split monomorphisms extend to split short exact sequences. In particular, it makes sense to define the cross effects as the direct sum complements of Definition \ref{def:CR}.  Eilenberg and Mac Lane defined  cross effects  as the images of certain homomorphisms, but the direct sum complement definition is more useful for the types of computations we wish to do.  
\end{rmk}

\begin{ex}\label{ex:add-constant} Fix an object $A$ in an abelian category $\cA$ and consider the functor $F\colon\cA\to \cA$ defined by $F(X) \defeq A \oplus X$. Since $F(0) \cong A$, it follows that $\CR_1 F \cong \id$. Now
\[ X \oplus Y \oplus \CR_2F(X,Y)\cong \CR_1F(X) \oplus \CR_1F(Y) \oplus \CR_2F(X \oplus Y) \cong \CR_1F(X \oplus Y) \cong X \oplus Y\] implies that $\CR_2F(X,Y)\cong 0$, so all higher cross effects must also vanish.
\end{ex}

Definition \ref{def:CR} is \emph{functorial} in the sense that a natural transformation $F \to G$ induces a naturally-defined map $\CR_nF \to \CR_n G$. To precisely describe this functoriality, write:
\begin{itemize}
\item $\Fun(\cB,\cA)$ for the category of all functors from $\cB$ to $\cA$; 
\item $\Fun_*(\cB,\cA)$ for the category of all \emph{strictly reduced} functors from $\cB$ to $\cA$, that is, functors for which $F(\star)\cong 0$; and
\item $\Fun_*(\cB^{n},\cA)$ for the category of all \emph{strictly multi-reduced} functors from $\cB^{n}$ to $\cA$, that is, functors  for which $F(X_1, \ldots , X_n) \cong 0$ if any $X_i=\star$.
\end{itemize}
An easy exercise \cite[Proposition 1.2]{JM:Deriving} shows that $\CR_n F$ is an object in $\Fun_*(\cB^{n},\cA)$.

 \begin{lem}\label{lem:CR-idempotency} For any $n \geq 1$ and any $F \colon \cB \to \cA$, $\CR_nF \cong \CR_n\CR_1F$.
 \end{lem}
 \begin{proof}
The isomorphism $F(\star) \cong F(\star) \oplus \CR_1F(\star)$ implies that $\CR_1F(\star)\cong 0$. Now
\[ \CR_1F(X) \cong \CR_1F(\star) \oplus \CR_1(\CR_1F(X)) \cong \CR_1^2F(X)\] proves that $\CR_1 F \cong \CR_1^2F$. The general case now follows immediately from Definition \ref{def:CR} by induction.
\end{proof}

Importantly, the cross effect functors admit adjoints:

\begin{prop}\label{prop:coreflective} The inclusion of the category of strictly reduced functors into the category of not-necessarily reduced functors admits a right adjoint, namely the first cross effect functor:
\[
\begin{tikzcd}
\Fun_*(\cB,\cA) \arrow[r, yshift=0.5em] \arrow[r, phantom, "\bot"] & \Fun(\cB,\cA). \arrow[l, yshift=-0.5em,  "\CR_1"]
\end{tikzcd}
\] The component of the counit of this adjunction  at a functor $G \in \Fun(\cB,\cA)$ is the natural inclusion $\CR_1 G \hookrightarrow G$. 
\end{prop}
\begin{proof}
The universal property of the adjunctions asserts that if $F$ is reduced and $G$ is not necessarily reduced then any natural map $\alpha \colon F \to G$ factors uniquely through the inclusion $\CR_1G \hookrightarrow G$.  Let $!$ denote the natural transformation to the basepoint object in $\cB$, $!\colon\id\to \star$, and consider the diagram:
\[
\begin{tikzcd}
\CR_1 FX \defeq \mr{ker}\ \arrow[r, tail, "\cong"] \arrow[d, dashed, "\exists!"']  & FX \arrow[d, "\alpha_X"] \arrow[r, "F!"] & F\star \arrow[d, "\alpha_\star"] \\ \CR_1 GX \defeq \mr{ker}\ \arrow[r, tail] & GX \arrow[r, "G!"'] & G\star
\end{tikzcd}
\]
The right-hand square commutes by naturality of $\alpha$. The kernels of the right-hand horizontal maps define the first cross effects. By commutativity of the right-hand square, there exists a unique map $\CR_1FX \to \CR_1GX$ so that the left-hand square commutes. But because $F\star=0$, the top left-hand map is an isomorphism, which gives us the desired unique factorization.
\end{proof}

\begin{rmk} A full subcategory is \emph{coreflective} when the inclusion admits a right adjoint.
Proposition \ref{prop:coreflective} asserts that  $\Fun_*(\cB,\cA)$ defines a coreflective subcategory of $\Fun(\cB,\cA)$.  Because the left adjoint is full and faithful the unit is necessarily an isomorphism. By abstract nonsense, the inclusion of a coreflective subcategory is necessarily \emph{comonadic}, that is, $\Fun_*(\cB,\cA)$ is the category of coalgebras for the idempotent comonad $\CR_1$ acting on $\Fun(\cB,\cA)$. This formalism gives a characterization of the reduced functors: namely $F$ is reduced if and only if the natural map $\CR_1 F \hookrightarrow F$ is an isomorphism. 
 \end{rmk}

 Johnson and McCarthy observe that when the construction of the $n$th cross effect is restricted to reduced functors, it is right adjoint to pre-composition with the diagonal functor \cite[Example 1.8]{JM:Deriving}. Proposition \ref{prop:coreflective} allows us to extend this adjunction to non-reduced functors:

\begin{cor}\label{cor:comonad} There is an adjunction 
\[
\begin{tikzcd}
\Fun(\cB,\cA) \arrow[r, yshift=-0.5em, "\CR_n"'] \arrow[r, phantom, "\bot"]  & \Fun_*(\cB^n,\cA) \arrow[l, yshift=0.5em, "\Delta^*"']
\end{tikzcd}
\]
between the $n$th cross effect functor and the functor given by precomposing with the diagonal $\Delta \colon \cB \to \cB^{n}$, inducing a comonad $C_n$ on $\Fun(\cB,\cA)$ defined by \[C_n F(X) \defeq \CR_nF(X,\ldots,X).\]
\end{cor}
\begin{proof}
By Lemma \ref{lem:CR-idempotency}, $\CR_n \cong \CR_n\CR_1$ for $n \geq 1$. Thus, the adjunction $\Delta^* \dashv \CR_n$ is the composite of the adjunctions
\[
\begin{tikzcd}
\Fun(\cB,\cA) \arrow[r, yshift=-0.5em, "\CR_1"'] \arrow[r, phantom, "\bot"] & \Fun_*(\cB,\cA) \arrow[l, yshift=0.5em] \arrow[r, yshift=-0.5em, "\CR_n"'] \arrow[r, phantom, "\bot"] & \Fun_*(\cB^n,\cA) \arrow[l, yshift=0.5em, "\Delta^*"']
\end{tikzcd}
\]
of Proposition \ref{prop:coreflective} with the adjunction of \cite[Example 1.8]{JM:Deriving}.
\end{proof}

\begin{rmk}\label{rmk:counit} Note that by \cite[Example 1.8]{JM:Deriving}, the counit of the adjunction $\Delta^* \dashv \CR_n$ between $\Fun_*(\cB, \cA)$ and $\Fun_*(\cB^n, \cA)$ is 
\[
\begin{tikzcd}
 \CR_n G(X, \ldots, X)  \arrow[r, tail] & G(\vee_{i=1}^n X) \arrow[r, "G(+)"] & G(X)
\end{tikzcd}
\]
where $+\colon\vee_n M \to M$ is  the ``fold'' map from the coproduct.  Putting this together with Proposition \ref{prop:coreflective}, we see that the counit of the adjunction of Corollary \ref{cor:comonad} for functors which are not necessarily reduced is the map $\epsilon$ defined to be the composite:
\[
\begin{tikzcd}
 \CR_n \CR_1G(X, \ldots, X)  \arrow[r, tail] & \CR_1G(\vee_{i=1}^n X) \arrow[r, "G(+)"] & \CR_1G(X) \arrow[r, tail] & G(X).
\end{tikzcd}
\]
\end{rmk}

  As observed in \cite{JM:Deriving}, adjoint functors between abelian categories can be used to construct contractible chain complexes; for instance:
  
  \begin{lem}\label{lem:natural-contraction}
  Let $
\begin{tikzcd}
\cA \arrow[r, yshift=-0.5em, "R"'] \arrow[r, phantom, "\bot"] & \cB \arrow[l, yshift=0.5em, "L"']
\end{tikzcd}
$
define an adjunction between abelian categories inducing a comonad $C=LR$ on $\cA$ with counit $\epsilon \colon LR \Rightarrow \id$. Then for each $A \in \cA$ the chain complex in $\cB$ with differentials defined to be the alternating sums $\sum_{i \geq 0}^k (-1)^i R(LR)^{\times i}\epsilon$
\[
\begin{tikzcd}[column sep=50pt]
\cdots R(LR)^{\times 3}(A) \arrow[r, "R\epsilon-RLR\epsilon+R(LR)^{\times 2}\epsilon"] & R(LR)^{\times 2}A \arrow[r, "R\epsilon - RLR\epsilon"] \arrow[l, bend left, dashed, "s_2"] & RLRA \arrow[r, "R{\epsilon}"] \arrow[l, bend left, dashed, "s_1"] & RA \arrow[l, bend left, dashed, "s_0"]
\end{tikzcd}
\]
admits a contracting homotopy, and these splittings are natural in $\cA$.
 \end{lem} 

This contracting homotopy comes from a contracting simplicial homotopy: the adjoint pair $L \dashv R$ define an augmented cosimplicial object in $\Fun(\cA,\cB)$ admitting a splitting. 

\begin{proof} Define $s_k = \eta R(LR)^{\times k}$ using the unit $\eta \colon \id \Rightarrow RL$ of the adjunction $L \dashv R$.
\end{proof}

The functor categories $\Fun(\cB,\cA)$ and $\Fun_*(\cB^n,\cA)$ are abelian, with kernels and direct sums and so forth defined objectwise in $\cA$.   In particular, these categories have short exact sequences.  A functor from one abelian category to another which preserves exact sequences is called \emph{exact}.

\begin{prop}\label{p:cr is exact}
For  each $n\geq 1$, the functors $\CR_n \colon \Fun(\cB,\cA) \to \Fun_*(\cB^n,\cA)$ and $C_n \colon \Fun(\cB,\cA) \to \Fun(\cB,\cA)$ are exact. 
\end{prop}

\begin{proof}  Let $0\to F\to G\to H\to 0$ be a short exact sequence of functors.
Consider the diagram below
\[
\begin{tikzcd}
0\arrow[r]\arrow[d]& {\CR_1F}\arrow[r]\arrow[d]&\CR_1G\arrow[r]\arrow[d]&\CR_1H\arrow[r]\arrow[d]&0\arrow[d]\\
0\arrow[r]\arrow[d]&F\arrow[r]\arrow[d]&G\arrow[r]\arrow[d]&H\arrow[r]\arrow[d]&0\arrow[d]\\
0\arrow[r]&F(\star)\arrow[r]&G(\star)\arrow[r]&H(\star)\arrow[r]&0.
\end{tikzcd}
\]

Since the columns are exact, and the bottom two rows are exact, the $3\times 3$ lemma \cite[1.3.2]{Weibel} guarantees that the top row is exact.  Then assuming that $\CR_n$ is exact, we can apply the same lemma to the diagram 
\[\begin{tikzcd}[column sep=1em]
\CR_{n+1}F(A_1, A_2, B_2, \dots, B_n)\arrow[r]\arrow[d]&\CR_{n+1}G(A_1, A_2, B_2, \dots, B_n)\arrow[r]\arrow[d]&\CR_{n+1}H(A_1, A_2, B_2, \dots, B_n)\arrow[d]\\
\CR_nF(A_1\vee A_2, B_2, \dots,B_n) \arrow[r]\arrow[d]&\CR_nG(A_1\vee A_2, B_2, \dots,B_n)\arrow[r]\arrow[d]&\CR_nH(A_1\vee A_2, B_2, \dots,B_n)\arrow[d]\\{\txt{$\CR_nF(A_1,B_2,\ldots,B_n)$\\$\oplus \CR_nF(A_2, B_2,\ldots, B_n)$}}\arrow[r]&{\txt{$\CR_nG(A_1,B_2,\ldots,B_n)$\\$\oplus \CR_nG(A_2, B_2,\ldots, B_n)$}}\arrow[r]&{\txt{$\CR_nH(A_1,B_2,\ldots,B_n)$\\$\oplus \CR_nH(A_2, B_2,\ldots, B_n)$}}
\end{tikzcd}
\]
to conclude that $\CR_{n+1}$ is exact.  The proof follows by induction.  The result for $C_n$ follows immediately.  
\end{proof}

\section{A categorical context for abelian functor calculus}\label{sec:kleisli}

\emph{Abelian functor calculus} --- also called \emph{additive} or \emph{discrete functor calculus} \cite{BJM}  --- considers arbitrary functors valued in an abelian category. The linear approximation  defined in \cite{JM:Deriving}  satisfies a universal property ``up to homotopy.'' 
For this to make sense, the target abelian category must be replaced by some sort of homotopical category in which strict universal properties (asserting that certain diagrams commute on the nose) can be replaced by weak ones (where the commutativity is up to some sort of homotopy relation).  For an abelian category $\cA$, let $\Ch\cA$ denote the category of chain complexes on $\cA$ concentrated in non-negative degrees.   In general, the linear approximation of $F \colon \cB \to \cA$ defines a functor $D_1F \colon \cB \to \Ch\cA$; a precise definition of this is given in Definition \ref{def:D1}. 
More generally, \cite{JM:Deriving} define a linear approximation $D_1F \colon \cB \to \Ch\cA$ for any functor $F \colon \cB \to \Ch\cA$ so that when this construction is applied to a functor concentrated in degree zero it recovers the construction for $F \colon \cB \to \cA$.

Now consider a composable pair of functors $G\colon\cC\to \cB$ and $F\colon\cB\to \cA$  between  abelian categories.\footnote{As in Section \ref{sec:cross-effects}, it suffices to assume that the domain of $G$ is a pointed category with finite coproducts, but when considering composite functors it is linguistically convenient to suppose that all of the categories are abelian.} Their linear approximations $D_1G \colon \cC \to \Ch\cB$ and $D_1F \colon \cB \to \Ch\cA$ are not obviously composable. The standard trick, which appears in \cite[Lemma 5.7]{JM:Deriving}, is to prolong the second functor, applying $D_1F$ degreewise to a chain complex in $\cB$ to produce a chain complex in chain complexes in $\cA$, and then convert this double complex into a chain complex in $\cA$ by means of the totalization. That is, the composite $D_1F \circ D_1G$ is defined to be the functor
\[
\begin{tikzcd}
D_1F \circ D_1G\colon  \cC \arrow[r, "D_1G"] & \Ch\cB \arrow[r, "\Ch(D_1F)"] & \Ch\Ch\cA \arrow[r, "\Tot"] & \Ch\cA.
\end{tikzcd}
\]
If $G$ is reduced, then \cite[Lemma 5.7]{JM:Deriving} (a stronger version of which appears as Proposition \ref{lem:5.7} below) proves that this composite is quasi-isomorphic  to $D_1(F \circ G)$, i.e., that $D_1$ is ``functorial up to quasi-isomorphism'' with respect to the composition structure just introduced.

To keep track of which functors should be composable and which are not (and to avoid a proliferation of $\Ch$'s) it is convenient to regard both $F$ and $D_1F$ as functors from $\cB$ to $\cA$. This sort of bookkeeping is effortlessly achieved by the categorical formalism of a Kleisli category for a monad that we now introduce.

\begin{obs}\label{obs:Ch-monad} There is a  pseudomonad\footnote{In category theory, the prefix ``pseudo'' is attached to structures which hold up to specified coherent isomorphism.  A \emph{pseudomonad} is given by the same underlying 1-categorical data as a monad, but with structure diagrams which commute only up to specified coherent natural isomorphisms.  For a more detailed definition, see \cite{marmolejo}.} $\Ch(-)$ acting on the (large) 2-category of  abelian categories, arbitrary functors between them, and natural transformations. Here we are not interested in the 2-dimensional aspects, so we instead describe the quotient monad $\Ch(-)$ acting on the 1-category $\cat{AbCat}$  of abelian categories and isomorphism classes of functors:
\begin{itemize}
\item The monad carries an abelian category $\cA$ to the category $\Ch\cA$ of non-negatively graded chain complexes in $\cA$.

\item The monad carries a functor $F \colon \cB \to \cA$ to its prolongation $\Ch F \colon \Ch\cB \to \Ch\cA$.  Because $F$ might not preserve zero maps, the definition of its prolongation is somewhat delicate. Making use of the Dold-Kan equivalence between non-negatively graded chain complexes and simplicial objects, the functor $\Ch F$ is defined to be the composite
\[
\begin{tikzcd}
\Ch F \colon \Ch\cB \arrow[r, "\simeq"] & \cB^{\DDelta^\op} \arrow[r, "F_*"] & \cA^{\DDelta^\op} \arrow[r, "\simeq"] & \Ch\cA.
\end{tikzcd}
\]
where the action of $F$ on simplicial objects is by post-composition. Note this operation is not strictly functorial: if $G \colon \cC \to \cB$ and $F\colon\cB \to\cA$ are composable functors, then $\Ch F \circ \Ch G$ and $\Ch(F\circ G)$ are naturally isomorphic but not identical. This natural isomorphism satisfies additional ``coherence'' conditions, which imply, in particular, that the composite natural isomorphisms between {$\Ch F \circ (\Ch G \circ \Ch H) = (\Ch F\circ \Ch G)\circ \Ch H$ and $\Ch(F\circ(G\circ H))= \Ch((F\circ G)\circ H)$} coincide. None of this higher coherence structure is visible in the 1-category $\cat{AbCat}$ where the monad $\Ch(-)$ acts strictly functorially on natural isomorphism classes of functors.
\item The components of the unit of the monad are the functors 
\[
\begin{tikzcd}
\cA \arrow[r, "\fun{deg}_0"] & \Ch\cA
\end{tikzcd}
\] 
that embed $\cA$ as the subcategory of chain complexes concentrated in degree zero.
\item The components of the multiplication of the monad are the functors
\[
\begin{tikzcd}
\Ch\Ch\cA \arrow[r, "\Tot"] & \Ch\cA
\end{tikzcd}
\]
that convert a double complex in $\cA$ into a chain complex in $\cA$ by forming the total complex.\footnote{Here a \emph{double complex} is a chain complex of chain complexes, whose squares commute. On account of this convention, a sign must be introduced into the definition of the differentials in the totalization.  See, for example,  Remark \ref{rmk:commutative-bicomplexes}.}
\end{itemize}
We leave to the reader the straightforward verification that these functors define the components of a monad $\Ch(-)$ on $\cat{AbCat}$.
\end{obs}

For any category acted upon by a monad there is an associated Kleisli category (see, e.g., \cite[5.2.9]{riehl:context}), which we describe explicitly in the special case of $\Ch(-)$ acting on $\cat{AbCat}$.  

\begin{defn} There is a (large) category\footnote{It is more categorically  natural to  describe the Kleisli bicategory for the pseudomonad $\Ch(-)$ acting on the 2-category of abelian categories, functors, and natural transformations. The category $\cat{AbCat}_{\Ch}$ is then the quotient 1-category whose morphisms are isomorphism classes of parallel 1-cells. While it is somewhat inelegant to define the morphisms in a 1-category to be isomorphism class of functors, much of our work actually takes place in a further quotient of $\cat{AbCat}_{\Ch}$ where naturally chain homotopically equivalent functors are identified; see Definition \ref{defn:hoabcat}.}   $\cat{AbCat}_{\Ch}$ whose:
\begin{itemize}
\item objects are  abelian categories;
\item morphisms $\cB \kto \cA$ are natural isomorphism classes of functors $\cB \to \Ch\cA$;
\item identity morphisms $\cA \kto \cA$ are the functors $\fun{deg}_0 \colon \cA \to \Ch\cA$; and in which
\item composition of morphisms $\cC \kto \cB$ and $\cB \kto \cA$, corresponding to the pair of functors $G \colon \cC \to \Ch\cB$ and $F \colon \cB \to \Ch\cA$, is defined by
\[ 
\begin{tikzcd}
F \circ G: \cC \arrow[r, "G"] & \Ch\cB \arrow[r, "\Ch F"] & \Ch\Ch\cA \arrow[r, "\Tot"] & \Ch\cA
\end{tikzcd}
\] 
\end{itemize}
Note that $\cat{AbCat}$ defines a subcategory of $\cat{AbCat}_{\Ch}$, where a functor $F \colon \cB \to \cA$ is identified with the morphism $\cB \kto \cA$ in $\cat{AbCat}_{\Ch}$ represented by the functor 
\[ 
\begin{tikzcd}
\cB \arrow[r,"F"] & \cA \arrow[r, "\fun{deg}_0"] & \Ch\cA
\end{tikzcd}
\]
\end{defn}

Note that the results of Section \ref{sec:cross-effects} apply to the homs in the Kleisli category $\cat{AbCat}_{\Ch}$ by considering the functor categories $\Fun(\cB,\Ch\cA)$, henceforth denoted simply by $\Fun(\cB,\cA)$.  

\begin{conv}
Henceforth, we work in the Kleisli category $\cat{AbCat}_{\Ch}$ without further comment. Practically, this means that we may simply write ``$F \colon \cB \kto \cA$'' to denote what is really a functor $F \colon \cB \to \Ch\cA$. There is no ambiguity in the special case where this functor is concentrated in degree zero, i.e., is actually a functor valued in $\cA$. Composition of functors is implemented by the Kleisli construction: by prolonging the second functor and taking the totalization of the resulting double complex.   In particular, all functors whose target is an abelian category are regarded as chain-complex valued, so for instance we can safely refer to chain homotopy equivalences classes of functors or ask whether a given functor is objectwise contractible.

In everything that follows, it is always possible to relax the hypotheses on the domain category and assume only the existence of finite coproducts and a zero object, but we often refer simply to functors ``between abelian categories'' to avoid being overly pedantic with our language. Accordingly, we now write ``$\oplus$'' and ``$0$'' for the coproduct and basepoint in our domain categories.
\end{conv}

The Kleisli category $\cat{AbCat}_{\Ch}$ is the appropriate context to study composition relations involving the Johnson-McCarthy polynomial functors $P_n\colon \Fun(\cB, \cA) \to \Fun(\cB, \cA)$ and linearization functors $D_1\colon \Fun(\cB, \cA) \to \Fun(\cB, \cA)$.  However, there remains one additional technical issue that needs to be addressed. Many of the key properties of \cite{JM:Deriving} are proven up to pointwise quasi-isomorphism in the codomain $\Ch\cA$. In fact, \cite{JM:Deriving} introduce various models of the functors $P_nF$ and $D_1F$ that are only well-defined up to quasi-isomorphism. Unfortunately, quasi-isomorphisms are only preserved under composition with exact functors, while the functors we include as morphisms in $\cat{AbCat}_{\Ch}$ are far from being exact. Thus, if we take the point of view that the functor  $D_1G$ is only defined up to quasi-isomorphism, the Kleisli composite $F \circ D_1G$ will not be well-defined for arbitrary functors $F$.

However, the composition operation in $\cat{AbCat}_{\Ch}$ does respect chain homotopy equivalence of functors, and with some care we will show in Sections \ref{sec:functor-calculus} and \ref{sec:LinearApproximations} that the universal properties up to quasi-isomorphism established in \cite{JM:Deriving} in fact hold up to chain homotopy equivalence. Two functors $H,G\colon\cB \kto \cA$ are  \emph{pointwise chain homotopy equivalent} if the chain complexes $H(X)$ and $G(X)$ are chain homotopy equivalent in $\Ch\cA$ for every object $X\in \cB$, and  
 \emph{naturally chain homotopy equivalent}  if the chain complexes $H(X)$ and $G(X)$ are chain homotopy equivalent in $\Ch\cA$, naturally in $X\in \cB$.  The following lemma proves that composition defined in the Kleisli category $\cat{AbCat}_{\Ch}$ respects (natural) chain homotopy.

\begin{lem}\label{lem:che-comp} Suppose $G,H \colon \cC \kto \cB$ are (naturally) chain homotopic functors. Then for any pair of functors $F \colon \cB \kto \cA$ and $K \colon \cD \kto \cC$, the composite functors \[FGK, FHK \colon \cD \kto \cA\] are again (naturally) chain homotopic. 
\end{lem}
\begin{proof}
It is obvious that if $G$ and $H$ are pointwise or naturally chain homotopic, then so are the restrictions $GK$ and $HK$. The non-trivial part is to show that a  (natural) chain homotopy between $G$ and $H$ is preserved by post-composition with $F$.

We first argue that if $G,H \colon \cC \to \Ch\cB$ are pointwise chain homotopic, then the composite functors
\[
\begin{tikzcd}
\cC \arrow[r, "G"] & \Ch\cB \arrow[r, "\Ch F"] &\Ch\Ch\cA & \mathrm{and} & \cC \arrow[r, "H"] & \Ch\cB \arrow[r, "\Ch F"] &\Ch\Ch\cA 
\end{tikzcd}
\]
are pointwise chain homotopic, where we define chain homotopies in $\Ch\Ch\cA$ as in any category of chain complexes valued in an abelian category (which in this case happens to be $\Ch\cA$). This is a consequence of the Dold-Kan prolongation used to define the monad $\Ch(-)$. The Dold-Kan equivalence takes takes chain homotopies of chain complexes in $\Ch\cB$ to  simplicial homotopies of simplicial objects in $\cB^{\DDelta^\op}$ and vice-versa. Simplicial homotopies in $\cB^{\DDelta^\op}$ are structurally defined, and so preserved by post-composition with $F \colon \cB \to \Ch\cA$. Thus, we conclude that a pointwise chain homotopy between $G$ and $H$ is carried to a pointwise chain homotopy between $\Ch F \circ G$ and $\Ch F \circ H$.

To prove that the Kleisli composites $F \circ G, F \circ H \colon \cC \to \Ch\cA$ are pointwise chain homotopic, we need only argue that the totalization functor $\Tot \colon \Ch\Ch\cA \to \Ch\cA$ preserves chain homotopies. The chain homotopies we are considering in $\Ch\Ch\cA$ are a special case of the more general notion of ``chain homotopy of bicomplexes,'' so this follows from \cite[5.7.3]{Weibel}.

If $G$ and $H$ are chain homotopic by a natural chain homotopy, then the Dold-Kan prolongation will produce a natural chain homotopy between $\Ch F\circ G$ and $\Ch F\circ H$.  Similarly, the totalization functor will preserve naturality, so that $F\circ G$ and $F\circ H$ are naturally chain homotopic. 

\end{proof}

Each hom-set in the category $\cat{AbCat}_{\Ch}$ is equipped with a ``homotopy equivalence of functors'' relation defined as pointwise chain homotopy equivalence in the codomain of the functor.  We choose to use pointwise chain homotopy equivalence to define our homotopy relation rather than natural chain homotopy equivalence in order to make it simpler to establish the chain homotopies needed in what follows.  We will often omit the word ``pointwise'' when we mean ``pointwise chain homotopy''.  Lemma \ref{lem:che-comp} implies that the pointwise chain homotopy equivalence classes in each $\Fun(\cB, \cA)$ are respected by the composition operation in $\cat{AbCat}_{\Ch}$.  We denote these equivalence classes by $[ \cB, \cA]$, and we let $\Ho\cat{AbCat}_{\Ch}$ denote the category with the same objects as $\cat{AbCat}_{\Ch}$ and with hom-sets $[ \cB, \cA]$. This is the advertised categorical context for abelian functor calculus.
  
 \begin{defn}[a category for functor calculus]\label{defn:hoabcat} 
There is a (large) category $\Ho\cat{AbCat}_{\Ch}$ whose:
\begin{itemize}
\item objects are abelian categories;
\item morphisms $\cB \kto \cA$ are pointwise chain homotopy equivalence classes of functors $\cB \to \Ch\cA$;
\item identity morphisms $\cA \kto \cA$ are the functors $\fun{deg}_0 \colon \cA \to \Ch\cA$; and in which
\item composition of morphisms $\cC \kto \cB$ and $\cB \kto \cA$, corresponding to the pair of functors $G \colon \cC \to \Ch\cB$ and $F \colon \cB \to \Ch\cA$, is defined by
\[ 
\begin{tikzcd}
F \circ G\colon \cC \arrow[r, "G"] & \Ch\cB \arrow[r, "\Ch F"] & \Ch\Ch\cA \arrow[r, "\Tot"] & \Ch\cA
\end{tikzcd}
\] 
\end{itemize}
This category is an identity-on-objects quotient of $\cat{AbCat}_{\Ch}$ with hom-sets denoted by $[\cB,\cA]$ defined by taking pointwise chain homotopy equivalence classes of functors in $\Fun(\cB,\cA)$.
\end{defn}

\section{The Taylor tower in abelian functor calculus}\label{sec:functor-calculus}

In this section we review the Taylor tower of a functor constructed in Section 2 of \cite{JM:Deriving}.   The constructions we will provide are essentially the same as those in \cite{JM:Deriving}, though our presentation differs in two ways that are relevant to our treatment of the chain rule. First, we have chosen to emphasize functors which are not necessarily reduced.  The results of Section \ref{sec:cross-effects} make it clear that both the reduced and non-reduced cases can be treated simultaneously. Second, we refine the notion of \emph{degree $n$} functors, replacing the requirement of acyclicity (quasi-isomorphism) in \cite[Definition 2.9]{JM:Deriving} with a stronger contractibility (chain homotopy equivalence) condition. Proposition \ref{p:Pnproperties} proves that the up to quasi-isomorphism universal properties of the polynomial approximations proven in \cite[2.13]{JM:Deriving} become up to chain homotopy equivalence universal properties in the present context.

The Taylor tower will consist of a list of polynomial degree $n$ functors with natural transformations between functors of degree $n$ and degree $n-1$.  We begin this section by explaining what it means for a functor to be {\it degree $n$}.

\begin{defn} \label{def:degree-n} A functor $F \colon \cB \kto \cA$ is \emph{degree $n$} if $\CR_{n+1}F \colon \cB^{ n+1} \kto \cA$ is  \emph{contractible}, i.e., pointwise chain homotopy equivalent to zero.
\end{defn}

We write ``$\simeq$'' when there exists a pointwise chain homotopy equivalence between functors.  Note that if $\CR_kF \simeq 0$, then $\CR_\ell F \simeq 0$ for all $\ell > k$. This  follows from Lemma \ref {lem:B-is-contractible} and the fact that higher cross effects are direct summands of lower cross effects. In particular, degree $k$ functors are also degree $\ell$ for all $\ell > k$.  

Following \cite{JM:Deriving}, we now define the universal polynomial degree $n$ approximations to a functor $F\colon\cB\kto \cA$.
From any comonad acting on an abelian category and object in that category one can extract a chain complex \cite[Definition 2.4]{JM:Deriving}.
 The following definition is a combination of the two definitions of $P_nF$ for reduced and unreduced functors of \cite[Definition 2.8]{JM:Deriving} into a single definition.

\begin{defn}[\cite{JM:Deriving}] \label{def:Pn} The \emph{$n$th polynomial approximation} $P_nF \colon \cB \kto \cA$ of a functor $F \colon \cB \kto \cA$ is the functor that carries $X \in \cB$ to 
the totalization
of the chain complex of chain complexes in $\cA$
\[
\begin{tikzcd}[column sep=45pt]
{} \arrow[r, phantom, "\cdots"] & C_{n+1}^{\times 3} F(X) \arrow[r, "\epsilon-C_{n+1}\epsilon+C_{n+1}^{\times 2}\epsilon"] & C_{n+1}^{\times 2}F(X) \arrow[r, "\epsilon - C_{n+1}\epsilon"] & C_{n+1}F(X) \arrow[r, "\epsilon"] & F(X)
\end{tikzcd}
\]
defined by $(P_nF(X))_k \defeq (C_{n+1})^{\times k}F(X)$ for $k\geq 1$ with differentials defined to be the alternating sums $\sum_i (-1)^i C_{n+1}^{\times i}\epsilon$ of the counit map.
\end{defn}

\begin{rmk}\label{rmk:P0}  When $n=0$, it is easy to give an explicit computation of the chain complex $P_0F$.  By Lemma \ref{lem:CR-idempotency}, $C_1^{\times k}F \cong \CR_1F$ for all $k\geq 1$.  Thus, we can explicitly compute this functor as the chain complex of functors
\[ \begin{tikzcd}
 \cdots \CR_1F \arrow[r,"0"] & \CR_1F \arrow[r, "\id"] & \CR_1F \arrow[r,"0"] & \CR_1F \arrow[r,"\epsilon"] & F,
 \end{tikzcd}
 \]
where the differentials continue to alternate between $0$ and $\id$.  Since $F\cong \CR_1F \oplus F(0)$, we can rewrite the chain complex $P_0F$ as a direct sum of the two chain complexes
\[ \begin{tikzcd} \cdots \CR_1F \arrow[r,"0"] & \CR_1F \arrow[r,"{\id}"] & \CR_1F \arrow[r,"0"] & \CR_1F \arrow[r,"\id"] &\CR_1F\end{tikzcd}\]
and
\[ \begin{tikzcd} \cdots 0 \arrow[r] & 0 \arrow[r] & 0 \arrow[r] & 0 \arrow[r] &F(0). \end{tikzcd}
\]
Here we have used the fact that the map $\epsilon \colon \CR_1F \to F$ is the identity on the component $\CR_1F$ of $F\cong \CR_1F \oplus F(0)$.  The chain complex on the top line is contractible.  Thus, we will use the chain complex in the second line as our model for $P_0F$, and accordingly we will write $P_0F(X)\cong F(0)$.  \end{rmk}

Recalling our convention of working in the Kleisli category described in Section \ref{sec:kleisli}, the $n$th polynomial approximation construction defines a functor $P_n\colon \Fun(\cB, \cA) \to \Fun(\cB, \cA)$. Proposition \ref{p:cr is exact} implies that this functor has good homotopical properties.

\begin{prop}[{\cite[2.13]{JM:Deriving}}]\label{p:Pnproperties} For any $n \geq 0$, \begin{enumerate}
\item $P_n \colon \Fun(\cB,\cA) \to \Fun(\cB,\cA)$ is exact. 
\item $P_n$ preserves  preserves chain homotopies, chain homotopy equivalences, and contractibility.
\end{enumerate}
\end{prop}
\begin{proof}
For (i), let $0\rightarrow F\rightarrow G\rightarrow H\rightarrow 0$ be an exact sequence of functors.  By Proposition \ref{p:cr is exact}, 
\[
0\rightarrow P_nF\rightarrow P_nG\rightarrow P_nH\rightarrow  0
\]
is a sequence of chain complexes that is exact in each degree.  Hence, the sequence is exact. The properties enumerated in (ii) are consequences of (i): exact functors preserve chain homotopies and the zero chain complex.
\end{proof}

The functor $P_nF$ receives a natural transformation $p_n\colon F\to P_nF$ defined by inclusion into the degree zero part of the chain complex $P_nF$. The basic properties of the $n$th polynomial approximation are summarized in the following proposition, an adaptation of \cite[Lemma 2.11]{JM:Deriving}.

\begin{prop} \label{prop:poly-facts} For $F \colon \cB \kto \cA$,
\begin{enumerate}
\item\label{itm:degree} The functor $P_nF$ is degree $n$.
\item\label{itm:q-iso} If $F$ is degree $n$, then the map $p_n \colon F \to P_nF$ is a chain homotopy equivalence.
\item\label{itm:univ} The pair $(P_nF, p_n \colon F \to P_nF)$ is universal up to chain homotopy equivalence with respect to degree $n$ functors receiving natural transformations from $F$.
\end{enumerate}
\end{prop}
\begin{proof}
 
We start with \eqref{itm:degree}, writing $F_k$ for the $k$th degree component of $F \colon \cB \to \Ch\cA$.  By definition, $P_nF$ is the total complex of the bicomplex whose $k$th column is $P_n(F_k)$. To prove that $P_nF$ is degree $n$, we must show that the bicomplex $\CR_{n+1}P_nF$ is contractible. Using \cite[Definition 5.7.3]{Weibel}, a contraction for a bicomplex is given by horizontal contractions $s^h$ and vertical contractions $s^v$ which satisfy  $s^hd^v = d^vs^h$, $s^vd^h = d^hs^v$, and  $1=(s^hd^h+d^hs^h)+(s^vd^v+d^vs^v)$.  In particular, when the bicomplex has naturally contractible columns we see that it is contractible by setting $s^h=0$.
Lemma \ref{lem:natural-contraction}, applied to the adjunction of Corollary \ref{cor:comonad}, supplies a natural contracting homotopy for each  of the columns $\CR_{n+1}P_nF_k$. 
Since totalization preserves chain homotopy equivalence, we conclude that the chain complex $\CR_{n+1}P_nF$ is contractible.

For \eqref{itm:q-iso}, recall that $P_nF$ is the totalization of the first-quadrant bicomplex whose $k$th row is $C_{n+1}^{\times k}F.$ The map $p_n \colon F \to P_nF$ is the natural inclusion of the zeroth row into the totalization. By Corollary \ref{cor:chain-htpy-equiv}, to prove that this is a chain homotopy equivalence, it suffices to show that each chain complex $C_{n+1}^{\times k}F$ is contractible for $k \geq 1$. By exactness of $C_{n+1}$, it suffices to prove that the chain complex $C_{n+1}F$ is contractible. But the hypothesis that $F$ is degree $n$ tells us immediately that $C_{n+1}F(X) := \CR_{n+1}F(X,\ldots, X)$ is contractible. 

Part \eqref{itm:univ} follows formally. Let $\tau \colon F \to G$ be a natural transformation from $F$ to a degree $n$ functor $G$. The natural map $p_n$ provides a commutative diagram
\[ 
\begin{tikzcd}
F \arrow[r, "\tau"] \arrow[d, "p_n"'] & G \arrow[d, "p_n", "\simeq"'] \\ P_n F \arrow[r, "P_n\tau"'] & P_nG
\end{tikzcd}
\]
where the right hand map is a chain homotopy equivalence by part \eqref{itm:q-iso}.  This shows that $\tau$ factors through $p_n \colon F \to P_nF$ up to pointwise chain homotopy equivalence. For uniqueness, consider another factorization 
\[ 
\begin{tikzcd}
F \arrow[rr, bend left, "\tau"] \arrow[r, "p_n"'] \arrow[d, "p_n"'] & P_nF \arrow[r, "\sigma"'] \arrow[d, "p_nP_n"', "\simeq"] & G \arrow[d, "p_n", "\simeq"'] \\ P_n F \arrow[rr, bend right, "P_n\tau"'] \arrow[r, "P_np_n"] & P_nP_n F \arrow[r, "P_n\sigma"] & P_nG
\end{tikzcd}
\]
The maps $p_nP_n, P_np_n \colon P_nF \to P_nP_nF$ are not identical, but do agree up to a natural automorphism of $P_nP_nF$. In particular, because $P_nF$ is degree $n$ by \eqref{itm:degree}, $p_nP_n$ is a chain homotopy equivalence by \eqref{itm:q-iso}, and thus so is $P_np_n$. In this way, we see that $\sigma$ is determined up to chain homotopy equivalence by $P_n\tau$, and so is unique up to chain homotopy equivalence.
\end{proof}

Since $P_{n-1}F$ is also degree $n$, the universal property of Proposition \ref{prop:poly-facts}\eqref{itm:univ} provides a factorization
\[
\begin{tikzcd}
F \arrow[r, "p_n"] \arrow[dr, "p_{n-1}"'] & P_nF \arrow[d, dashed, "q_n"] \\ & P_{n-1}F
\end{tikzcd}
\]
The resulting tower of functors:
\[
\begin{tikzcd}
& & F \arrow[dl, "p_{n+1}" description]  \arrow[dll]  \arrow[d, "p_n"] \arrow[dr, "p_{n-1}" description] \arrow[drrr, "p_0"] \\
\cdots \arrow[r] & P_{n+1}F \arrow[r, "q_{n+1}"'] & P_nF \arrow[r, "q_n"'] & P_{n-1}F \arrow[r] & \cdots \arrow[r, "q_1"'] & P_0F
\end{tikzcd}
\]
is called the \emph{algebraic  Taylor tower} of $F$.

\begin{rmk}\label{rmk:rho}
Explicitly, the map $q_n \colon P_nF \to P_{n-1}F$ is induced by a natural transformation of comonads $\rho_n \colon C_{n+1} \to C_{n}$ given on components by the composite
\[ C_{n+1}F(X) = \CR_{n+1}F(X,\ldots, X) \hookrightarrow \CR_{n}F(X \oplus X,X,\ldots X) \xrightarrow{\CR_nF(\mr{fold})} \CR_nF(X,\ldots,X)=C_{n}F(X)\]
of the canonical direct summand inclusion followed by the image of the fold map $X\oplus X \to X$ \cite[p.~770]{JM:Deriving}.
\end{rmk}

\begin{ex}\label{ex:add-constant2}  Recall the functor $F\colon\cA \to \cA$ defined by $F(X)=A\oplus X$ from Example \ref{ex:add-constant}.
Since $\CR_1F(X)=X=\id(X)$ and $\CR_nF\cong 0$ for $n\geq 2$, it follows immediately from Definition \ref{def:Pn}  that
\[ P_nF(X) = \left( \cdots \to 0 \to 0 \to F(X) \right) \]
for all $n\geq 1$.  When $n=0$, it follows immediately from Remark \ref{rmk:P0} that $P_0F(X) \cong A$.  Indeed, since the chain homotopy $P_0F(X)\simeq F(0)$ is given by the chain map which projects $F(X)$  onto $F(0)$  in degree 0, one sees that the map $q_1\colon P_1F(X) \to P_0F(X)$, or $q_1:F(X) \to A$,  is also this projection map.
\end{ex}

\section{Linear approximations} \label{sec:LinearApproximations}

In this section we consider the linear approximation to a functor, or the homotopy fiber of the map $q_1:P_1F \to P_0F$.  The linearization of a functor will be an important ingredient in defining the directional derivative and the two types of higher order derivatives considered in Sections \ref{sec:faa} and \ref{sec:chain-rule}.  Since the linearization of a functor $F$ with codomain $\cB$ is a functor $D_1F$ with codomain $\Ch\cB$, the Kleisli category conventions of Section \ref{sec:kleisli} will be particularly useful here, especially when we consider linearization together with composition of functors.  In order to maximize the benefits of this structure, we make two main changes in our presentation of the linear approximation from the presentation in \cite{JM:Deriving}.  

First, we insist that the linear approximation $D_1F$ be a morphism in $\cat{AbCat}_{\Ch}$.  In particular, this means we insist that $D_1F$ take values in {\it non-negatively graded} chain complexes.    In \cite[2.14.7]{JM:Deriving}, $D_1F$ (and indeed all of the homotopy fibers of the maps $q_n:P_nF\to P_{n-1}F$) are defined using a particular model for the homotopy fiber given by taking the mapping cone of $q_n$ shifted down one degree.  This shift means that the resulting chain complex will have a non-zero object in degree $-1$.  In Definition \ref{def:D1} we give an alternate model of $D_1F$ which is chain homotopic to the model in \cite{JM:Deriving}, but which is concentrated in non-negative degrees.

Second, since composition by prolongation does not always preserve quasi-iso\-morph\-isms, we have chosen to work up to the stronger notion of chain homotopy equivalence.  In particular, we have provided explicit chain homotopy equivalences for the $D_1$ chain rule in Proposition \ref{lem:5.7}, which strengthens \cite[Lemma 5.7]{JM:Deriving}.  This chain rule is strengthened further in Proposition \ref{prop:D_1chainrule}, where we present the analogue of the $D_1$ chain rule for functors which are not necessarily reduced.  Indeed, all of the properties of the linear approximation $D_1F$ are developed up to chain homotopy equivalence, and for functors which need not be reduced.  These properties occupy the majority of this section.

\begin{defn} \label{def:D1} The \emph{ linearization} of $F\colon \cB \kto \cA$ is the functor $D_1F\colon \cB\kto \cA$ given as the totalization of the explicit chain complex of chain complexes $(D_1F_*, \partial_*)$ where:
\[ (D_1F)_k \defeq \begin{cases}  
 C_2^{\times k}F& k\geq 1\\
\CR_1F & k=0 \\
0 & \text{otherwise}.\end{cases}\] 
The chain differential $\partial_1\colon (D_1F)_1 \to (D_1F)_0$ is given by the map $\rho_1$ of \ref{rmk:rho}, and the chain differential $\partial_k \colon (D_1F)_k \to (D_1F)_{k-1}$ is given by $\sum_{i=0}^{k-1} (-1)^i C_2^{\times i}\epsilon$ 
when $k\geq 1$.  Note that in the case $n=1$, the map $\rho_1$ is precisely the counit $\epsilon$ for the component $\CR_1F$ (see Remark \ref{rmk:counit}).  Since $C_2F \cong C_2 \CR_1F$ by Lemma \ref{lem:CR-idempotency}, the fact that $\partial_1\circ \partial_2=0$ follows immediately by construction.
\end{defn}

\begin{rmk}\label{rmk:D1} $\quad$
\begin{enumerate}
\item 
Using the model of $P_0F$ described in Remark \ref{rmk:P0}, the map $q_0 \colon P_1F \to P_0F$ is surjective, so the homotopy fiber $D_1F$ is simply the kernel of $q_0$.
 There is an inclusion map from  $D_1F$ into the homotopy fiber of $q_1$, and this inclusion is a chain homotopy equivalence.  Thus, this model is chain homotopy equivalent to the one defined in \cite{JM:Deriving}.  
\item Viewed another way, we see that Definition \ref{def:D1} is exactly the same chain complex as $P_1 (\CR_1 F)$ by Lemma \ref{lem:CR-idempotency}. The latter is the degree 1 approximation to the reduction of $F$.  
\end{enumerate}
These two observations imply that, up to chain homotopy equivalence, the two processes of reducing a functor and taking its degree 1 approximation produce the same result when applied in either order.
\end{rmk}

\begin{ex} \label{ex:linearization-of-adding-constant} Returning to the functor $F(X)=A \oplus X$ from Examples \ref{ex:add-constant} and \ref{ex:add-constant2}, since $\CR_2F\cong 0$ and $\CR_1F\cong \id$, the linearization $D_1F$ is given by the chain complex which is the functor $\id$ concentrated in degree 0.  Thus, the linearization of $F$ is $\id$.
\end{ex}

We now turn our attention to the properties of $D_1$.  First, we see that the linearization construction defines an exact functor:

\begin{prop}\label{p:Dnproperties} $\quad$
 \begin{enumerate}
\item $D_1 \colon \Fun(\cB,\cA) \to \Fun(\cB,\cA)$ is exact. 
\item $D_1$ preserves chain homotopies, chain homotopy equivalences, and contractibility.
\end{enumerate}
\end{prop}
\begin{proof}
Since $D_1F$ is isomorphic to $P_1(\CR_1F)$, (i) is a direct consequence of Proposition \ref{p:Pnproperties}, and (ii) follows from exactness.
\end{proof}

Next, we would like to justify the use of the term \emph{linearization} for $D_1$.  We start by defining what it means for a functor to be linear.  Recall that a functor is degree one if $cr_2F$ is contractible, as in Definition \ref{def:degree-n}.

\begin{defn} \label{defn:linear} A functor $F \colon \cB \to \cA$ is \emph{linear} if it is degree one and also reduced, meaning that $F(0)$ is contractible. Equivalently, $F$ is linear if it preserves finite direct sums up to  chain homotopy equivalence.
\end{defn}

We will start with reduced functors.  If $F$ is reduced, $D_1F$ is exactly the analog of the linear approximation of a function for functors.  That is, a function $f:{\mathbb R}\to {\mathbb R}$ whose graph passes through the origin is linear if $f(x+y)=f(x)+f(y)$.

\begin{lem} \label{lem:D1-linearity} $\quad$
\begin{enumerate}
\item For any $F \colon \cB \kto \cA$, the functor $D_1F \colon \cB \kto \cA$ is strictly reduced,  and for any $X, Y \in \cB$, the natural map
\[ D_1F(X) \oplus D_1F(Y) \simto D_1F(X \oplus Y)\] is a chain homotopy equivalence.  In particular, $D_1F$ is linear.
\item  The functor $D_1 \colon [\cB,\cA] \to [\cB,\cA]$ is linear in the sense that $D_10 \cong 0$ and for any pair of functors $F, G \in [\cB,\cA]$, \[D_1F \oplus D_1G \cong D_1(F \oplus G).\]
\end{enumerate}

\end{lem}

\begin{proof}  First, note that $D_1F$ is strictly reduced because the functors $C_2F$ and $\CR_1F$ are both strictly reduced functors.  

To prove the second part of (i), note that the map 
\[ i\colon D_1F(X) \oplus D_1F(Y) \to D_1F(X\oplus Y)\]
is the inclusion of a direct summand with complement $\CR_2 D_1F(X, Y)$.  But $D_1F\cong P_1 (\CR_1F)$ by Remark \ref{rmk:D1}, and $\CR_2 P_1 (\CR_1F) (X,Y)$ is chain contractible by Proposition \ref{prop:poly-facts}(i). It follows that the map $i$ is a chain homotopy equivalence.

Part (ii) is a consequence of Proposition \ref{p:Dnproperties}: any exact functor preserves finite direct sums up to isomorphism.
\end{proof}

Johnson and McCarthy prove that $D_1$ is functorial up to quasi-isomorphism whenever the first of a composable pair of functors is reduced. The proof of \cite[Lemma 5.7]{JM:Deriving} proceeds by showing that both $D_1(F\circ G)$ and $D_1F\circ D_1G$ are quasi-isomorphic to a third chain complex, coming from a tricomplex containing each of the other two.   The next proposition, whose technical proof is deferred to  Appendix \ref{a:AppB}, shows that in fact $D_1(F\circ G)$ is chain homotopy equivalent to $D_1F \circ D_1G$.

\begin{prop}\label{lem:5.7} For any composable pair of functors $F \colon \cB \kto \cA$ and $G \colon \cC \kto \cB$ with $G$ reduced, there is a chain homotopy equivalence
\[ D_1(F \circ G) \simeq D_1F \circ D_1G.\]
\end{prop}

Our next task is to extend Proposition \ref{lem:5.7} to  not necessarily reduced functors $G \colon \cC\kto\cB$ and $F \colon\cB\kto\cA$. The proof follows easily from the next two lemmas. The first of these indicates that the linearization of $F$ is the same as the linearization of its reduced component, $\CR_1F$.  This is analogous to the statement that the linearizations of the functions $f(x)$ and $f(x)-f(0)$ have the same slope.  The second function, $f(x)-f(0)$, is reduced in the sense that its graph goes through the origin.  This analogy explains precisely how to linearize unreduced functors.

\begin{lem}\label{lem:reduced-vs-unreduced}
For any $F \colon \cB \kto\cA$, the natural map $D_1\CR_1F \to D_1F$ is an iso\-mor\-phism.
\end{lem}
\begin{proof}
 By Lemma \ref{lem:CR-idempotency} and Corollary \ref{cor:comonad},  $C_2\CR_1F\cong C_2F$ and $\CR_1 \CR_1 F \cong \CR_1F$.  The result now follows from the definition of $D_1F$.
\end{proof}

\begin{lem} \label{lem:CR_1chainrule} Let $F\colon\cB\kto \cA$ and $G\colon\cC\kto \cB$ be  composable functors.
Then 
\[\CR_1(F\circ G)(X) \cong (\CR_1F \circ \CR_1G)(X) \oplus \CR_2F(G(0), \CR_1G(X)).\]
\end{lem}
\begin{proof}
Note that
\begin{align*} \CR_1(F \circ G)(X) \oplus  F(G(0)) &\cong F(G(X)) \cong F(G(0) \oplus \CR_1G(X))\\ &\cong F(0) \oplus \CR_1F(G(0) \oplus \CR_1G(X)) \\ &\cong F(0) \oplus \CR_1F(G(0)) \oplus \CR_1F(\CR_1G(X)) \oplus \CR_2F(G(0), \CR_1G(X)).\end{align*}
On the other hand,  $F(G(0)) \cong F(0) \oplus \CR_1F(G(0))$.  Taking complements we conclude the claimed isomorphism.
\end{proof}

\begin{prop} \label{prop:D_1chainrule} If $F\colon\cB\kto \cA$ and $G\colon\cC\kto \cB$ are composable functors,  then there is a chain homotopy equivalence
\[ D_1(F\circ G)\simeq D_1F\circ D_1G \oplus D_1 \CR_2F(G(0), \CR_1G).\]
\end{prop}

Note that when $G$ is reduced, $\CR_2F(G(0), G(X)) = \CR_2F(0, G(X))\simeq 0$, so the second term vanishes.  Thus Proposition \ref{prop:D_1chainrule} reduces to Proposition \ref{lem:5.7} in this case.

\begin{proof} Taking the linearization of $F\circ G$, or equivalently, the linearization of $\CR_1 (F\circ G)$, we have
\begin{align*} D_1(F\circ G)&\cong D_1\left( \CR_1(F\circ G)\right) \cong D_1\left((\CR_1F \circ \CR_1G) \oplus \CR_2F(G(0), \CR_1G)\right),\\
\intertext{by Lemmas \ref{lem:reduced-vs-unreduced} and \ref{lem:CR_1chainrule}. Using the linearity of $D_1$ established in Lemma \ref{lem:D1-linearity}, this is isomorphic  to}
&\cong D_1(\CR_1F \circ \CR_1G) \oplus D_1\CR_2F(G(0), \CR_1G).\\
\intertext{Since $\CR_1G$ is reduced, we can apply Proposition \ref{lem:5.7} to the first summand to obtain a chain homotopy equivalence between the last term and}
&\simeq D_1(\CR_1F) \circ D_1(\CR_1G) \oplus D_1\CR_2F(G(0), \CR_1G).\
\intertext{Applying the isomorphisms $D_1(\CR_1F)\cong D_1F$ and $D_1(\CR_1G)\cong D_1G$ of Lemma \ref{lem:reduced-vs-unreduced}, we obtain:}
&\cong D_1F \circ D_1G \oplus D_1\CR_2F(G(0),\CR_1G). \qedhere\end{align*}
\end{proof}

We finish this section with a few observations about linearizations of functors of more than one variable.  When $F\colon \cB^{ n} \kto \cA$ is a functor of $n$ variables, it is possible to linearize $F$ with respect to a subset of the variables.  

\begin{conv}\label{lem:multilinearization} Given a functor $F\colon \cB^{ n} \kto \cA$, let $F_i\colon \cB \kto \cA$ be the functor defined by 
\[ F_i(Y) \defeq F(X_1, \ldots , X_{i-1}, Y , X_{i+1}, \ldots , X_n)\]
where $X_1, \ldots , X_{i-1}, X_{i+1}, \ldots , X_n$ are fixed objects of $\cB$.  Write $D_1^iF(X_1, \ldots , X_n)$ for $D_1 F_i(X_i)$.
\end{conv}

When linearizing $F$ with respect to two or more variables, we have more than one option for how to proceed.  
\begin{itemize}
\item We can linearize $F$ with respect to two or more variables {\emph{simultaneously}}.  For $i<j$, the linearization of $F$ with respect to its $i$ and $j$ variables simultaneously is the linearization of the functor $F_{i\times j} \colon \cB \times\cB\kto \cA$ defined by
\[ F_{i\times j}(Y,Z) \defeq F(X_1, \ldots , X_{i-1}, Y , X_{i+1}, \ldots , X_{j-1}, Z , X_{j+1} , \ldots , X_n)\]
where the $X_k$'s are fixed objects in $\cB$.  We denote this simultaneous linearization by $D_1^{i\times j} F$.  Similarly, we write $D_1^{i_1\times \cdots \times i_k}F$ for the linearization with respect to $k$ variables simultaneously.
\item We can linearize $F$ with respect to two or more variables {\emph{sequentially}}.  For $i<j$, the linearization of $F$ with respect to its $i$ and $j$ variables sequentially is the linearization of the functor $D_1^i F$ with respect to the $j$th variables.  That is, 
\[ D_1^j (D_1^i F) \defeq D_1 (D_1^iF)_j\]
where $(D_1^iF)_j$ is the functor obtained from $D_1^iF$ by holding all but the $j$th variable constant, as in \ref{lem:multilinearization}.  The special case in which $F \colon \cB^{ n} \kto \cA$ has been linearized with respect to each of its $n$ variables sequentially is denoted $D_1^{(n)}F$.  That is, 
\[ D_1^{(n)}F \defeq D_1^n \cdots D_1^1 F.\]
\end{itemize}

Since the variables of a given multilinear functor are not always ordered, it is possible for confusion to arise when we are multilinearizing a functor either simultaneously or sequentially.  To disambiguate, we will use the name of the variable in this case.  For example, the notation
\[ D_1^{X\times Y} \CR_3 F(Z, Y, X)\]
indicates that we should simultaneously linearize the third cross effect of $F$ with respect to the variables $X$ and $Y$, which occur in the second and third slots of $\CR_3F$, respectively.  We could equally well use the notation $D_1^{2\times 3}\CR_3F$ to indicate the same multilinearization, and context will determine which one is more practical.  

The simultaneous linearization of functors of several variables is often a trivialization of the functor, as in the next lemmas.  
\begin{lem}\label{lem:Kristine's}  Suppose that $H\colon\cB^{n} \kto \cA$ is strictly multi-reduced.  Then for any $1\leq i< j \leq n$, 
$D_1^{i\times j} H(X_1, \ldots , X_n)$ is  contractible.  
\end{lem}
\begin{proof} For simplicity of notation, we prove the lemma for a functor $H(X,Y)$ of two variables.   Since $D_1^{1\times 2}H(X,Y)$ is linear simultaneously in the variables $(X,Y)$,  the inclusion \[D_1^{1 \times 2}H((X_1,Y_1)\oplus (X_2, Y_2)) \to D_1^{1 \times 2}H(X_1,Y_1)\oplus D_1^{1 \times 2}H(X_2,Y_2)\] is a chain homotopy equivalence by Lemma \ref{lem:D1-linearity}.  Note that $(X,Y) \cong (X,0) \oplus (0,Y)$, so as a special case $D_1^{1 \times 2}H(X,Y) $ is chain homotopic to $D_1^{1 \times 2}H(X, 0) \oplus D_1^{1 \times 2}H(0, Y)$, which is  zero because $H$ is strictly multi-reduced.  Thus $D_1^{1\times 2}H(X,Y)$ must be  contractible as well.

\end{proof}

Lemma \ref{lem:Kristine's} is often applied in the following form:

\begin{cor}\label{cor:Kristine's} Suppose $F \colon \cB \kto \cA$ factors as
\[ 
\begin{tikzcd}
\cB \arrow[rr, squiggly,  "F"] \arrow[dr, "\Delta"'] & & \cA \\
& \cB^n \arrow[ur, squiggly, "H"'] 
\end{tikzcd}
\] where $\Delta \colon \cB \to \cB^{ n}$ is the diagonal functor and $H$ is strictly multi-reduced. Then $D_1F$  is contractible.
\end{cor}

\begin{proof}
The diagonal functor is strictly reduced so by Proposition \ref{lem:5.7}, $D_1 F $ is chain homotopic to $ D_1 H \circ D_1 \Delta$, but here $D_1 H$ is a simultaneous linearization to which Lemma \ref{lem:Kristine's} applies, implying that $D_1H$ is  contractible. Hence we have a chain contraction $D_1F \simeq 0$.  
\end{proof}

\begin{ex}  The simultaneous linearization of the cross-effects functor $\CR_nF(X_1, \ldots , X_n)$ is contractible by Lemma \ref{lem:Kristine's}.  However, the sequential multilinearization 
\[ D_1^{(n)} \CR_nF(X_1, \ldots , X_n)\] 
is generally  not contractible.  If $D_1^{(n)} \CR_nF(X_1, \ldots , X_n)$ were contractible, then its homotopy orbits $\left( D_1^{(n)} \CR_nF(X_1, \ldots , X_n)\right)_{h\Sigma_n}$  would be acyclic since  homotopy orbits preserve quasi-isomorphisms.  However,  $\left( D_1^{(n)} \CR_nF(X, \ldots , X)\right)_{h\Sigma_n}$ is quasi-isomorphic to  the $n$th layer of the Taylor tower of $F$ by \cite[Proposition 3.9]{JM:Deriving}.  The $n$th layer of the Taylor tower of $F$, $D_nF$, is the homotopy fiber of $q_n:P_nF \to P_{n-1}F$. In many cases  the $n$th  layers of Taylor towers  are known to be non-trivial (see, e.g., \cite[Section 6]{JM:Deriving} and \cite{JM:Sym}).

\end{ex}

When linearizing sequentially, it does not matter in which order we linearize the variables.

\begin{lem}\label{lem:D1-commutativity}  For any $F\colon\cB^{ n} \kto \cA$, there is an isomorphism
\[ D_1^i D_1^j F \cong D_1^j D_1^i F.\]
\end{lem}
\begin{proof}
By using the definition of the cross effects, it is possible to show that $\CR_n (\CR_m F_i)_j \cong \CR_m (\CR_n F_j)_i$, where $F_i$ denotes the functor defined in Convention \ref{lem:multilinearization}.  This is a somewhat lengthy, though straightforward, argument.  Instead we simply note that both functors are reduced and linear in the $i$th and $j$th variables, respectively and indeed that $P_1^i P_1^j (\CR_1^i \CR_1^j F)$ and $P_1^j P_1^i (\CR_1^j \CR_1^i F)$ satisfy the same universal property. 
\end{proof}

Finally, it will be useful to record how linearization behaves with respect to products and projections.  The remaining results in this section describe this behavior.  

\begin{lem}\label{lem:product}
Let $\pi \colon \cB \times \cB \to \cB$ be the functor $\pi:(X,Y) \mapsto X$. Then $D_1^{X \times Y} \pi \cong \pi$.
\end{lem}
\begin{proof}
The functor $\pi \colon \cB \times\cB \to \cB$ is strictly linear so is isomorphic to its  linearization.
\end{proof}

\begin{cor}\label{cor:product} For any $F \colon \cB \kto\cA$, let $\tilde{F}\colon \cB\times \cB \kto \cA$ be the functor defined by 
\[ \tilde{F}(X,Y)\defeq F(X).\]
Then
$D_1^{X \times Y} \tilde{F}(X) \cong D_1^XF(X)$.
\end{cor}
\begin{proof}
By direct computation, 
\[ F(X) \cong \tilde{F}(X,Y) \cong \tilde{F}(0,0) \oplus \CR_1\tilde{F}(X,Y) \cong F(0) \oplus \CR_1\tilde{F}(X,Y),\]
from which we conclude that  $\CR_1\tilde{F}(X,Y) \cong \CR_1F(X)$. Similarly,
\begin{align*} \CR_1\tilde{F}((X_1, Y_1)\oplus (X_2, Y_2)) & \cong \CR_1\tilde{F}(X_1, Y_1) \oplus \CR_1\tilde{F}(X_2, Y_2) \oplus \CR_2 \tilde{F}((X_1, Y_1),(X_2, Y_2))\\
&= \CR_1F(X_1)\oplus \CR_1F(X_2) \oplus \CR_2 \tilde{F}((X_1, Y_1),(X_2, Y_2)).
\end{align*}
Since $\CR_1\tilde{F}((X_1, Y_1)\oplus (X_2, Y_2))\cong \CR_1F(X_1\oplus X_2)$, we conclude from the definition of $\CR_2F$ that $\CR_2F(X_1, X_2)\cong \CR_2 \tilde{F}((X_1, Y_1),(X_2, Y_2))$.  By definition of $D_1$, we now conclude that $D_1\tilde{F}(X,Y) \cong D_1 F(X)$.
\end{proof}

Note that $\tilde{F}$ is a composition of the functor $F$ and the reduced functor $\pi$;  thus it is possible to obtain Corollary \ref{cor:product} as an immediate consequence of Lemma \ref{lem:product} and Proposition \ref{lem:5.7}.  However, this result would only allow us to conclude that $D_1\tilde{F}$ and $D_1F$ are chain homotopy equivalent.  The direct proof  provides a stronger result.

\begin{lem}\label{lem:other-product}
For a functor $\langle F,G \rangle \colon \cC \kto \cA \times \cB$ with components $F\colon \cC \kto \cA$ and $G \colon \cC \kto \cB$, 
\[ D_1\langle F, G \rangle \cong \langle D_1F, D_1G \rangle.\]
\end{lem}
\begin{proof}
In the product abelian category $\cA \times \cB$, direct sums are defined coordinate-wise:
\[ (A_1,B_1) \oplus (A_2,B_2) \cong (A_1 \oplus A_2, B_1 \oplus B_2).\] 
This implies that the cross effects are defined coordinate-wise, and the result follows.
\end{proof}

\section{The first directional derivative}\label{sec:first-directional-derivative}

Applying the results of the previous section, we will now define the directional derivative for a functor  in analogy with the directional derivative in function calculus. We will show that our definition, which is new, recovers the directional derivatives defined in \cite{JM:Deriving}. We review, generalize, and prove new properties of this functor, in analogy with the directional derivative from functional calculus except that equations  are replaced with chain homotopy equivalences.  These results can be summarized in a single theorem:  the directional derivative satisfies the axioms for a categorical derivative in the sense of \cite{BCS:Cartesian}, equipping  $\Ho\cat{AbCat}_{\Ch}$ with the structure of a cartesian differential category.  In particular, we show that the directional derivative satisfies the chain rule.
The properties which we need in order to establish that $\Ho\cat{AbCat}$ is a differential category are precisely the properties we will need in order to obtain a higher order chain rule analogous to \cite[Theorem 3]{HMY}, which we accomplish in Section \ref{sec:chain-rule}.

The directional derivative of a differentiable function $f \colon \RR^n \to \RR^m$ (or more generally, a function  between Banach spaces) at the point $x\in \RR^n$ in the direction $v\in \RR^n$ measures how the value of $f$ at $x$ changes while translating along the infinitesimal vector from $x$ in the direction $v$. One way to make this idea precise is to define $\nabla f(v;x)$ to be the derivative of the composite function, substituting the affine linear function $t \mapsto x + tv$ into the argument of $f$, evaluated at $t=0$:
\[\nabla f(v;x) = \frac{\partial}{\partial t} f(x+tv)\big\vert_{t=0}.\]
This construction motivates the following definition for functors of abelian categories.  

\begin{defn}\label{defn:dir-derivative} For $F \colon \cB\kto\cA$ and $X,V \in \cB$, define a functor $\nabla F\colon \cB\times \cB \kto \cA$ by
\[ \nabla F(V;X) \defeq D_1 F(X \oplus -)(V).\]
\end{defn}

Alternatively, \cite{JM:Deriving} define the directional derivative for a functor $F \colon \cB \kto\cA$ from a category with finite coproducts and basepoint to an abelian category
by using a sequence of analogies with the formula
\[ \nabla f(v;x)= \lim_{t\to 0} \frac{1}{t}\left[ f(x+tv)-f(x)\right]\] from the calculus of functions of a real variable. Let $X\in \cB$ be the `point', corresponding to the point $x\in \RR^n$ at which we will evaluate our derivative.  Let $V\in \cB$ denote the `direction'.  Sums in $\cB$ are coproducts, which we will denote $\oplus$.  The difference operation in the expression $f(x+tv)-f(x)$ becomes a kernel operation in our analogy, 
\[ \ker\left( F(X\oplus V) \to F(X) \right)\]
where the map $F(X\oplus V)\to F(X)$ in $\cA$ is obtained by applying $F$ to the map $X\oplus V\to X$ in $\cB$ which sends the $V$ component to the basepoint.  The linearization $D_1$ (in the variable $V$) produces the degree 1, reduced component of the functor, and hence corresponds to taking the limit as $t$ goes to $0$ in the expression for the directional derivative of $f$.

\begin{defn}\label{defn:dir-derivative-JM} Let $F\colon\cB\kto \cA$ be a functor of abelian categories, and let $X$ and $V$ be objects in $\cB$.  The directional derivative of $F$ at $X$ in the direction $V$ is the bifunctor defined by
\[ \nabla F(V; X) \defeq  D_1^V\left( \ker (F(X\oplus V)\to F(X) )\right).\]
\end{defn}

The following lemma proves that our two definitions agree.
\begin{lem}\label{lem:directional} For a functor $F \colon \cB\kto\cA$ between abelian categories and any pair of objects $X,V \in \cB$,  there is an isomorphism
\[ D_1^V(\ker(F(X \oplus V) \to F(X)))\cong  D_1F(V) \oplus D_1^V\CR_2F(X,V)\]
and 
a chain homotopy equivalence
\[ D_1 F(X \oplus -)(V) \simeq D_1F(V) \oplus D_1^V\CR_2F(X,V).\] 

\end{lem}

\begin{proof} 

For the first isomorphism, recall that by the definition of the cross effects, \[F(X\oplus V)\cong F(0)\oplus \CR_1F(X)\oplus \CR_1F(V)\oplus \CR_2F(X, V) \cong F(X)\oplus \CR_1F(V)\oplus \CR_2F(X, V).\]  The direct summand inclusion $F(X) \to F(X\oplus V)$ is a section to the projection $F(X\oplus V)\to F(X)$, so the kernel $\ker\left( F(X\oplus V)\to F(X) \right)$ is isomorphic to $\CR_1F(V)\oplus \CR_2F(X, V)$. By Lemma \ref{lem:D1-linearity}(ii), $D_1^V$ is linear, so
\begin{align*}  D_1^V( \ker (F(X\oplus V)\to F(X) )) &\cong D_1^V(\CR_1F(V)\oplus \CR_2F(X, V)) \\ &\cong D_1^V\CR_1F(V) \oplus D_1^V\CR_2F(X,V) \\&\cong D_1^VF(V) \oplus D_1^V\CR_2F(X,V),\end{align*} with the last simplification by Lemma \ref{lem:reduced-vs-unreduced}.

For the chain homotopy equivalence, apply Proposition \ref{prop:D_1chainrule} to $F$ and $X \oplus - \colon \cB \to \cB$ to obtain a chain homotopy equivalence
\begin{align*} D_1 F(X \oplus -)(V) &\simeq (D_1F)\circ(D_1(X \oplus -))(V) \oplus D_1^V\CR_2F(X, \CR_1(X \oplus V)) \\
\intertext{We simplify the cross effect term on the right using Lemma 
\ref{lem:reduced-vs-unreduced}:}
 &\simeq (D_1F)\circ(D_1(X \oplus -))(V) \oplus D_1^2\CR_2F(X, D_1^V\CR_1(X \oplus V)) \\
&\cong (D_1F)\circ(D_1(X \oplus -))(V) \oplus D_1^2\CR_2F(X, D_1^V(X \oplus V)). \\
 \intertext{Example \ref{ex:linearization-of-adding-constant} computes that  $D_1 (X \oplus -)$ is the identity functor, so this simplifies to:} 
 &\cong D_1F(V) \oplus D_1^V\CR_2F(X, V). \qedhere
\end{align*}
\end{proof}

In the proofs that follow it will be useful to have both Definitions \ref{defn:dir-derivative} and \ref{defn:dir-derivative-JM} and the third description of Lemma \ref{lem:directional} to characterize  $\nabla F(V;X)$.  One important observation about the directional derivative is that it preserves chain homotopy equivalences between functors.

\begin{obs}If $F$ and $G$ are pointwise chain homotopic, then $\nabla F \simeq \nabla G$. This observation follows immediately from the fact that $\nabla F (V; X) = D_1F(X\oplus - )(V)$, since $D_1F \simeq D_1G$ whenever $F\simeq G$.  Thus, $\nabla$ is  an operation on chain homotopy equivalence classes of functors from $\cB$ to $\cA$.  \end{obs}

The properties of the directional derivative of a functor are reminiscent of the properties of the directional derivative of a function  in classical calculus of  several real variables.  The analogies between the classical notion of differentiation and this more categorical approach has been axiomatized in the notion of a \emph{cartesian differential category}, introduced in \cite{BCS:Cartesian} though presented here in the equivalent form of \cite[Proposition 4.2]{CC}. 
The following theorem is a collection of seven lemmas which will be very useful to us in proving the higher order chain rule and analyzing the higher order directional derivatives in the next two sections.   Collectively, these properties tell us that the directional derivative $\nabla$ equips the category $\Ho\cat{AbCat}_{\Ch}$ with a cartesian differential structure.

\begin{thm} \label{thm:diffcat} The category $\cat{AbCat}_{\Ch}$, together with the directional derivative $\nabla$, satisfies the following seven properties:
\begin{enumerate}
\item For  functors $F,G \colon \cB \kto \cA$, \[\nabla(F\oplus G)(V; X) \cong\nabla F(V;X)\oplus \nabla G(V; X).\]
\item For any $F \colon \cB\kto\cA$, $\nabla F$ is linear in the direction variable: i.e., there is a chain homotopy equivalence  \[\nabla F(V\oplus W; X) \simeq \nabla F(V; X)\oplus \nabla F(W; X)\] and an isomorphism \[  \nabla F(0; X) \cong 0.\]
\item The directional derivative of the identity functor $\id \colon \cA \kto \cA$ is given by projection onto the direction variable: \[\nabla \id(V; X) \cong V.\]
\item For a functor $\langle F,G \rangle \colon \cC \kto \cA \times \cB$ with components $F\colon \cC \kto \cA$ and $G \colon \cC \kto \cB$,  \[\nabla \langle F,G\rangle (V; X)\cong \langle\nabla F(V;X), \nabla G(V; X)\rangle.\]
\item Whenever $G \colon \cC \kto \cB$ and $F \colon \cB \kto \cA$ are composable, there is a chain homotopy equivalence
\[\nabla (F\circ G) (V; X)\simeq \nabla F(\nabla G(V; X); G(X)).\]
\item For any $F \colon \cB \kto \cA$,  there is an isomorphism 
\[\nabla (\nabla F)((Z; 0); (0; X))\cong \nabla F(Z; X).\]
\item For any $F \colon \cB \kto \cA$, there is a chain homotopy equivalence
\[ \nabla (\nabla F)((Z;W); (V; X))\simeq \nabla(\nabla F)((Z; V); (W; X)).\]
\end{enumerate}
\end{thm}

Before proving Theorem \ref{thm:diffcat}, we summarize its statements. A \emph{cartesian differential category} is a category satisfying the following axioms. Firstly, it is \emph{left additive}, meaning that each hom-set is a commutative monoid and pre-comp\-osit\-ion with any map is a monoid homomorphism. A morphism in a left-additive category  is then called \emph{additive} if post-composing with it is also a monoid homomorphism. Secondly, the category has finite products so that the diagonal and product-projection maps are additive, and moreover the product of two additive maps is again additive. Finally, a cartesian differential category is equipped with an operation $\nabla$ that takes a morphism $f \colon b \to a$ and produces a morphism $\nabla f \colon b \times b \to a$ satisfying the properties (i)-(vii) of Theorem \ref{thm:diffcat}. See \cite{BCS:Cartesian}, \cite{CC}. 

The usual cartesian product of two categories that are abelian is again abelian (see, for example, \cite[Exercise VIII.3.2]{MacLane}). 
One can verify that this construction defines a product in the  category $\Ho\cat{AbCat}_{\Ch}$. 
The hom-sets $[\cB,\cA]$ in this category are (large) commutative monoids, with addition of chain homotopy equivalence classes of functors defined pointwise using the direct sum in $\cA$. Pre-composition manifestly respects this structure, as does post-composition with diagonal and product-projection maps. It is similarly easy to verify that the product of two additive maps is additive.
Thus Theorem \ref{thm:diffcat} can be summarized by the following corollary:

\begin{cor}\label{cor:CDC} The homotopy category
$\Ho\cat{AbCat}_{\Ch}$ is a (large) cartesian differential category.
\end{cor}

The proof of Theorem \ref{thm:diffcat} will occupy the remainder of this section.

\begin{proof}[Proof of Theorem \ref{thm:diffcat}(i)]
For functors $F,G \colon \cB \kto \cA$, 
\[ \ker \left((F\oplus G)(X\oplus V) \to  (F\oplus G)(X)\right) \cong \ker \left(F(X\oplus V)\to F(X)\right) \oplus \ker \left(G(X\oplus V) \to G(X)\right).\] The isomorphism 
\[\nabla(F\oplus G)(V; X) \cong \nabla F(V;X)\oplus \nabla G(V; X)\] now follows from linearity of $D_1$ in the form of Lemma \ref{lem:D1-linearity}(ii).
\end{proof}

\begin{proof}[Proof of Theorem \ref{thm:diffcat}(ii)]
The chain homotopy equivalence \[\nabla F(V\oplus W; X) \simeq \nabla F(V; X)\oplus \nabla F(W; X)\] and the isomorphism \[\nabla F(0; X) \simeq 0\] are immediate from either definition of $\nabla F$ and Lemma \ref{lem:D1-linearity}(i).
\end{proof}

\begin{proof}[Proof of Theorem \ref{thm:diffcat}(iii)]
Evidently $\ker\left( X\oplus V \to X \right) \cong V$. The linearization of the identity functor is the identity, so by Definition \ref{defn:dir-derivative-JM} 
\[\nabla \id(V; X) = D_1^V\ker\left( X \oplus V \to X \right) \cong D_1^V\id V \simeq V.\qedhere\]
\end{proof}

\begin{proof}[Proof of Theorem \ref{thm:diffcat}(iv)]
Consider a functor $\langle F,G \rangle \colon \cC \kto \cA \times \cB$. By Definition \ref{defn:dir-derivative-JM}
\begin{align*} \nabla\langle F, G\rangle (V; X) &\defeq D_1^V\ker \left( \langle F, G \rangle (X\oplus V) \to \langle F, G \rangle (X) \right) \\
&\cong D_1^V \ker  \left( \langle F(X\oplus V), G(X\oplus V) \rangle \to \langle F(X), G(X) \rangle \right)\\ \intertext{Since limits in product categories are computed coordinate-wise:}
&\cong D_1^V\langle \ker  \left( F(X\oplus V) \to F(X)\right), \ker \left( G(X\oplus V) \to G(X) \right) \rangle\\
\intertext{By Lemma \ref{lem:other-product}:}
&\cong \langle D_1^V \ker  \left( F(X\oplus V) \to F(X)\right),D_1^V \ker \left( G(X\oplus V) \to G(X) \right) \rangle\\
&=: \langle \nabla F(V; X), \nabla G(V; X)\rangle. \qedhere\\
\end{align*}
\end{proof}

\begin{proof}[Proof of Theorem \ref{thm:diffcat}(v)]
Whenever $G \colon \cC \kto \cB$ and $F \colon \cB \kto \cA$ are composable and $G$ is reduced, \cite[Proposition 5.6]{JM:Deriving} proves a  quasi-isomorphism
\[\nabla (F\circ G) (V; X)\simeq \nabla F(\nabla G(V; X); G(X)).\]
We will upgrade this to a chain homotopy equivalence and drop the hypotheses on $G$.

To make use of Lemma \ref{lem:directional} for $\nabla(F \circ G)$,
we first compute the second cross effect of $F \circ G$, which is defined by 
\begin{align}(F\circ G)(X \oplus V) &\cong(F\circ G)(0) \oplus \CR_1(F\circ G)(X) \oplus \CR_1(F\circ G)(V) \oplus \CR_2(F\circ G)(X,V)\notag \\ &\cong(F\circ G)(X) \oplus \CR_1(F\circ G)(V)  \oplus \CR_2(F\circ G)(X,V).\label{eq:one-side}\end{align} This in turn is isomorphic to
\begin{align*} F(G(X \oplus V)) &\cong F(G(X) \oplus \CR_1G(V)  \oplus \CR_2G(X,V)) \\
&\cong  F(0) \oplus \CR_1F(G(X))\oplus \CR_1 F(\CR_1G(V))  \oplus \CR_1 F(\CR_2G(X,V)) \\ &\quad\oplus \CR_2F(G(X), \CR_1G(V))  \oplus \CR_2F(G(X), \CR_2G(X,V))  \\ &\quad \oplus \CR_2F(\CR_1G(V), \CR_2G(X,V))\oplus \CR_3F(G(X), \CR_1G(V), \CR_2G(X,V)) \\ 
\intertext{The first two terms sum to $F(G(X))$, so we have identified that piece of \eqref{eq:one-side}. Lemma \ref{lem:CR_1chainrule} tells us that $\CR_1(F \circ G)(V) \cong \CR_1F(\CR_1G(V)) \oplus
\CR_2F(G(0), \CR_1G(V))$. To separate this piece, we substitute $G(X)\cong G(0) \oplus \CR_1G(X)$ into the fifth summand and then re-express that second cross effect as a sum of three terms:}
&\cong  F(G(X))\oplus \CR_1 F(\CR_1G(V))  \oplus \CR_1 F(\CR_2G(X,V)) \\ &\quad\oplus \CR_2F(G(0), \CR_1G(V)) \oplus \CR_2F(\CR_1G(X), \CR_1G(V))\\ &\quad\oplus \CR_3(G(0),\CR_1G(X),\CR_1G(V)) \oplus \CR_2F(G(X), \CR_2G(X,V))  \\ &\quad \oplus \CR_2F(\CR_1G(V), \CR_2G(X,V))\oplus \CR_3F(G(X), \CR_1G(V), \CR_2G(X,V))
\end{align*}
Taking complements to cancel these  terms in our two formulas for $F(G(X \oplus V))$, we conclude that
\begin{align*}
\CR_2(F\circ G)(X,V) &\cong  \CR_1 F(\CR_2G(X,V)) \oplus \CR_2F(\CR_1G(X), \CR_1G(V))\\ &\quad\oplus \CR_3(G(0),\CR_1G(X),\CR_1G(V)) \oplus \CR_2F(G(X), \CR_2G(X,V))  \\ &\quad \oplus \CR_2F(\CR_1G(V), \CR_2G(X,V))\oplus \CR_3F(G(X), \CR_1G(V), \CR_2G(X,V))\end{align*}
Now each of these terms is multi-reduced, so upon applying $D_1^V$,  Corollary \ref{cor:Kristine's} implies that any term in which $V$ appears more than once will vanish. Hence there is a chain homotopy equivalence:
\begin{align*}
D_1^V\CR_2(F\circ G)(X,V) &\simeq D_1^V\CR_1 F(\CR_2G(X,V)) \oplus D_1^V\CR_2F(\CR_1G(X), \CR_1G(V))\\ &\quad\oplus D_1^V\CR_3(G(0),\CR_1G(X),\CR_1G(V)) \oplus D_1^V\CR_2F(G(X), \CR_2G(X,V))  \end{align*}
By Proposition \ref{prop:D_1chainrule} and this computation, this proves that there is a chain homotopy equivalence
\begin{align*} 
\nabla (F \circ G)(V;X) &\simeq D_1(F \circ G)(V) \oplus D_1^V \CR_2(F \circ G)(V,X)\\
&\simeq D_1F(D_1G(V)) \oplus D_1^V\CR_2F(G(0),\CR_1G(V)) \\ &\oplus D_1^V\CR_1 F(\CR_2G(X,V)) \oplus D_1^V\CR_2F(\CR_1G(X), \CR_1G(V))\\ &\quad\oplus D_1^V\CR_3(G(0),\CR_1G(X),\CR_1G(V)) \oplus D_1^V\CR_2F(G(X), \CR_2G(X,V))  \\
\intertext{The second, fourth, and fifth terms sum to $D_1^V\CR_2F(G(X),\CR_1G(V))$, so this simplifies to:}
&\cong D_1F(D_1G(V)) \oplus D_1^V\CR_2F(G(X),\CR_1G(V)) \\ &\quad\oplus
D_1^V\CR_1 F(\CR_2G(X,V)) \oplus  D_1^V\CR_2F(G(X), \CR_2G(X,V)) 
\end{align*}

Now let's work out the other side:
\begin{align*}
\nabla F(\nabla G(V;X); G(X)) &\cong \nabla F(D_1G(V) \oplus D_1^V\CR_2G(X,V);G(X)) \\ 
\intertext{By Theorem \ref{thm:diffcat}(ii):}
&\cong \nabla F(D_1G(V); G(X)) \oplus \nabla F(D_1^V\CR_2G(X,V);G(X)) \\ 
&\cong D_1F(D_1G(V)) \oplus D_1^2\CR_2F(G(X), D_1G(V)) \\
	&\quad\oplus D_1F(D_1^V\CR_2G(X,V)) \oplus D_1^2\CR_2F(G(X), D_1^V\CR_2G(X,V)) \\ 
\intertext{Substituting $D_1G\cong D_1\CR_1G$ in the second term and $D_1F \cong D_1\CR_1F$ in the third term, we can apply Proposition \ref{lem:5.7} to each of the last three terms to obtain a chain homotopy equivalence:}
&\simeq D_1F(D_1G(V)) \oplus D_1^V\CR_2F(G(X), \CR_1G(V)) \\&\oplus D_1^V\CR_1F(\CR_2G(X,V)) \oplus D_1^V\CR_2F(G(X), \CR_2G(X,V))\end{align*}
 which agrees with the formula for $\nabla (F \circ G)(V;X)$ above.

 \end{proof}

To prove the remaining statements Theorem \ref{thm:diffcat}(vi) and (vii) we will provide a general formula for $\nabla (\nabla F)$ in terms of the cross-effects of $F$ that will prove important for our  investigation of higher order chain rules in the next section. Our formula will make use of the following characterization of  $\CR_2\nabla F$.

\begin{lem}\label{lem:nabla-cr2} For any $F \colon \cB \kto\cA$ and objects $Z,W,V,X \in \cB$
\begin{align*} \CR_2\nabla F((Z;W)&,(V;X)) \\ &\cong D_1^Z\CR_2F(Z,X) \oplus D_1^Z\CR_3F(Z,W,X) \oplus D_1^V\CR_2F(V,W) \oplus D_1^V\CR_3F(V,W,X).\end{align*}
\end{lem}
\begin{proof}
Note that $\nabla F(0;0) \cong 0$, which implies that $\CR_1\nabla F \cong \nabla F$. Thus $\CR_2F$ is defined as the direct sum complement:
\begin{align*}
\nabla F(Z \oplus V; W \oplus X) &\cong \nabla F(Z;W) \oplus \nabla F(V;X) \oplus \CR_2\nabla F((Z;W),(V;X)). \\ 
\intertext{By Theorem \ref{thm:diffcat}(ii)}
\nabla F(Z \oplus V; W \oplus X) 
&\cong \nabla F(Z; W \oplus X) \oplus \nabla F(V;W \oplus X) \\
&= D_1^ZF(Z) \oplus D_1^Z\CR_2F(Z, W\oplus X) \oplus D_1^VF(V) \oplus D_1^V\CR_2F(V,W \oplus X) \\  
\intertext{By linearity of $D_1$ (Lemma \ref{lem:reduced-vs-unreduced}) and the formula defining the third cross effects (Definition \ref{def:CR}):}  
&\cong D_1^ZF(Z) \oplus D_1^Z\CR_2F(Z,W) \oplus D_1^Z\CR_2F(Z,X) \oplus D_1^Z\CR_3F(Z,W,X) \\ 
	&\quad\oplus D_1^VF(V) \oplus D_1^V\CR_2F(V,X) \oplus D_1^V\CR_2F(V,W) \oplus D_1^V\CR_3F(V,W,X) \\ 
&\cong \nabla F(Z;W) \oplus \nabla F(V;X) \oplus D_1^Z\CR_2F(Z,X) \oplus \\ 
	&\quad\oplus D_1^Z\CR_3F(Z,W,X) \oplus D_1^V\CR_2F(V,W) \oplus D_1^V\CR_3F(V,W,X)
\end{align*}
and the claim follows by cancelling the isomorphic complements.
\end{proof} 

With Lemma \ref{lem:nabla-cr2} in hand, we can now compute $\nabla(\nabla F)((Z; W); (V; X))$.  The next lemma shows that $ \nabla(\nabla F)((Z;W);(V;X))$ is chain homotopy equivalent to $\nabla F(Z;X)$ plus an error term, which will be interpreted in the next section; see Proposition \ref{prop:n=2nablas}.

\begin{lem} \label{lem:nablanabla} For any  $F: \cB\kto \cA$, there is a chain homotopy equivalence
 \begin{align*} \nabla(\nabla F)((Z;W)&;(V;X))   \\ &\simeq D_1^ZF(Z)  \oplus D_1^Z\CR_2F(Z,X)  \oplus D_1^WD_1^V\CR_2F(V,W) \oplus D_1^WD_1^V\CR_3F(V,W,X).\end{align*}
 
\end{lem}

\begin{proof}
By Lemma \ref{lem:directional}
\begin{align*} 
\nabla(\nabla F)((Z;W)&;(V;X)) \\ &\cong D_1^{Z \times W}\nabla F(Z;W) \oplus D_1^{Z \times W}\CR_2\nabla F((Z;W);(V;X)) \\ 
\intertext{By linearity of $D_1$ and Lemma \ref{lem:nabla-cr2}} 
&\cong D_1^{Z \times W}( D_1^ZF(Z) \oplus D_1^Z\CR_2F(Z,W) \oplus D_1^Z\CR_2F(Z,X) \\ &\quad \oplus D_1^Z\CR_3F(Z,W,X) \oplus D_1^V\CR_2F(V,W) \oplus D_1^V\CR_3F(V,W,X)) 
\end{align*}
The functors $D_1^Z\CR_2F(Z,W)$ and $D_1^Z\CR_3F(Z,W,X)$ are strictly multi-reduced in the $Z$ and $W$ variables, so after distributing the linearization $D_1^{Z \times W}$ over the direct sum, Lemma \ref{lem:Kristine's} implies that these two terms are contractible. Thus, the previous term is chain homotopy equivalent to: 
\[  D_1^{Z \times W}( D_1^ZF(Z)  \oplus D_1^Z\CR_2F(Z,X)  \oplus D_1^V\CR_2F(V,W) \oplus D_1^V\CR_3F(V,W,X))\]
and by Corollary \ref{cor:product} this is isomorphic to 
\[ D_1^ZF(Z)  \oplus D_1^Z\CR_2F(Z,X)  \oplus D_1^WD_1^V\CR_2F(V,W) \oplus D_1^WD_1^V\CR_3F(V,W,X)
.\qedhere\]
\end{proof}

The proofs of Theorem \ref{thm:diffcat} (vi) and (vii) follow immediately from Lemma \ref{lem:nablanabla}.  However, using Lemma \ref{lem:nablanabla} will only provide a chain homotopy equivalence between the desired term.  When $V=W=0$, we can strengthen this to an isomorphism by reexamining the proof of Lemma \ref{lem:nablanabla}.

\begin{proof}[Proof of Theorem \ref{thm:diffcat}(vi)]
From the proof of  Lemma \ref{lem:nablanabla}, recall that 
  \begin{align*} \nabla(\nabla F)((Z;W)&;(V;X))   \\ 
  &\cong D_1^{Z \times W}( D_1^ZF(Z) \oplus D_1^Z\CR_2F(Z,W) \oplus D_1^Z\CR_2F(Z,X) \\ &\quad \oplus D_1^Z\CR_3F(Z,W,X) \oplus D_1^V\CR_2F(V,W) \oplus D_1^V\CR_3F(V,W,X)).\end{align*}
  When $V=W=0$, Lemma \ref{lem:D1-linearity}(i) implies that the second term and the last three terms are 0, so
 \[\nabla (\nabla F)((Z; 0); (V; X))\simeq D_1^{Z}F(Z)  \oplus D_1^Z\CR_2F(Z,X) \cong \nabla F(Z; X).\qedhere\]
\end{proof}

The proof of \ref{thm:diffcat}(vii) is now an immediate consequence of Lemma \ref{lem:nablanabla}.

\begin{proof}[Proof of Theorem \ref{thm:diffcat}(vii)]
By inspecting the formula of Lemma \ref{lem:nablanabla}, Lemma \ref{lem:D1-commutativity} and the symmetry of the cross effects implies that there is a chain homotopy equivalence
 \[ \nabla (\nabla F)((Z;W); (V; X))\simeq \nabla(\nabla F)((Z; V); (W; X)). \qedhere\] 
\end{proof}

\section{Higher order directional derivatives and the Fa\`{a} di Bruno formula}\label{sec:faa}

The directional derivative $\nabla F \colon \cB \times \cB \kto \cA$ of a functor $F \colon \cB\kto \cA$ can be iterated to a higher order directional derivative in several different ways. The na\"{i}ve iteration $\nabla^{\times 2} F \defeq \nabla (\nabla F) \colon \cB^{  4} \kto \cA$ has some redundancy stemming from the fact that $\nabla F$ is linear in its first variable; see Theorem \ref{thm:diffcat}(ii) and Lemma \ref{lem:nablanabla}. This problem compounds for higher iterations: in some sense $\nabla^{\times n} F\colon \cB^{ 2n} \kto \cA$ has $n-1$ too many variables.

In this section, we'll consider two ``more efficient'' formulations of the higher directional derivative that bear a very interesting relationship to one another captured by Theorem \ref{thm:restatedgoal}. The first of these is a functor $\nabla^n F\colon \cB^{ n+1} \kto \cA$ introduced by Johnson--McCarthy, which we refer to as the $n$th \emph{iterated partial directional derivative}. The second of these is a new functor $\Delta_n \colon \cB^{ n+1} \kto \cA$ modeled after the higher directional derivatives of Huang, Marcantognini, and Young, which we call the $n$th \emph{higher order directional derivative}.

\begin{defn}[{\cite[Definition 5.8]{JM:Deriving}}]\label{def:tildenabla^n}  Let $F\colon\cB\kto \cA$ be a functor.  The  $n$th \emph{iterated partial directional derivative} of $F$ is defined recursively by 
 \begin{align*} {\nabla}^0F(X) &\defeq F(X),\\  {\nabla}^1 F(V;X)&\defeq\nabla F(V;X), \end{align*}
  and for $n\geq 2$ and objects $V_1, \ldots, V_n$ and $X$ in $\cB$, the $n$th iterated partial directional derivative of $F$ at $X$ in the directions $V_1, \ldots , V_n$ is 
\[ {\nabla}^n F(V_n, \ldots , V_1; X) \defeq \nabla \left({\nabla}^{n-1}F(V_{n-1}, V_{n-2},\ldots , V_1; -)\right)(V_n; X).\]
\end{defn}

Johnson--McCarthy prove that the $n$th iterated partial directional derivative can be understood as a multilinearization of the $n$th and $(n+1)$st cross effects. Recall that for a functor of  $n$-variables, $D_1^{(n)}$ indicates the linearization in each of its first $n$ variables sequentially. By Lemma \ref{lem:D1-commutativity}, the order in which these linearizations are performed does not matter.

\begin{prop}\cite{JM:Deriving} \label{prop:5.9} For any $F \colon \cB\kto \cA$ and variables $V_n,\ldots, V_1,X$
\begin{enumerate}
\item There is a natural isomorphism
\[ \nabla^n F(V_n,\ldots, V_1;X) \cong D^{(n)} \left(\CR_{n+1}F(V_n,\ldots, V_1,X) \oplus \CR_n F(V_n,\ldots, V_1)\right).\]
\item For any permutation $\sigma \in \Sigma_n$, there is a natural isomorphism
\[ \nabla^n F(V_n,\ldots, V_1;X) \cong \nabla^n F(V_{\sigma(n)},\ldots, V_{\sigma(1)};X).\]
\end{enumerate}
\end{prop}
\begin{proof} The first statement is \cite[Proposition 5.9]{JM:Deriving}. The second statement, appearing as \cite[Corollary 5.10]{JM:Deriving}, follows immediately from the symmetry of the cross effects and Lemma \ref{lem:D1-commutativity}.
\end{proof}

On the other hand, the $n$th higher order directional derivative of \cite{HMY} suggests that the following definition should be considered.
\begin{defn} \label{def:delta_n}  For a functor $F\colon\cB\kto \cA$ the $n$th \emph{higher order directional derivatives} of $F$ are defined recursively by 
\begin{align*} \Delta_0F(X) &\defeq F(X),\\ \Delta_1F(V; X) &\defeq \nabla F(V; X), \end{align*}
and for $n\geq 2 $ and objects $V_1, \ldots , V_n$ and $X$ in $\cB$, the $n$th higher order directional derivative of $F$ at $X$ in the directions $V_1, \ldots , V_n$ is defined to be
\[ \Delta_n F(V_n, \ldots , V_1; X) \defeq \nabla(\Delta_{n-1} F)\left( (V_n, \ldots, V_2; V_1); (V_{n-1}, \ldots , V_1; X) \right).\]
\end{defn}

Note that $\nabla^1 F = \nabla (\nabla^0 F)$ and $\Delta_1 F = \nabla (\Delta_0 F)$, so the recursive formulas 
for ${\nabla}^n$ and $\Delta_n$ given in Definitions \ref{def:tildenabla^n} and \ref{def:delta_n} hold for all $n = 1$ as well.

\begin{rmk}\label{r:diagrammatically} 
Both the iterated partial directional derivative $\nabla^n$ and the higher order directional derivative $\Delta_n$ 
 can be thought of as an iteration of the directional derivative $\nabla$ where a new direction is introduced at each instance of differentiation.  The difference between these two is that in the first case, the $n$th iterated partial directional derivative is considered as  a functor of a single variable with its first $(n-1)$ ``directions'' (that is, $V_{n-1}, \ldots, V_1$)  fixed.  The new single variable functor of $X$ is then differentiated in a new direction $V_n$.  In the second case, the $n$-variable functor $\Delta_{n-1}F$ is differentiated in the ``direction'' $(V_n, \ldots ,V_2, V_1)$ and at the ``point'' $(V_{n-1}, \ldots , V_1, X)$.   This ``point'' should be thought of as having recorded the history of how the functor has been differentiated in previous iterations, that is, in the directions $V_{n-1}, \ldots , V_1$, respectively.  The effect of this is that the functor $\Delta_n$  can be understood as a restriction of $\nabla(\Delta_{n-1})$  along a diagonal functor
\[ L_n: \cB^{n+1}\to \cB^n\times \cB^n\]
defined by $L_n(V_n, \ldots , V_1, X) = \left((V_n, \ldots , V_1), (V_{n-1}, \ldots , V_1, X)\right)$.  Thus, as an operation on functors from $\cB$ to $\cA$, we have
\[ \Delta_n \defeq L_n^* \circ \nabla \circ \Delta_{n-1}\]
where $L_n^*$ is the functor which precomposes with $L_n$.  We can write this diagrammatically as
\[ \begin{tikzcd} {\Delta_n \defeq [\cB, \cA]}\arrow[r, "{\Delta_{n-1}}"] & {[\cB^{n}, \cA]}\arrow[r, "\nabla"] & {[\cB^{n}\times \cB^n, \cA]}\arrow[r,"{L_n^*}"] & {[\cB^{n+1}, \cA]}\end{tikzcd}\]
where $[ \cB, \cA]$ denotes the hom-set in $\cat{HoAbCat}_{\Ch}$. 
\end{rmk}

By definition, $\nabla^0 F = F = \Delta_0 F$ and $\nabla^1 F = \nabla F = \Delta_1 F$.  However, the two possible ways of iterating the directional derivative diverge for $n \geq 2$. Their precise relationship for $n=2$ is explained in the following proposition.
\begin{prop} \label{prop:n=2nablas} For a functor $F\colon\cB\kto \cA$ and objects $(Z, W, V, X)$ in $\cB^4$, there is a chain homotopy equivalence
\[ \nabla(\nabla F)\left((Z; W) ; (V; X)\right) \simeq \nabla F(Z; X) \oplus {\nabla}^2 F( W, V; X).\]
Specializing to the object $(V_2, V_1, V_1, X)$ in the image of $L_2$, we have a chain homotopy equivalence
\[ \Delta_2F(V_2, V_1; X) \simeq \nabla F(V_2; X) \oplus {\nabla}^2F(V_1, V_1; X).\]
\end{prop}

\begin{proof}  By Lemma \ref{lem:nablanabla}, we have a chain homotopy equivalence
 \begin{align*} \nabla(\nabla F)((Z;W)&;(V;X))   \\ &\simeq D_1^ZF(Z)  \oplus D_1^Z\CR_2F(Z,X)  \oplus D_1^WD_1^V\CR_2F(V,W) \oplus D_1^WD_1^V\CR_3F(V,W,X).\end{align*}
By the characterization of $\nabla F$ in Lemma \ref{lem:directional}, $D_1^ZF(Z) \oplus D_1^Z\CR_2F(Z,X)\cong \nabla F(Z; X)$.  We need only show that \[D_1^WD_1^V\CR_2F(V,W) \oplus D_1^WD_1^V\CR_3F(V,W,X)\cong {\nabla}^2 F(V, W; X),\] but this is precisely the natural isomorphism of Proposition \ref{prop:5.9}(i).
\end{proof}
Proposition \ref{prop:n=2nablas} can be generalized to explain the relationship between $\Delta_n$ and  ${\nabla}^n$ for all $n\geq 0$.  This relationship  is surprisingly rich.  In particular, its formulation is reminiscent of the Fa\`{a} di Bruno formula for the $n$th derivative of $f\circ g$ for composable {\it functions} $f$ and $g$.  The relationship is stated precisely in the next theorem.

For a partition $\pi=\{S_1, \ldots, S_k\}$ of the set $\{1,\ldots, n\}$, we use the following notation:
\[
\nabla^{\pi}F(V_n, \ldots, V_1;X)\defeq\nabla^{|\pi|}F(V_{|S_1|}, \ldots, V_{|S_k|};X)
\]
where $|S|$ denotes the cardinality of the set $S$.  Note that by Proposition \ref{prop:5.9}(ii), the order in which we write the $V_{|S_i|}$'s doesn't matter.

\begin{ex}\label{ex:base-case}
For the partition $\pi = \{ \{1\}, \{2\}\}$, $\nabla^\pi F(V_1,V_2;X) \defeq \nabla^2F(V_1,V_1;X)$, while for $\pi = \{\{1,2\}\}$, $\nabla^\pi F(V_1,V_2;X) \defeq \nabla F(V_2;X)$. Not coincidentally, these are the two terms appearing in the formula for $\Delta_2 F(V_2,V_1;X)$ proven in Proposition \ref{prop:n=2nablas}.
\end{ex}

We are now prepared to state the general result:

\begin{thm}\label{thm:restatedgoal} For any $F \colon \cB\kto \cA$ and objects $V_n,\ldots, V_1,X$, there is a chain homotopy equivalence
\[\Delta_nF(V_n,\ldots, V_1;X)\simeq  \bigoplus _{ \pi=\{S_1, \ldots, S_k\}}\nabla^{\pi}F(V_n, \ldots, V_1;X)\]
where the sum is taken over all partitions $\pi$ of the set $\{1,2,\ldots, n\}$.  
\end{thm}

Theorem \ref{thm:restatedgoal} is a consequence of the following lemma.

\begin{lem}\label{l:nabla1nablatwiddle} For a partition $\pi=\{S_1, \ldots, S_k\}$ of $\{1,\ldots, n\}$, there is a chain homotopy equivalence
\begin{align}\label{e:lemma} \nabla( \nabla^{\pi}F&)((V_{n+1}, \ldots; V_1);(V_n, \ldots, V_1; X)) \simeq \\ \nabla^{|\pi|+1}F(V_{|S_1|}, &\ldots, V_{|S_k|}, V_1; X)\oplus\bigoplus_{i=1}^k\nabla^{|\pi|}F(V_{|S_1|}, \ldots, V_{|S_{i-1}|}, V_{|S_i|+1}, V_{|S_{i+1}|},\ldots, V_{|S_k|}; X).\notag
\end{align}
\end{lem}

This lemma permits an expeditious proof of Theorem \ref{thm:restatedgoal}.

\begin{proof}[Proof of Theorem \ref{thm:restatedgoal}]
The right-hand side of \eqref{e:lemma} can be interpreted as  the sum of terms of the form $\nabla^{\sigma}F(V_{n+1}, \ldots, V_1; X)$ taken over all partitions $\sigma$ of $\{1, 2, \ldots, n+1\}$ that are obtained from the partition $\pi=\{S_1, \ldots, S_k\}$ of $\{1,\ldots, n\}$ by either adding $\{n+1\}$ as a separate set (contributing the first summand) or adding $n+1$ to one of the $S_i$'s (contributing the remaining $k$ summands).  Since all partitions of $\{1,\ldots, n+1\}$ can be obtained in this way from partitions of $\{1,\ldots, n\}$, Theorem \ref{thm:restatedgoal} is proven by a simple induction with Proposition \ref{prop:n=2nablas} as the base case, as explained by Example \ref{ex:base-case}, and Lemma \ref{l:nabla1nablatwiddle} as the inductive step. 
\end{proof}

 Next we prove the lemma.

\begin{proof}[Proof of Lemma \ref{l:nabla1nablatwiddle}]

Let $\pi = \{S_1,\ldots, S_k\}$ be a partition of $\{1,\ldots, n\}$. By Proposition \ref{prop:5.9}(i)
\begin{align*}
 \nabla ^{\pi}F(V_n,\ldots, V_1; X)&\defeq \nabla ^kF(V_{|S_1|}, \ldots, V_{|S_k|}; X)\\
&\cong D_1^{(k)}\CR_kF(V_{|S_1|}, \ldots, V_{|S_k|})\oplus D_1^{(k)}\CR_{k+1}F(V_{|S_1|}, \ldots, V_{|S_k|},X).
\end{align*}
By Definition \ref{defn:dir-derivative-JM}, $\nabla ( \nabla^{\pi}F)((\overline V_{n+1}, \ldots; \overline V_1);(V_n, \ldots, V_1; X))$ is equal to \begin{equation}\label{e:nablav}D_1^{\overline V_{n+1}\times \cdots \times \overline V_1}{\rm ker}( \nabla^{\pi}F(\overline V_{n+1}\oplus V_n, \ldots; \overline V_1\oplus X)\rightarrow  \nabla^{\pi}F(V_n, \ldots, V_1; X)),\end{equation} 
where we have added new variables $\overline V_i$ to keep track of where the linearizations are being applied.  
 Recalling Convention \ref{lem:multilinearization}, the notation $D_1^{\overline V_{n+1}\times \cdots \times \overline V_1}$ indicates the linearization $D_1H$ of a functor $H\colon\cB^{n+1}\to \cA$, where $H$ is the functor \[{\rm ker}( \nabla^{\pi}F(-\oplus V_n, \ldots; -\oplus X)\rightarrow  \nabla^{\pi}F(V_n, \ldots, V_1; X))\] 
and the $V_i$'s and $X$ are fixed values of $\cB$.  Because $H$ is a functor of several variables, this linearization involves the simultaneous linearization of the variables $\overline V_1, \ldots , \overline V_{n+1}$.

Our strategy will be to use Proposition 5.9 to rewrite \eqref{e:nablav} in terms of multilinearizations of cross effects, use Lemma \ref{lem:Kristine's} and Corollary \ref{cor:Kristine's} to eliminate many of the summands after taking the kernel and applying $D_1^{( V_{n+1}, \ldots,  V_1)}$, and use Proposition \ref{prop:5.9}(i) to verify that what is left is equivalent to (\ref{e:lemma}).

By Proposition \ref{prop:5.9}(i), \eqref{e:nablav} is isomorphic to $D_1^{\overline V_{n+1}\times \cdots \times \overline V_1}$ applied to the kernel of the map 
\[
\begin{tikzcd}
D_1^{(k)}\CR_kF(\overline V_{|S_1|+1}\oplus V_{|S_1|}, \ldots, \overline V_{|S_k|+1}\oplus V_{|S_k|})\oplus D_1^{(k)}\CR_{k+1}F(\overline V_{|S_1|+1}\oplus V_{|S_1|}, \ldots,\overline V_1\oplus X) \arrow[d] \\
D_1^{(k)}\CR_kF(V_{|S_1|}, \ldots, V_{|S_k|})\oplus D_1^{(k)}\CR_{k+1}F(V_{|S_1|}, \ldots, V_{|S_k|}, X)
\end{tikzcd}
\]
Since the domain of this map is linear in its first $k$ variables, it is chain homotopy equivalent to an expansion of this expression as illustrated for instance by:
\begin{align*}
&D_1^{(k)}\CR_kF(\overline V_{|S_1|+1}\oplus V_{|S_1|}, \ldots, \overline V_{|S_k|+1}\oplus V_{|S_k|})\\
&\simeq D_1^{(k)}\CR_kF(\overline V_{|S_1|+1}, \ldots, \overline V_{|S_k|+1}\oplus V_{|S_k|})\oplus D_1^{(k)}\CR_kF(V_{|S_1|}, \ldots, \overline V_{|S_k|+1}\oplus V_{|S_k|})\\
&\simeq D_1^{(k)}\CR_kF(\overline V_{|S_1|+1}, \overline V_{|S_2|+1}, \ldots, \overline V_{|S_k|+1}\oplus V_{|S_k|})\oplus D_1^{(k)}\CR_kF(\overline V_{|S_1|+1},V_{|S_2|}, \ldots, \overline V_{|S_k|+1}\oplus V_{|S_k|}) \\
&\quad\oplus D_1^{(k)}\CR_kF(V_{|S_1|},\overline V_{|S_2|+1}, \ldots, \overline V_{|S_k|+1}\oplus V_{|S_k|})\oplus D_1^{(k)}\CR_kF(V_{|S_1|}, V_{|S_2|}, \ldots, \overline V_{|S_k|+1}\oplus V_{|S_k|})\\
&\simeq \cdots 
\end{align*}
but rather than carrying out these expansions completely,  we note that because we are applying $D_1^{\overline V_{n+1}\times \cdots \times \overline V_1}$, Corollary \ref{cor:Kristine's} guarantees that any summand that has more than one $\overline V_i$ will be contractible.  
Hence we see that $D_1^{\overline V_{n+1}\times \cdots \times \overline V_1}( \nabla^{\pi}F)(\overline V_{n+1}\oplus V_n, \ldots; \overline V_1\oplus X)$ is chain homotopy equivalent to $D_1^{\overline V_{n+1}\times \cdots \times \overline V_1}$ of 
\begin{align*}
&\bigoplus _{i=1}^kD_1^{(k)}\CR_kF(V_{|S_1|}, \ldots, V_{|S_{i-1}|}, \overline V_{|S_i|+1}, V_{|S_{i+1}|},\ldots, V_{|S_k|})\\
&\oplus \bigoplus _{i=1}^k D_1^{(k)}\CR_{k+1}F(V_{|S_1|}, \ldots, V_{|S_{i-1}|}, \overline V_{|S_i|+1}, V_{|S_{i+1}|},\ldots, V_{|S_k|}, \overline V_1\oplus X)\\
&\oplus D_1^{(k)}\CR_{k+1}F(V_{|S_1|}, \ldots, V_{|S_k|}, \overline V_1\oplus X)\oplus D_1^k\CR_kF(V_{|S_1|}, \ldots, V_{|S_k|}).
\end{align*}
  That is, the only terms with a single variable of the form $\overline V_t$ appear in the kernel. 
A term of the form \[D_1^{(k)}\CR_{k+1}F(V_{|S_1|}, \ldots, V_{|S_{i-1}|}, \overline  V_{|S_i|+1}, V_{|S_{i+1}|},\ldots, V_{|S_k|}, \overline V_1\oplus X)\] can be expanded as 
\begin{align*}
&D_1^{(k)}\CR_{k+1}F(V_{|S_1|}, \ldots, V_{|S_{i-1}|}, \overline V_{|S_i|+1}, V_{|S_{i+1}|},\ldots, V_{|S_k|}, \overline V_1)\\
&\oplus D_1^{(k)}\CR_{k+1}F(V_{|S_1|}, \ldots, V_{|S_{i-1}|}, \overline V_{|S_i|+1}, V_{|S_{i+1}|},\ldots, V_{|S_k|}, X)\\ 
&\oplus D_1^{(k)}\CR_{k+2}F(V_{|S_1|}, \ldots, V_{|S_{i-1}|}, \overline V_{|S_i|+1}, V_{|S_{i+1}|},\ldots, V_{|S_k|}, \overline V_1,X).
\end{align*}
But, again by Corollary \ref{cor:Kristine's}, the first and third summands will be contractible after applying $D_1^{\overline V_{n+1}\times \cdots \times \overline V_1}$.  Hence, after computing the kernel and applying $D_1^{\overline V_{n+1}\times \cdots \times \overline V_1}$, we see that \[\nabla ( \nabla^{\pi}F)((V_{n+1}, \ldots; V_1);(V_n, \ldots, V_1; X))\] is chain homotopy equivalent to
\begin{align*}
&\bigoplus _{i=1}^k D_1^{(k)}\CR_kF(V_{|S_1|}, \ldots, V_{|S_{i-1}|}, \overline V_{|S_i|+1}, V_{|S_{i+1}|}, \ldots, V_{|S_k|})\\
&\oplus \bigoplus _{i=1}^k D_1^{(k)}\CR_{k+1}F(V_{|S_1|}, \ldots, V_{|S_{i-1}|}, \overline V_{|S_i|+1}, V_{|S_{i+1}|}, \ldots, V_{|S_k|}, X)\\
&\oplus D_1^{(k+1)}\CR_{k+1}F(V_{|S_1|}, \ldots, V_{|S_k|},\overline V_1)\oplus D_1^{(k)}\CR_{k+2}F(V_{|S_1|}, \ldots, V_{|S_k|}, \overline V_1, X).
\end{align*}
By Proposition \ref{prop:5.9}(i), this is isomorphic to:
\[
\nabla^{|\pi|+1}F(V_{|S_1|}, \ldots, V_{|S_k|}, V_1; X)\oplus\bigoplus_{i=1}^k\nabla^{|\pi|}F(V_{|S_1|}, \ldots, V_{|S_{i-1}|}, V_{|S_i|+1}, V_{|S_{i+1}|},\ldots, V_{|S_k|};X) \qedhere
\]
\end{proof}


\section{The higher order  chain rule}\label{sec:chain-rule}

This brings us to the original motivation for this paper, the higher order directional derivative chain rule.  Having established a relationship between the $n$th higher directional derivative and the $n$th iterated partial directional derivative,  our goal is to establish a higher order chain rule in the style of \cite{HMY} for the former of these two.  

\begin{thm}\label{t:higherchainrule}  For a composable pair $F\colon\cB\kto \cA$ and $G\colon\cC\kto \cB$ of functors of abelian categories, the $n$th directional derivative satisfies
\[ \Delta_n (F\circ G)(V_n, \ldots , V_1; X) \simeq \Delta_n F(\Delta_nG(V_n, \ldots , V_1; X), \ldots , \Delta_1 G(V_1; X); G(X)).\]
\end{thm}

For the case $n=2$, the third and fifth authors carried out a very careful computation verifying this theorem using only properties of linearization and cross effects.  A short paper containing this computation is in preparation.

We will explain how this chain rule can be derived as a natural consequence of structures associated to cartesian differential categories. Cockett and Cruttwell introduce a more general notion of a \emph{tangent category} \cite{CC} and show that any cartesian differential category gives an example. The key component of a tangent category structure is an endofunctor $T$. In our case, $T$ is defined as follows.  

\begin{defn}\label{def:tangent-functor} 
There is a functor $T \colon \Ho\cat{AbCat}_{\Ch} \to \Ho\cat{AbCat}_{\Ch}$ defined on objects by $\cA \mapsto \cA \times \cA$ and on morphisms $F \colon \cB \kto \cA$ by
\[TF \defeq \langle \nabla F, F\pi_R \rangle \colon \cB \times \cB \kto \cA \times \cA,\]
where $\pi_R\colon\cB \times \cB \to \cB$ denotes the projection onto the second, or right-hand, factor. That is
\[
TF(V,X) \defeq \langle\nabla F(V;X), F(X)\rangle.\]
\end{defn}

Importantly, Theorem \ref{thm:diffcat}(v) implies that $T$ is \emph{functorial} up to chain homotopy equivalence, as asserted in Definition \ref{def:tangent-functor}.

\begin{lem}\label{lem:tangent-functor}  For a composable pair $F\colon\cB\kto \cA$ and $G\colon\cC\kto \cB$ of functors of abelian categories,  $T(F\circ G)$ is chain homotopy equivalent to $TF \circ TG$.  
\end{lem}
\begin{proof}  We use Theorem \ref{thm:diffcat}(v) to show that $T$ is a functor up to chain homotopy equivalence. By definition, $TF\circ TG$ is given by 
\[ 
\begin{tikzcd}[row sep=tiny]
\cC \times \cC \arrow[r, squiggly, "TG"] & \cB \times \cB \arrow[r, squiggly, "TF"] & \cA \times \cA \\
(V,X) \arrow[r, mapsto] & \langle\nabla G(V;X), G(X)\rangle \arrow[r, mapsto] & \langle \nabla F(\nabla G(V;X); G(X)), F(G(X))\rangle
\end{tikzcd}
\]
while $T(F\circ G)$ is given by
\[ 
\begin{tikzcd}[row sep=tiny]
\cC \times \cC \arrow[r, squiggly, "T(F\circ G)"] &  \cA \times \cA \\
(V,X) \arrow[r, mapsto] &  \langle \nabla(F\circ G)(V;X) , F\circ G(X) \rangle
\end{tikzcd}
\]
By Theorem \ref{thm:diffcat}(v) these outputs are chain homotopy equivalent:
\[  \langle \nabla(F\circ G)(V;X) , F\circ G(X) \rangle \simeq \langle \nabla F(\nabla G(V;X); G(X)), F(G(X)) \rangle.  \qedhere
\]
\end{proof}

Lemma \ref{lem:tangent-functor} shows that $T$ is an endofunctor of $\Ho\cat{AbCat}_{\Ch}$. As an immediate corollary, its iterates $T^n$ are again functorial: that is, $T^n(F\circ G) \simeq T^nF \circ T^n G$. Since chain homotopy equivalences are respected by pre- and post-composition in $\Ho\cat{AbCat}_{\Ch}$, the composite functors
\begin{equation}\label{eq:real-8.1}
\begin{tikzcd}
\cC^{n+1} \arrow[r, "{d_n^\ast}"] & \cC^{2^n} \arrow[r, squiggly, bend left, "T^n(F \circ G)"] \arrow[r, squiggly, bend right, "T^nF \circ T^nG"'] \arrow[r, phantom, "\simeq"] & \cA^{2^n} \arrow[r, "\pi_L"] & \cA
\end{tikzcd}
\end{equation} 
are again chain homotopy equivalent, where $\pi_L$ is the projection onto the leftmost component and $d_n^*$ is a diagonal functor to be described below. We claim that the top composite of \eqref{eq:real-8.1} is the functor $\Delta_n(F\circ G)$, while the bottom composite is the functor appearing on the right-hand side of the chain homotopy equivalence of Theorem \ref{t:higherchainrule}. Once we make these identifications, which will be achieved by the following series of combinatorial lemmas, we will be able to conclude that Theorem \ref{t:higherchainrule} is a consequence of the functoriality of $T$.

The diagonal functor $d_n^\ast$ is a special case of a ``reindexing functor'' in the sense of the following definition.

\begin{defn}\label{defn:reindex} Write ${\bf n} = \{ 0, 1, \cdots , n-1\}$ for the set with $n$ elements, and let $c\colon {\bf n} \to {\bf k}$ be any function. The product category $\cB^n$ can be thought of as the category of functors from the discrete category ${\bf n}$ to $\cB$, from which perspective pre-composition with $c \colon {\bf n} \to {\bf k}$ defines a \emph{reindexing functor} $c^*\colon\cB^{k} \to \cB^{n}$. Explicitly, $c^*$ is the functor 
\[ \langle \pi_{c(0)}, \ldots , \pi_{c(n-1)} \rangle \colon \cB^k \to \cB^n\] whose components
 $\pi_i\colon \cB^{k}\to \cB$ project onto the $i$th factor.
\end{defn}

Diagonal functors are reindexing functors which come from surjections of sets.  Importantly, reindexing functors commute with the directional derivative $\nabla$ in the sense of the following result.

\begin{lem}\label{lem:nabla-and-diagonals} Suppose $K \colon \cB^k \kto \cA$ factors as
\[ 
\begin{tikzcd}
\cB^k \arrow[rr, squiggly,  "K"] \arrow[dr, "c^*"'] &\arrow[d, phantom, "\simeq"]  & \cA \\
& \cB^n \arrow[ur, squiggly, "H"'] 
\end{tikzcd}
\] up to chain homotopy equivalence, where $c^* \colon \cB^k \to \cB^{ n}$ is the reindexing functor associated to some function $c \colon n \to k$.
Then $\nabla K$ factors as
\[ 
\begin{tikzcd}
\cB^k \times \cB^k \arrow[rr, squiggly,  "\nabla K"] \arrow[dr, "c^* \times c^*"'] & \arrow[d, phantom, "\simeq"]  & \cA \\
& \cB^n \times \cB^n \arrow[ur, squiggly, "\nabla H"'] 
\end{tikzcd}
\] 
up to chain homotopy equivalence
\end{lem}
\begin{proof} Let $V, X \in \cB^k$. By Theorem \ref{thm:diffcat}(v), $\nabla K(V;X) \simeq \nabla H (\nabla c^*(V;X); c^*(X))$. By Theorem \ref{thm:diffcat}(iii) and (iv), $\nabla c^*(V;X) \simeq c^* V$, so $\nabla K(V;X) \simeq \nabla H(c^*(V); c^*(X))$.
\end{proof}

This observation allows us to make the relationship between $T^nF$ and $\nabla^{\times k}F$ explicit, which we describe in the next lemma by examining the coordinates of $T^nF\colon\cB^{2^n} \to \cA^{2^n}$ one at a time.  For convenience, we identify the cardinal $2^n$ with the set ${\mathcal P}({\bf n})$, the power set of ${\bf n}$, which has cardinality $2^n$.  
Under this identification, $T^nF \colon \cB^{2^n} \kto \cA^{2^n}$ is a collection of functors $(T^nF)_S\colon\cB^{2^n}\kto \cA$, indexed by elements $S$ of ${\mathcal P}({\bf n})$, i.e., by subsets $S \subset {\bf n}$. 

\begin{lem}\label{lem:Tn-component} For any $S \subset {\bf n}$, there is a chain homotopy equivalence
\[ 
\begin{tikzcd}
\cB^{2^n} \arrow[rr, squiggly,  "(T^nF)_S"] \arrow[dr, "\iota_S^*"'] &  \arrow[d, phantom, "\simeq"] & \cA \\
& \cB^{2^{|S|}} \arrow[ur, squiggly, "\nabla^{\times |S|}F"'] 
\end{tikzcd}
\] 
where $\iota_S \colon \mathcal{P}(S)\hookrightarrow\mathcal{P}({\bf n})$ is the inclusion.
\end{lem}
\begin{proof}
We prove this by induction on $n$. The case $n=1$ is given by the definition of $T$, but we include it in detail here in order to begin to examine the definition of components of $T$ in terms of subsets $S$ of $\mathcal{P}({\bf n})$.  When $n=1$, the power set $\mathcal{P}({\bf 1})$ contains the two elements $\emptyset$ and $\{ 0\}$.  This corresponds to the two diagrams
\[ 
\begin{tikzcd}
\cB\times \cB \arrow[rr, squiggly,  "(TF)_{\{0\}}"] \arrow[dr, "\iota_{\{0\}}^*"'] &  \arrow[d, phantom, "\simeq"] & \cA & & \cB\times \cB\arrow[rr, squiggly,  "(T^nF)_\emptyset"] \arrow[dr, "\iota_\emptyset^*"'] &  \arrow[d, phantom, "\simeq"] & \cA \\
& \cB\times \cB \arrow[ur, squiggly, "\nabla F"'] & & & & \cB \arrow[ur, squiggly, "F"'] 
\end{tikzcd}
\] 
This in turn corresponds directly to the definition of $TF(V, X) = \langle \nabla F(V;X), F(X) \rangle$ where the second coordinate is indexed by $\emptyset$ and the first coordinate is indexed by $\{ 0 \}$.  In particular, $TF_{S}=\nabla F$ when $S$ contains the ``top'' element $0 \in {\bf 1}$, and $TF_S=F$ when $S$ does not contain $0$.

 For the inductive step, first consider those $S \subset {\bf n}$ that do not contain the top element. For these components, $(T^nF)_S = T(T^{n-1}F_S) = (T^{n-1}F)_S \circ \pi_R$, and here $\pi_R \colon \cC^{2^n} \to \cC^{2^{n-1}}$ is the map defined by restricting along the map $\mathcal{P}({\bf n-1})\subset \mathcal{P}({\bf n})$ induced by the inclusion of the first $n-1$ elements. By the inductive hypothesis, $(T^{n-1}F)_S \simeq \nabla^{\times |S|} F \circ \iota_S^*$ and the claim follows.

Now consider the components indexed by a subset $S \subset {\bf n}$ that does contain the top element $n-1 \in {\bf n}$. Here $(T^nF)_S =  \nabla ((T^{n-1}F)_{S\backslash n-1})$. By the inductive hypothesis, $(T^{n-1}F)_{S\backslash n-1} \simeq \nabla^{\times |S|-1}F \circ \iota_{S\backslash n-1}^*$, so applying Lemma \ref{lem:nabla-and-diagonals}
\[ (T^nF)_S \simeq \nabla ( \nabla^{\times |S|-1}F \circ \iota_{S\backslash n-1}^*) \simeq \nabla^{\times |S|} F \circ (\iota_{S\backslash n-1}^* \times \iota_{S\backslash n-1}^*).\]
The proof is completed by the observation that the map \[\iota_{S\backslash n-1} \times \iota_{S\backslash n-1} \colon \mathcal{P}(S\backslash n-1) \times \mathcal{P}(S\backslash n-1) \hookrightarrow \mathcal{P}({\bf n}) \times \mathcal{P}({\bf n})\] coincides with the map $\iota_S \colon \mathcal{P}(S) \hookrightarrow \mathcal{P}({\bf n})$.
\end{proof}

We now construct the function $d_n\colon \mathcal{P}({\bf n})\to {\bf n+1}$ inducing the diagonal functor of \eqref{eq:real-8.1}. The function $d_n\colon \mathcal{P}({\bf n})\to {\bf n+1}$ is defined by $d_n(S)=|S|$, where $|S|$ denotes the cardinality of $S$.
Equivalently, if we think of the elements of $2^n$ as $n$-tuples of binary digits (the vertices of a unit $n$-cube), 
there is a function $d_n \colon 2^n \to {\bf n+1}$ that counts the number of 1s. This can be thought of as an order-preserving projection onto the diagonal of the cube.

Recall Remark \ref{r:diagrammatically}, which decomposes the functor $\Delta_n\colon [\cB, \cA] \to [\cB^{n+1}, \cA]$ as a composite
\[ \begin{tikzcd}
{\Delta_n \defeq [\cB, \cA]}\arrow[r,"{\Delta_{n-1}}"] & {[\cB^{n}, \cA]}\arrow[r, "{\nabla}"] & {[\cB^{n}\times \cB^n, \cA]}\arrow[r, "{L_n^*}"] & {[\cB^{n+1}, \cA],}\end{tikzcd}\] first applying the directional derivative $\nabla$ to $\Delta_{n-1}$ and then 
 restricting the variables along the diagonal functor $L_n\colon\cB^{n+1} \to \cB^n\times \cB^n$. This diagonal functor is also defined by precomposing with a certain function $a_n \colon {\bf 2} \times {\bf n} \to {\bf n+1}$. Writing elements in the domain as ordered pairs of elements in ${\bf 2}= \{0,1\}$ and ${\bf n} = \{0,\ldots, n-1\}$, the function $a_n$ adds the two coordinates. 

\begin{lem}\label{lem:commuting-projections} When $2^n \cong 2 \times 2^{n-1}$ is decomposed into its left-most coordinate paired with its remaining coordinates,  the following diagram commutes
\[ 
\begin{tikzcd}[column sep=large] 2^n \arrow[r, "d_n"] \arrow[d, "\cong"']  & {\bf n+1} \\ 2 \times 2^{n-1} \cong 2^{n-1} + 2^{n-1}  \arrow[r, "d_{n-1} + d_{n-1}"] &  {\bf n}+{\bf n} \cong {\bf 2} \times {\bf n} \arrow[u, "a_n"'] 
\end{tikzcd}
\]
\end{lem}
\begin{proof}
For a binary $n$-tuple $(e_1,\ldots, e_n)$, by definition $d_n (e_1,\ldots, e_n) = \sum_{i=1}^n e_i$. The lower composite sends $(e_1,\ldots, e_n)$ first to $(e_1, (e_2,\ldots, e_n))$, then to $(e_1, \sum_{i=2}^n e_i)$ then to $e_1 + \sum_{i =2}^n e_i$.
\end{proof}

\begin{lem}\label{lem:redefineDeltan} The higher order directional derivative $\Delta_n$ is chain homotopy equivalent to the composite
\[
\begin{tikzcd}
{[\cB,\cA]} \arrow[r, "\nabla^{\times n}"] & {[\cB^{2^n},\cA]} \arrow[r, "{(d_n^*)^*}"] & {[\cB^{n+1},\cA]}
\end{tikzcd}
\] which restricts $\nabla^{\times n}$ along the diagonal $d_n^* \colon \cB^{n+1} \to \cB^{2^n}$.
\end{lem} 

\begin{proof} 
 We prove this by induction on $n$. For the base case $n=1$, 
\[
\nabla F\circ d_1^*(V_1; X)=\nabla F(V_1; X)=\Delta_1F(V_1;X).
\]
For $n>1$, consider the diagram below:
\[
\begin{tikzcd}[row sep=large]
{[\cB, \cA]} \arrow[r, "{\nabla^{\times {n-1}}}"] \arrow[dr, "{\Delta_{n-1}}"'] &{ [\cB^{2^{n-1}},\cA]}\arrow[r, "{\nabla}"] \arrow[d, "{(d_{n-1}^{\ast})^{\ast}}"'] & {[\cB^{2^n},\cA]}\arrow[r, "{(d_{n}^{\ast})^{\ast}}"] \arrow[d, "{(d_{n-1}^* \times d_{n-1}^\ast)^\ast}"']
 & {[\cB^{n+1},\cA]}\\
&{[\cB^{n},\cA]} \arrow[r, "\nabla"'] &{[\cB^{2n},\cA]} \arrow[ur,"{(a_n^\ast)^{\ast}}"'] &
\end{tikzcd}
\]
Composition from $[\cB,\cA]$ to $[\cB^{n+1},\cA]$ along the top is $(d_n^{\ast})^{\ast}\circ \nabla^{\times n}$, whereas composition along the bottom is $\Delta_n$, since $a_n^*=L_n$.  The right triangle commutes by Lemma \ref{lem:commuting-projections}. The middle square commutes up to chain homotopy equivalence by Lemma \ref{lem:nabla-and-diagonals}. By the inductive hypothesis, the left triangle commutes up to chain homotopy equivalence, so  the result  follows by induction.  
\end{proof}

Write $\pi_R$ (respectively $\pi_L$) for the projection from a product category onto the product formed by its $k$ rightmost (respectively leftmost) variables, where the correct arity $k$ is determined by the context. For example,  the diagonal functor $L_n \colon \cB^{n+1} \to \cB^{2n}$ can be understood as  the pairing of functors $\pi_{L} \colon \cB^{n+1}\to \cB^n$ and $\pi_R \colon \cB^{n+1} \to \cB^n$, i.e., \begin{align*} L_n(V_n, \ldots , V_1, X) &= (\pi_{L}(V_n, \ldots , V_1, X), \pi_{R}(V_n, \ldots , V_1, X))\\ &= ((V_n, \ldots , V_1), (V_{n-1}, \ldots, X)).\end{align*} 

Let
\[ K_n \colon [\cC,\cB] \to [\cC^{n+1},\cB^{n+1}]\] be the functor defined by
\[ K_nG \defeq (\Delta_nG, \Delta_{n-1}G \circ \pi_R, \ldots, G \circ \pi_R).\]
Note that this expression appears as the argument for the  $\Delta_n F$ appearing on the right-hand side of the chain homotopy equivalence of Theorem \ref{t:higherchainrule}. The next lemma tells us that $K_nG \colon \cC^{n+1}\kto \cB^{n+1}$ and  $T^nG \colon \cC^{2^n} \kto \cB^{2^n}$ are related by restricting the domain of $T^nG$ and the codomain of $K_nG$ along the diagonal functor $d_n^*$.

\begin{lem}\label{lem:Knd}
For $G \colon \cC \kto \cB$, $d_n^*\circ K_nG \simeq T^nG\circ d_n^*$ as functors $\cC^{n+1} \to \cB^{2^n}$.
\end{lem}
\begin{proof} 
Again it is convenient to identify $2^n$ with $\mathcal{P}(n)$. It suffices to prove that these functors $\cC^{n+1} \kto \cB^{2^n}$ have chain homotopy equivalent components indexed by each subset $S \subset n$.

Lemma \ref{lem:Tn-component} tells us that $(T^nG)_S \simeq \nabla^{\times |S|} G \circ \iota^*_S$. As functors from $\cC^{n+1}$ to $\cC^{2^{|S|}}$, $\iota^{\ast}_S\circ d_n^{\ast}=d_{|S|}^{\ast}\circ \pi_{R}$. Hence, 
\[(T^nG)_S\circ d_n^\ast \simeq \nabla^{\times |S|} G \circ \iota^*_S \circ d_n^\ast \simeq  \nabla^{\times |S|}G\circ d_{|S|}^{\ast}\circ \pi_{R}\simeq \Delta_{|S|}G\circ \pi_{R},
\]
the last chain homotopy equivalence by Lemma \ref{lem:redefineDeltan}. 

On the other hand, immediately from the definition of $K_nG$,
\[(d_n^{\ast}\circ K_nG)_S=\Delta_{|S|}G\circ \pi_{R}.  \qedhere
\]
 \end{proof}

 \begin{proof}[Proof of Theorem \ref{t:higherchainrule}]
Consider the diagram below.
 \[
\begin{tikzcd}
\cC^{2^n} \arrow[r, squiggly, "T^nG"]  \arrow[dr, phantom, "\simeq"] & \cB^{2^n} \arrow[r, squiggly,  "T^nF"] \arrow[dr, phantom, "\simeq"]  & \cA^{2^n} \arrow[d, "\pi_L"] \\ \cC^{n+1} \arrow[u, "d_n^*"] \arrow[r,  squiggly, "K_nG"'] & \cB^{n+1} \arrow[u,  "d_n^*"] \arrow[r, squiggly,  "\Delta_nF"'] & \cA
\end{tikzcd}
\]
 The left and right squares commute up to chain homotopy equivalence  by Lemmas \ref{lem:Knd} and  \ref{lem:redefineDeltan}, respectively.    Using this diagram and  Lemma \ref{lem:tangent-functor} we see that 
 \begin{align*}
 \Delta_nF(\Delta_nG, \Delta_{n-1}G\circ\pi_R, \dots, G\circ\pi_R)&=\Delta_nF\circ K_nG\\&\simeq\pi_L\circ T^nF\circ T^nG\circ d_n^*\\
 &\simeq \pi_L \circ T^n(FG)\circ d_n^*.
 \end{align*}
 By Lemma \ref{lem:redefineDeltan}, we have
\[
\pi_L \circ T^n(FG)\circ d_n^*\simeq \Delta_n(FG),
\]
which completes the proof.
 \end{proof}


\section{Derivatives}\label{sec:derivatives}

So far, we have considered the functors $\nabla F(V; X)$ as analogs of the directional derivatives.  In classical analysis, the derivative of a single variable function $f\colon\RR\to \RR$ is a special case of the directional derivative in which the canonical `direction' is the positive direction, or the direction $1$.  Indeed, we see that setting $v=1$ in the definition of the directional derivative yields the following formula:
\[ \nabla f(v; x) = \lim_{t\to 0} \frac{1}{t} [ f(x+t)-f(x)] \]
which recovers the usual single-variable derivative $f'(x)$.

In \cite{JM:Classification1}, this idea is extended from functions $f\colon\RR \to \RR$ to functors $F\colon\cC\to \cD$ by using generating objects of $\cC$ in place of the unit $1\in \RR$.  The canonical example of a generating object is the object $R$ in the category of $R$-modules $\mathcal{M}\mathrm{od}_R$, where $R$ is a commutative ring with unit.  The object $R$ is a ``generating object" in the sense that  all other $R$-modules can be obtained from  finitely generated free $R$-modules using free resolutions and colimits.  In this section, we consider functors $F\colon\mathcal{M}\mathrm{od}_R\kto \cA$ and use the generating object $R$ in $\mathcal{M}\mathrm{od}_R$ to define a derivative $\frac{d}{dR}F$.  We use this derivative to obtain analogs of the chain rules for compositions of functions involving at least one function of a single variable.  We note that these results could be generalized to functors $F\colon\cC\kto \cA$ by following the definition of $\frac{d}{dC}F$ in \cite[Definition 2.13]{JM:Classification1}, 
however we have chosen to work with the category $\mathcal{M}\mathrm{od}_R$ for the sake of exposition.

\begin{defn}\label{defn:derivatives} \cite[Definition 2.13]{JM:Classification1}  For a functor $F \colon \mathcal{M}\mathrm{od}_R\kto \cA$ and $n \geq 1$, the $n$th \emph{derivative} of $F$ at $X$ is
\[ \frac{d^n}{dR^n}F(X) \defeq \nabla^nF(R,\ldots, R;X).\]
\end{defn}

The following lemma indicates another way to define $ \frac{d^n}{dR^n}F(X)$.

\begin{lem}\label{lem:derivatives} For a functor $F \colon \mathcal{M}\mathrm{od}_R \kto \cA$, there is a chain homotopy equivalence
\[ \nabla^nF(R,\ldots, R;X) \simeq \Delta_nF(0,\ldots, 0,R;X).\]
\end{lem}
\begin{proof}
Theorem \ref{thm:restatedgoal} expresses $\Delta_nF(V_n,\ldots, V_1;X)$ as a direct sum, indexed by partitions $\pi = \{S_1,\ldots, S_k\}$ of $\{1,\ldots, n\}$,  of terms $\nabla^{k}F(V_{|S_1|},\ldots, V_{|S_k|};X)$. Because $\nabla^kF$ vanishes if any of its ``direction'' variables is zero, the only partition that contributes a non-zero summand when $V_2=\cdots = V_n = 0$ is $\{ \{1\},\ldots, \{n\}\}$, the partition into one-element subsets. Thus 
\[\Delta_nF(0,\ldots, 0, R;X)\simeq  \nabla^nF(R, \ldots, R;X)\]
as claimed.
\end{proof}

As a consequence of Theorem \ref{t:higherchainrule} we obtain an analog of \cite[Theorem 2]{HMY}, which is stated in the introduction in equation \eqref{eq:HMY1}.

\begin{thm}\label{thm:HMY2}
For a composable pair $F\colon\cB\kto \cA$ and $G\colon\mathcal{M}\mathrm{od}_R\kto \cB$ of functors of abelian categories there is a chain homotopy equivalence
\[ \frac{d^n}{dR^n}(F\circ G)(X) \simeq \Delta_nF\left( \frac{d^n}{dR^n}G(X),\ldots, \frac{d}{dR}G(X); G(X)\right).\]
\end{thm}
\begin{proof}
By Theorem \ref{t:higherchainrule}
\[ \Delta_n (F\circ G)(0, \ldots , 0,R; X) \simeq \Delta_n F(\Delta_nG(0, \ldots , 0,R; X), \ldots , \Delta_1  G(R; X); G(X)).\] 
Applying  Lemma \ref{lem:derivatives} this becomes
\[ \frac{d^n}{dR^n}(F\circ G)(X) \simeq \Delta_nF\left( \frac{d^n}{dR^n}G(X),\ldots, \frac{d}{dR}G(X); G(X)\right).\qedhere\]
\end{proof}

Combining Theorems \ref{thm:HMY2} and \ref{thm:restatedgoal} gives the next corollary, a Fa\`a di Bruno-style characterization of the derivatives.  We note that Arone and Ching have obtained similar results for functors of spaces and spectra, by expressing such characterizations of derivatives in terms of composition products of symmetric sequences of derivatives (see \cite[Theorem 0.2]{AC} and \cite[Theorem 1.15]{C}).  

\begin{cor}\label{cor:faachainrule}
For a composable pair $F\colon\cB\kto \cA$ and $G\colon\mathcal{M}\mathrm{od}_R\kto \cB$ of functors of abelian categories there is a chain homotopy equivalence
\[ \frac{d^n}{dR^n}(F\circ G)(X) \simeq \bigoplus _{ \pi=\{S_1, \ldots, S_k\}}\nabla^{\pi}F\left( \frac{d^n}{dR^n}G(X), \ldots , \frac{d}{dR}G(X);G(X) \right)\]
where the sum is taken over all partitions $\pi$ of $\{ 1, \ldots , n \}$.
\end{cor}

The derivatives $\frac{d^n}{dR^n} F$ play an important role in functor calculus in that  they classify homogeneous degree $n$ functors, that is, degree $n$ functors $F$ with the property that $P_kF\simeq 0$ for $k<n$.   For a functor of abelian categories $F\colon\cB \kto \cA$, Johnson-McCarthy define the $n$th layer of its Taylor tower to be the homogeneous degree $n$ functor $D_nF:=\mathrm{hofiber}(P_nF\rightarrow P_{n-1}F)$, i.e., the mapping cone shifted down one degree.
This defines $D_nF$ up to quasi-isomorphism.  The next few results relate our work to their characterizations of this functor, which hold up to quasi-isomorphism and are denoted with the symbol $\simeq_{qi}$.
There is a quasi-isomorphism
\begin{align*} D_nF(X) & \simeq_{qi} (D_1^{(n)}\CR_nF(X, \ldots , X))_{h\Sigma_n}\\
					& \simeq_{qi} \nabla^nF(X, \ldots , X; 0)_{h\Sigma_n} \end{align*}
by \cite[Proposition 3.9, Corollary 5.11]{JM:Deriving}.  The $\Sigma_n$-action is the natural action that permutes the variables of $\CR_nF$.   These homogeneous layers are analogous to the homogeneous terms  $f^{(n)}(x)/n!$ of the Maclaurin series of a function $f$ in the calculus of functions.  In terms of the derivative of Definition \ref{defn:derivatives}, this yields the  following corollary.

\begin{cor} For $F \colon \mathcal{M}\mathrm{od}_R \kto \cA$ there is a quasi isomorphism:
\[ D_nF(R) \simeq  \left({\nabla}^nF(R,\ldots, R;0)\right)_{h\Sigma_n} \simeq_{qi} \left(\frac{d^n}{dR^n}F(0)\right)_{h\Sigma_n}.\]
\end{cor}

The objects $D_nF(R)$ are classifying objects for homogeneous degree $n$ functors by \cite[Theorem 5.13]{JM:Classification2}.  Thus, combining these results with Theorem \ref{thm:HMY2} and Corollary \ref{cor:faachainrule}, we have the following formula for the homogeneous layers of a composition.

\begin{cor} \label{cor:layers} For functors $G\colon\mathcal{M}\mathrm{od}_R\to \cB$ and $F\colon\cB\to \cA$,  there is a quasi-isomorphism
\[ D_n(F\circ G)(R) \simeq_{qi} \left (\Delta_nF\left( \frac{d^n}{dR^n}G(0), \ldots , \frac{d}{dR}G(0); G(0)\right) \right)_{h\Sigma_n}.\]
Furthermore, there is a quasi-isomorphism
\[D_n(F\circ G)(R)\simeq_{qi} \left( \bigoplus _{ \pi=\{S_1, \ldots, S_k\}}\nabla^{\pi}F\left( \frac{d^n}{dR^n}G(0), \ldots , \frac{d}{dR}G(0);G(0)\right ) \right)_{h\Sigma_n}\]
where the sum is taken over all partitions $\pi$ of $\{ 1, \ldots , n \}$.
\end{cor}
A more thorough investigation of the relationship between Corollary \ref{cor:layers}, the classification of functors of $\mathcal{M}\mathrm{od}_R $, and modules over these is warranted.  This will be the subject of future work.

\appendix

\section{A general bicomplex retraction}

Throughout this paper, and especially in Sections 4 and 5, we have replaced many quasi-isomorphisms related to properties of polynomial approximation and linearization by explicit chain homotopy equivalences.  In several cases (especially Proposition \ref{lem:5.7}) we need to have such equivalences between the total complexes of  bicomplexes whose rows are chain homotopy equivalent. In this appendix, we establish  general criteria under which the total complexes of  such bicomplexes are chain homotopy equivalent under weaker conditions than the ``chain homotopy of bicomplexes'' used elsewhere in this paper.    In Appendix B, we provide a proof of Proposition \ref{lem:5.7} as an application.

Let $A_{\bullet,\bullet}$ and $B_{\bullet,\bullet}$ be  first-quadrant bicomplexes. Contrary to the conventions in use elsewhere,  in this appendix we assume that bicomplexes have anti-commutative squares, as this convention will simplify the signs.

\begin{defn}\label{defn:row-wise-retraction} A morphism of first-quadrant bicomplexes $\iota \colon A_{\bullet, \bullet}\rightarrow B_{\bullet,\bullet}$ \emph{admits a row-wise strong deformation retraction} if  for all $p \geq 0$ there is a map of chain complexes $f_{p,\bullet}\colon B_{p,\bullet} \rightarrow A_{p,\bullet}$ 
\begin{enumerate} 
\item  that is a retraction for $\iota_{p,\bullet}$ and
\item so that $\iota_{p,\bullet}$ and $f_{p,\bullet}$ induce a strong chain homotopy equivalence between $A_{p,\bullet}$ and $B_{p,\bullet}$. That is,
there are morphisms $s\colon B_{p,q}\rightarrow B_{p,q+1}$ such that $ds+sd=1-\iota_{p,q}f_{p,q}$ and $s\iota_{p,q}=0$.
\end{enumerate} 
\end{defn}

Our aim is to prove:

\begin{thm}\label{thm:bicomplex-htpy-equiv}   Let $\iota\colon A_{\bullet, \bullet}\rightarrow B_{\bullet,\bullet}$ be a morphism of first-quadrant bicomplexes that admits a row-wise strong deformation retraction. 
Then $\iota$ induces a chain homotopy equivalence of total complexes $\Tot(A_{\bullet, \bullet})\rightarrow \Tot(B_{\bullet, \bullet})$.  \end{thm}

We prove Theorem \ref{thm:bicomplex-htpy-equiv} by constructing an explicit retraction in Proposition \ref{prop:generalretraction} and chain homotopy in Proposition \ref{prop:generalchain-homotopy}. We first establish our notation. 
Explicitly, the bicomplex $A_{\bullet, \bullet} = \{ A_{p,q} \mid p,q \geq 0\}$ has:
\begin{itemize}
\item horizontal differentials $d^A \colon A_{p,q} \to A_{p,q-1}$ so that $(d^A)^2=0$, and 
\item vertical differentials $e^A \colon A_{p,q} \to A_{p-1, q}$ so that $(e^A)^2=0$,
\item so that $e^Ad^A+d^Ae^A=0$ (squares anti-commute).
\end{itemize}
Similarly, the bicomplex $B_{\bullet, \bullet} = \{ B_{p,q} \mid p,q \geq 0\}$ has:
\begin{itemize}
\item horizontal differentials $d^B \colon B_{p,q} \to B_{p,q-1}$ so that $(d^B)^2=0$, and 
\item vertical differentials $e^B \colon B_{p,q} \to B_{p-1, q}$ so that $(e^B)^2=0$,
\item so that $e^Bd^B+d^Be^B=0$ (squares anti-commute).
\end{itemize}
We will omit superscripts and simply use $d$ and $e$ when the context is clear.   
In addition to the chain homotopies described explicitly in the statement of Theorem \ref{thm:bicomplex-htpy-equiv} it is convenient to allow  additional maps $d \colon B_{p,0} \to B_{p,-1}$ and $s \colon B_{p,-1} \to B_{p,0}$ all equal to zero. In particular, $ds \colon B_{p,0} \to B_{p,0}$ is the  map $1-\iota f$.

The next lemma records some commutativity relations that follow from the defining relations and conditions of \ref{thm:bicomplex-htpy-equiv}:
\[ d^2=0 \qquad e^2 = 0 \qquad de + ed = 0 \qquad ds + sd = 1-\iota f\qquad s\iota=0 \] \[ f\iota = 1 \qquad fd = df \qquad \iota d = d \iota \qquad \iota e = e\iota .\]

\begin{lem}\label{lem:relations}
For any $k\geq 0$, 
\begin{enumerate}
\item $ef(-es)^k+df(-es)^{k+1}=f(-es)^ke+f(-es)^{k+1}d,$
\item $es(-es)^k+ds(-es)^{k+1}=-\iota f(-es)^{k+1}-sd(-es)^{k+1}$, 
\item  $s(-es)^{k+1}d+s(-es)^ke=-(-se)^{k+1}ds$, and
\item $sd(-es)^{k+1}=-(-se)^{k+1}ds$.
\end{enumerate}
Here $(-es)^k=(-1)^k(es)^k.$
\end{lem}
\begin{proof}  We prove (i) and leave the others as exercises for the reader. 
We prove this by induction on $k$.  When $k=0$, the left hand side is
\begin{align*}
ef+df(-es)&=ef+feds\\
&=ef+fe(1-sd-\iota f)\\
&=ef+fe-fesd-fe\iota f.\\
\intertext{
Using the facts that $\iota$ and $e$ commute, and $f\iota=1$, this becomes}
&=ef+fe-fesd-ef\\
&=fe-fesd,
\end{align*}
as desired.

Assuming the statement holds for $k$, we have
\begin{align*}
ef(-es)^{k+1}+df(-es)^{k+2}&=(ef(-es)^k+df(-es)^{k+1})(-es)\\
&=(f(-es)^ke+f(-es)^{k+1}d)(-es).\\
\intertext{Since $e^2=0$ and $d$ and $e$ anticommute, this is }
&=(-1)^{k+1}f(es)^{k+1}eds\\
&=(-1)^{k+1}f(es)^{k+1}e(1-sd-\iota f).
\intertext{Since $e$ and $\iota$ commute and $s\iota=0$, this is simply}
&=f(-es)^{k+1}e+f(-es)^{k+2}d
\end{align*}
and the proof is complete.
\end{proof}

Recall the \emph{total complex} $\Tot(C)$ of a first-quadrant bicomplex $C_{\bullet,\bullet}$ is the chain complex with \[\Tot(C)_n \defeq C_{n,0} + \cdots + C_{0,n}\] (where ``$+$'' is less-cluttered notation for ``$\oplus$'') and with differential given by the matrix
\[
\begin{tikzcd}[column sep=100pt, ampersand replacement=\&]
\Tot(C)_n \defeq C_{n,0}+\cdots + C_{0,n}  \arrow[r, "{\left(\begin{array}{cccccc} e & d & 0 & \cdots & \cdots & 0 \\ 0 & e & d & 0 & \cdots& 0 \\ \vdots & \ddots & \ddots & \ddots & \ddots & \vdots \\   0 & \cdots&0 & e & d & 0 \\ 0 &  \cdots& \cdots & 0 & e & d \end{array}\right)}"] \& C_{n-1,0} + \cdots + C_{0,n-1} \eqdef \Tot(C)_{n-1}.
\end{tikzcd}
\]
See \cite[3.1.27]{riehl:context} for an explanation of the matrix notation for a map between finite direct sums.

\begin{rmk}\label{rmk:commutative-bicomplexes}
The total complex defined above assumes we have a bicomplex with anti-commutative squares.  When working with a bicomplex with commutative squares, one must introduce  signs to define the total complex.  A standard convention is to multiply all of the maps in every odd row by $-1$ to obtain a bicomplex with anti-commutative squares.  As this does not affect the conditions needed for Theorem \ref{thm:bicomplex-htpy-equiv}, we can apply it to bicomplexes with commutative squares as well.  
\end{rmk}

\begin{prop}\label{prop:generalretraction} For a morphism of first-quadrant bicomplexes $\iota:A_{\bullet, \bullet}\to B_{\bullet,\bullet}$ satisfying the conditions of Theorem \ref{thm:bicomplex-htpy-equiv}, the induced morphism of total complexes $\Tot(\iota):\Tot(A)_\bullet\to \Tot(B)_\bullet$ admits a retraction $\rho \colon \Tot(B)_\bullet \to \Tot(A)_\bullet$ defined in degree $n$ by

\[
\begin{tikzcd}[column sep=110pt, ampersand replacement=\&]
\Tot(B)_n \defeq B_{n,0}+\cdots + B_{0,n}  \arrow[r, "{\left(\begin{array}{cccccc} f & 0 & \cdots & \cdots & 0 & 0 \\ f(-es) & f &0 & \ddots & \ddots & 0 \\ f(-es)^2 & f(-es) & f & 0 & \ddots & \vdots \\ f(-es)^3 & f(-es)^2 & f(-es) & f & 0 & \vdots \\
 \vdots & \vdots & \vdots & \vdots & \ddots & \vdots \\
f(-es)^{n-1}&f(-es)^{n-2}&\cdots&\cdots&f &0\\ 
 f(-es)^{n} & f(-es)^{n-1} & \cdots & \cdots&\cdots & f \end{array}\right)}"] \& A_{n,0} + \cdots + A_{0,n} \eqdef \Tot(A)_{n}.
\end{tikzcd}
\]
That is, the $(n+1)\times (n+1)$ matrix defining $\rho_n$ is the lower triangular matrix whose  entry in the $i$th row, $j$th column, $j\leq i$,  is $f_{n-(i-1),i-1}(-es)^{i-j}$.
Here $(-es)^k = (-1)^k(es)^k$.
\end{prop}
\begin{proof}
We need to show that $\rho$ is both a chain map and a retraction for $\iota$.  We start by showing that $\rho$ is chain map.  For each $n\geq 1$, we must verify that $\partial_n\rho_n=\rho_{n-1}\partial_n$ where $\partial_n$ is the total complex differential.  We do so by verifying that corresponding entries in the $(n-1)\times n$ matrices defining $\partial_n\rho_n$ and $\rho_{n-1}\partial_n$ agree.  Using $i$ for the row number and $j$ for the column number of the matrix (starting from the upper left corner), we consider three cases: $j>i+1$, $j=i+1$, and $j\leq i$.   We note that the entry in row $i$, column $j$ of these matrices is the component of our map from $B_{n-j+1,j-1}$ to $A_{n-i,i-1}$.  For each case, we display a square diagram whose horizontal arrows are components of $\rho$ and vertical arrows are components of the total complex differential $\partial$. The upper-right composite is the component of the map $\partial\rho$ and the lower-left composite is the component of the map $\rho\partial$.

When $j>i+1$, the diagram is

\[
\begin{tikzcd}[column sep=180pt, row sep=large, ampersand replacement=\&]
B_{n-j+1,j-1} \arrow[d,"{\left(\begin{array}{c} d \\ e \end{array}\right)}"']  \arrow[r, "{\left(\begin{array}{c} 0\\ 0 \end{array}\right)}"] \& A_{n+1-i,i-1} +A_{n-i,i} \arrow[d, "{\left( \begin{array}{cc} e & d \end{array}\right)}"] \\
B_{n-j+1,j-2} + B_{n-j,j-1}  \arrow[r, "{\left(\begin{array}{cc} 0& 0  \end{array}\right)}"'] \& A_{n-i,i-1}.
\end{tikzcd}
\]
As this diagram clearly commutes, we have  $\partial_n\rho_n=\rho_{n-1}\partial_n$ in this case.  

When $j=i+1$, the diagram is
\[
\begin{tikzcd}[column sep=180pt, row sep=large, ampersand replacement=\&]
B_{n-i,i} \arrow[d,"{\left(\begin{array}{c} d \\ e \end{array}\right)}"']  \arrow[r, "{\left(\begin{array}{c} 0\\ f_{n-i,i} \end{array}\right)}"] \& A_{n+1-i,i-1} +A_{n-i,i} \arrow[d, "{\left( \begin{array}{cc} e & d \end{array}\right)}"] \\
B_{n-i,i-1} + B_{n-i-1,i}  \arrow[r, "{\left(\begin{array}{cc} f_{n-i,i-1}& 0  \end{array}\right)}"'] \& A_{n-i,i-1}.
\end{tikzcd}
\]
The composition through the upper right corner is $df_{n-i,i}$ and the composition through the lower left corner is $f_{n-i,i-1}d$.  Since  $f$ is a map of chain complexes between the rows of $B_{\bullet,\bullet}$ and $A_{\bullet, \bullet}$, these are equal and we have $\partial_n\rho_n=\rho_{n-1}\partial_n$ in this case as well.  

Finally, when $j\leq i$, we must verify that the diagram below commutes:
\[
\begin{tikzcd}[column sep=180pt, row sep=large, ampersand replacement=\&]
B_{n-j+1,j-1} \arrow[d,"{\left(\begin{array}{c} d \\ e \end{array}\right)}"']  \arrow[r, "{\left(\begin{array}{c} f(-es)^{i-j}\\ f(-es)^{i-j+1} \end{array}\right)}"] \& A_{n+1-i,i-1} +A_{n-i,i} \arrow[d, "{\left( \begin{array}{cc} e & d \end{array}\right)}"] \\
B_{n-j+1,j-2} + B_{n-j,j-1}  \arrow[r, "{\left(\begin{array}{cc} f(-es)^{i-j+1}& f(-es)^{i-j}  \end{array}\right)}"'] \& A_{n-i,i-1},
\end{tikzcd}
\]
that is, that 
\[
f(-es)^{i-j+1}d+f(-es)^{i-j}e=ef(-es)^{i-j}+df(-es)^{i-j+1}.
\]
But this is part (i) of  Lemma \ref{lem:relations}, so $\rho$ is a chain map.  

To confirm that $\rho_n\iota_n=1$ where 
\[
\iota_n=\left (\begin{matrix}\iota_{n,0}&0 &0&\cdots &0\\
0&\iota_{n-1,1}&0&\cdots&0\\
\vdots&\vdots&\ddots&\ddots&0\\
0&0&0&\cdots&\iota_{0,n}
\end{matrix}\right )
\]  we note that the entry in row $i$, column $j$ of $\rho_n\iota_n$ is 
\begin{enumerate}
\item $0$, if $j>i$, 
\item $f\iota$, if $i=j$, and
\item $f(-es)^{i-j}\iota$, if $i>j$.
\end{enumerate}
Since $f\iota=1$ and $s\iota=0$, $\rho_n$ is a retraction for $\iota_n$.  
\end{proof}

{\color{red}
}
\begin{prop}\label{prop:generalchain-homotopy}
The composite map $\iota\rho \colon \Tot(B)_\bullet \to \Tot(B)_\bullet$ is chain homotopic to the identity via the chain homotopy $\sigma \colon \Tot(B)_\bullet \to \Tot(B)_{\bullet +1}$ defined in degree $n$ by
\[
\begin{tikzcd}[column sep=110pt, ampersand replacement=\&]
\Tot(B)_n \defeq B_{n,0}+\cdots + B_{0,n}  \arrow[r, "{\left(\begin{array}{cccccc} 0 & 0 & \cdots & \cdots & 0 & 0 \\ s & 0 & \ddots & \ddots & \ddots & 0 \\ s(-es) & s & 0 & \ddots & \ddots & \vdots \\ s(-es)^2 & s(-es) & s & 0 & \ddots & \vdots \\
 \vdots & \vdots & \vdots & \ddots & \ddots & \vdots \\   s(-es)^{n-1} & s(-es)^{n-2} & \cdots & s(-es) & s & 0 \\ s(-es)^n & \cdots & \cdots & \cdots & s(-es) & s \end{array}\right)}"] \& B_{n+1,0} + \cdots + B_{0,n+1} \eqdef \Tot(B)_{n+1}.
\end{tikzcd}
\]
\end{prop}

\begin{proof}
Writing $\partial_n \colon \Tot(B)_n \to \Tot(B)_{n-1}$ for the total complex differential, we must verify that $\partial_{n+1}\sigma_n + \sigma_{n-1}\partial_n \colon \Tot(B)_n \to \Tot(B)_n$ is the matrix
\[
\begin{tikzcd}[column sep=110pt, ampersand replacement=\&]
\Tot(B)_n \defeq B_{n,0}+\cdots + B_{0,n}  \arrow[r, "{\left(\begin{array}{ccccc} 1-\iota f & 0 &  \cdots & \cdots & 0 \\ -\iota f (-es) & 1-\iota f & 0  & \cdots& 0 \\
-\iota f(-es)^2&-\iota f(-es)&1-\iota f&\cdots  &0\\
 \vdots & \vdots & \vdots & \ddots & \vdots \\ -\iota f(-es)^{n-1}&\cdots&-\iota f(-es)&1-\iota f&0    \\  -\iota f(-es)^n &  \cdots& -\iota f(-es)^2 & -\iota f (-es) & 1-\iota f \end{array}\right)}"] \& B_{n,0} + \cdots + B_{0,n} \eqdef \Tot(B)_{n}.
\end{tikzcd}
\]
i.e., that 
\[\partial_{n+1}\sigma_n + \sigma_{n-1}\partial_n = \id_{\Tot(B)_n} - \iota_n\rho_n.\] For each pair $1 \leq i,j \leq n+1$, we must verify that the component of $\partial_{n+1}\sigma_n + \sigma_{n-1}\partial_n$ from $B_{n-j+1,j-1}$ to $B_{n-i+1, i-1}$  agrees with the matrix entry in the $j$th column and $i$th row, counting from the left and from the top. As in the proof of Proposition \ref{prop:generalretraction}, we do this in several cases. For each case, we display a square diagram whose horizontal arrows are components of $\sigma$ and vertical arrows are components of the total complex differential $\partial$. The upper-right composite is the component of the map $\partial\sigma$ and the lower-left composite is the component of the map $\sigma\partial$ whose sum we are computing.

We begin with the case where $j>i$.  In this case, the entries represent maps $B_{n-j+1,j-1} \to B_{n-i+1,i-1}$.  We must show that the sum of the two composites
\[
\begin{tikzcd}[column sep=180pt, row sep=large, ampersand replacement=\&]
B_{n-j+1,j-1} \arrow[d,"{\left(\begin{array}{c} d \\ e \end{array}\right)}"']  \arrow[r, "{\left(\begin{array}{c} 0 \\ 0 \end{array}\right)}"] \& B_{n-i+2,i-1} +B_{n-i+1,i} \arrow[d, "{\left( \begin{array}{cc} e & d \end{array}\right)}"] \\
B_{n-j+1,j-2} + B_{n-j,j-1}  \arrow[r, "{\left(\begin{array}{cc} 0 & 0  \end{array}\right)}"'] \& B_{n-i+1,i-1}
\end{tikzcd}
\]
is $0$, which is evident.

We next verify the diagonal entries $B_{n-j+1,j-1} \to B_{n-j+1,j-1}$ for $1\leq j\leq n+1$. Here we must show that the sum of the maps 
\[
\begin{tikzcd}[column sep=180pt, row sep=large, ampersand replacement=\&]
B_{n-j+1,j-1} \arrow[d,"{\left(\begin{array}{c} d \\ e \end{array}\right)}"']  \arrow[r, "{\left(\begin{array}{c} 0 \\ s \end{array}\right)}"] \& B_{n-j+2,j-1} +B_{n-j+1,j} \arrow[d, "{\left( \begin{array}{cc} e & d \end{array}\right)}"] \\
B_{n-j+1,j-2} + B_{n-j,j-1}  \arrow[r, "{\left(\begin{array}{cc} s & 0  \end{array}\right)}"'] \& B_{n-j+1,j-1}
\end{tikzcd}
\]
is $1-\iota f \colon B_{n-j+1,j-1} \to B_{n-j+1,j-1}$, which is so because $ds + sd = 1-\iota f$.

It remains to verify the lower diagonal entries. Here we wish to show that for $i > j > 0$ the sum of the maps
\[
\begin{tikzcd}[column sep=180pt, row sep=large, ampersand replacement=\&]
B_{n-j+1,j-1} \arrow[d,"{\left(\begin{array}{c} d \\ e \end{array}\right)}"']  \arrow[r, "{\left(\begin{array}{c} s(-es)^{i-j-1} \\ s(-es)^{i-j} \end{array}\right)}"] \& B_{n-i+2,i-1} +B_{n-i+1,i} \arrow[d, "{\left( \begin{array}{cc} e & d \end{array}\right)}"] \\
B_{n-j+1,j-2} + B_{n-j,j-1}  \arrow[r, "{\left(\begin{array}{cc} s(-es)^{i-j} & s(-es)^{i-j-1}  \end{array}\right)}"'] \& B_{n-i+1,i-1}
\end{tikzcd}
\]
is the  map $-\iota f(-es)^{i-j}\colon B_{n-j+1,j-1} \to B_{n-i+1,i-1}$. Here the sum of the composites is
\[es(-es)^{i-j-1}+ds(-es)^{i-j}+s(-es)^{i-j}d+s(-es)^{i-j-1}e.\]
Parts (ii),  (iii), and (iv) of Lemma \ref{lem:relations} guarantee that this is equal to $-\iota f(-es)^{i-j}$.  
\end{proof}

As an immediate corollary of \ref{thm:bicomplex-htpy-equiv}, we have the following.

\begin{cor}\label{cor:chain-htpy-equiv} Let $A_{\bullet, \bullet}$ be a first-quadrant bicomplex so that every row except the zeroth row $A_{0,\bullet}$ is contractible. Then the natural inclusion $A_{0,\bullet}\hookrightarrow\Tot(A)_\bullet$ is a chain homotopy equivalence.
\end{cor}

\begin{proof}
In this case, the evident inclusion $\iota \colon A_{0,\bullet} \hookrightarrow \Tot(A)_\bullet$ of chain complexes admits a retraction $\rho \colon \Tot(A)_\bullet \to A_{0,\bullet}$ defined in degree $n$ by
\[
\begin{tikzcd}[column sep=180pt, ampersand replacement=\&]
\Tot(A)_n \defeq A_{n,0}+\cdots + A_{0,n}  \arrow[r, "{\left(\begin{array}{cccccc} (-es)^n & (-es)^{n-1} & \cdots & (-es)^2 & (-es) & 1  \end{array}\right)}"] \& A_{0,n}.
\end{tikzcd} \qedhere
\]
\end{proof}

\section{Proof of Proposition \ref{lem:5.7}}\label{a:AppB}

We now apply Theorem \ref{thm:bicomplex-htpy-equiv} to prove: 

{
\renewcommand{\thethm}{\ref{lem:5.7}}
\begin{prop} For any composable pair of functors $F \colon \cB \kto \cA$ and $G \colon \cC \kto \cB$ with $G$ reduced, there is a chain homotopy equivalence
\[ D_1(F \circ G) \simeq D_1F \circ D_1G.\]
\end{prop}
\addtocounter{thm}{-1}
}

In fact, it suffices to assume that $F \colon\cB \kto \cA$ is also reduced. To see this, note that by  direct computation: 

\begin{lem}\label{lem:CR1} Given $F \colon \cB \kto \cA$ and $G \colon \cC \kto \cB$ with $G$ reduced, then
$\CR_1(FG) \cong \CR_1F\circ G$.
\end{lem}
\begin{proof}
Because $G$ is reduced, both terms are direct sum complements of $FG(0)$ in $FG$.
\end{proof}

Lemma \ref{lem:CR1} has the following corollary:

\begin{cor}\label{cor:just-reduce} For $G$ reduced and $F$ not necessarily reduced $D_1(F\circ G) \simeq D_1(\CR_1F \circ G)$. In particular, to prove that $D_1F \circ D_1G \simeq D_1(F\circ G)$ it suffices to assume that $F$ is also reduced.
\end{cor}
\begin{proof}
By Lemma \ref{lem:reduced-vs-unreduced}, $D_1F \cong D_1(\CR_1F)$ for all $F$. So \[ D_1(F \circ G) \cong D_1(\CR_1(F\circ G)) \cong D_1(\CR_1F \circ G),\] as claimed. Since $D_1F \circ D_1G \cong D_1\CR_1F \circ D_1G$ it suffices to consider the case when both functors are reduced.
\end{proof}

On account of Corollary \ref{cor:just-reduce}, for the remainder of this section, we restrict our consideration to functors that are  reduced. 

In the Kleisli category, the composition of $G \colon \cC\kto\cB$ with $F \colon\cB\kto\cA$ is defined by prolonging the functor $F \colon \cB\to \Ch\cA$ to a functor $\Ch(F) \colon \Ch\cB \to \Ch\Ch\cA$ via the Dold-Kan equivalence and then applying the totalization $\Tot \colon \Ch\Ch\cA \to \Ch\cA$. When $F$ is strictly reduced, there is a functor $\tilde F \colon \Ch\cB \to \Ch\Ch\cA$ defined by simply applying the functor directly to each term in the chain complex. When $F$ is also linear,  \cite[Lemma 5.4]{JM:Linearization} show that the two prolongations $\Tot(\tilde F)$ and $\Tot(\Ch(F))$ are quasi-isomorphic.  Lemma \ref{lem:prolong} shows that when $F$ is a linearization $D_1H$ of some functor $H \colon \cB \kto \cA$, in which case Lemma \ref{lem:D1-linearity}(i) shows that $F$ is strictly reduced and linear in the sense of Definition \ref{defn:linear} (preserving finite direct sums up to natural chain homotopy equivalence), these procedures are in fact  chain homotopy equivalent.  We will use this in the proof of Proposition \ref{lem:5.7} to construct the composite $D_1F \circ D_1G$ using the simpler form of the prolongation. 

We note that the definition of $\tilde F$ only makes sense when $F$ is strictly reduced.  Otherwise, applying $F$ degreewise to the objects and maps in a chain complex may not produce a chain complex, as $F(\partial\circ\partial)$ may no longer be $0$.  

The proof of Lemma \ref{lem:prolong} will proceed in several steps, corresponding to the verification of the hypotheses of Theorem \ref{thm:bicomplex-htpy-equiv}.  Write $K \colon \Ch\cB\to \cB^{\DDelta^\op}$ and $N \colon \cB^{\DDelta^\op} \to \Ch\cB$ for the functors in the Dold-Kan equivalence.   The functor $N$ is naturally chain homotopy equivalent to another functor $M\colon\cB^{\DDelta^\op}\rightarrow \Ch\cB$ with a simpler definition: for a simplicial object $Y \in \cB^{\DDelta^\op}$, $M(Y)_n = Y_n$ with differentials defined to be the alternating sum of the face maps; see \cite[\S 8.4]{Weibel} and \cite[III.2.4]{GJ} for definitions of these functors and the chain homotopy equivalence.
Let $F_\ast\colon \cB^{\DDelta^\op} \to \Ch{\cA}^{\DDelta^\op}$ be the functor defined by post-composition with $F \colon \cB \to \Ch\cA$.  Recall that $\Ch(F)=N\circ F_{\ast}\circ K$ defines the prolongation. 

\begin{rmk}\label{rem:iota} For a chain complex $X \in \Ch\cB$, there is a canonical comparison morphism $\iota_*: K\tilde F(X)\rightarrow F_{\ast}(KX)$ in $\Ch\cA^{\DDelta^\op}$ whose component in simplicial degree $n$ is the natural map 
\begin{equation*} \label{e:DK-inclusion}
\iota \colon \bigoplus_{[n]\twoheadrightarrow [k]}F(X_k)\rightarrow F\left (\bigoplus_{[n]\twoheadrightarrow [k]} X_k\right)
\end{equation*}
where the sums are over all surjections $[n]\twoheadrightarrow[k]$ in $\DDelta$.   
Composing with the functor $M$, this defines a canonical map of bicomplexes $\iota: =M(\iota_*)\colon MK\tilde F(X)\rightarrow MF_{\ast}(KX)$, where the latter functor is chain homotopy equivalent to $\Ch F$.  
\end{rmk}

The first step is to show that this map is a chain homotopy equivalence in each row.  This follows from the next two lemmas.

\begin{lem}\label{lem:B-is-contractible} 
The inclusion of chain complexes $\iota\colon A\rightarrow A\oplus B$ is a chain homotopy equivalence if and only if $B$ is contractible.
\end{lem}

\begin{proof}

If $\iota$ is a chain homotopy equivalence, then there is a map $\rho\colon A\oplus B\rightarrow A$ such that $\iota\circ \rho\simeq \mathrm{id}_{A\oplus B}$.  In other words, there is a chain homotopy 
\[
\begin{tikzcd}[ampersand replacement=\&, column sep=huge]
A_n \oplus B_n \arrow[r, "{s=\left(\begin{array}{cc} s_{AA} & s_{BA} \\ s_{AB} & s_{BB} \end{array}\right)}"] \& A_{n+1} \oplus B_{n+1}
\end{tikzcd}
\]
so that 
\begin{equation*}\label{eq:chain-relation}
ds+sd=\mathrm{id}_{A\oplus B}-\iota\circ \rho \colon A_n \oplus B_n \to A_n \oplus B_n.
\end{equation*}
The component of the map $\iota \circ \rho \colon A_n \oplus B_n \to A_n \oplus B_n$ from $B_n$ to $B_n$ is the zero map. Thus, the component from $B_n$ to $B_n$ of the relation \eqref{eq:chain-relation} asserts that $d_B s_{BB} + s_{BB} d_B = \id_B$, where $d_B$ is the differential in $B$. Thus, the component $s_{BB} \colon B_n \to B_{n+1}$ of $s$ provides the desired contracting homotopy.

On the other hand, if $B$ is contractible then there is a chain homotopy equivalence $h:B_n\to B_{n+1}$ so that
\[d_B h + hd_B = id_B.\]
Then, the map defined by 
\[
\begin{tikzcd}[ampersand replacement=\&, column sep=huge]
A_n \oplus B_n \arrow[r, "{s=\left(\begin{array}{cc} 0 & 0 \\ 0 & h=s_{BB} \end{array}\right)}"] \& A_{n+1} \oplus B_{n+1}
\end{tikzcd}
\]
is a chain homotopy equivalence between $id_{A\oplus B}$ and $\iota\pi$ where $\pi$ is the canonical projection map.  By direct verification, 
\[ ds+sd=d_Bh+hd_B= id_{A\oplus B}-\iota\pi. \qedhere \]
\end{proof}

\begin{lem}\label{lem:xfx-contract} Let $F=D_1H$.  Then the chain complex $\CR_nF$ is chain contractible for all $n\geq 2$. \end{lem}

\begin{proof}  We proceed by induction on $n$.  When $n=2$, the lemma is a consequence of Lemma \ref{lem:natural-contraction} and Proposition 4.5(i), since $D_1H\cong P_1(\CR_1 H)$ as in Remark 5.2(ii).  Now suppose that $\CR_n F$ is chain contractible.  By definition of the cross effects, $\CR_{n+1}F$ is a direct summand of $\CR_nF$:
\[ \CR_{n}F(X\oplus Y, Z_2, \ldots , Z_n)\cong \CR_nF(X, Z_2, \ldots , Z_n)\oplus \CR_nF(Y, Z_2, \ldots , Z_n) \oplus \CR_{n+1}F(X,Y,Z_2, \ldots , Z_n).\]
Since $\CR_nF(X, Z_2, \ldots , Z_n)$ and $\CR_nF(Y, Z_2, \ldots , Z_n)$ are contractible by hypothesis, the inclusion $\CR_{n+1}F(X,Y,Z_2, \ldots , Z_n)\hookrightarrow \CR_{n}F(X\oplus Y, Z_2, \ldots , Z_n)$ is a chain homotopy equivalence by Lemma \ref{lem:B-is-contractible}.  The result follows since $\CR_{n}F(X\oplus Y,Z_2, \ldots , Z_n)$ is contractible by hypothesis.
\end{proof}

An immediate consequence of Lemmas \ref{lem:xfx-contract} and \ref{lem:B-is-contractible}  is that the $n$th row of the bicomplex obtained by applying $F$ degreewise to a chain complex $X$ is chain homotopy equivalent to the $n$th row of the bicomplex  obtained by prolonging $F$ when evaluating on $X$.  This, together with Theorem \ref{thm:bicomplex-htpy-equiv}, proves the following lemma.

\begin{lem}\label{lem:prolong}
Let $H\colon\cB\kto \cA$ be a functor and $F=D_1H$.  Then $\mathrm{Tot}(\Ch(F))$ and $\mathrm{Tot}(\tilde F)$ are chain homotopy equivalent.  
\end{lem}

\begin{proof}  
We first show that $\Tot\ M\circ F_* \circ K$ and $\Tot\ \tilde{F}$ are chain homotopy equivalent.  In particular, we show that the totalization of the canonical comparison map $\iota\colon MK\tilde{F} \to MF_*K$ of Remark \ref{rem:iota} is a chain homotopy equivalence.  Note that $\iota$ is induced by the inclusion maps $X_k\to \oplus_{[n]\twoheadrightarrow [k]} X_k$; for this reason, $\iota_*$ commutes with the face maps of the simplicial objects $K\tilde{F}$ and $F_*K$.  Thus, $\iota$ itself is a map of bicomplexes.  The map $\iota$ has a retraction, $\rho$, in each row (where a row of $MK\tilde{F}$ or $MF_*K$ corresponds to a fixed  simplicial degree in $K\tilde{F}$ or $F_*K$, respectively), given by projection.  In particular, since 
\[F\left (\bigoplus_{[n]\twoheadrightarrow [k]} X_k\right) \cong \bigoplus_{[n]\twoheadrightarrow [k]} F(X_k) \oplus_{k=1}^n\left(\oplus_{j_1< ... < j_k} \CR_kF(X_{j_1}, \ldots , X_{j_k}) \right),\]
there is a canonical projection map $\rho_*\colon F_*KX \to K\tilde{F}$ given by mapping each of the $\CR_kF$ terms to $0$.  
This induces a map $\rho := M(\rho_*)$ of chain complexes on each row.  Since $\iota$ and $\rho$ are inclusion and projection of direct summands, the composite $\rho\circ\iota$ is the identity on each row of $MK\tilde{F}$.
 
 We must show that $\iota\circ \rho$ is chain homotopic to the identity map $1_{MF_*K}$ in each row.  Let $B=\oplus_{k=1}^n\left(\oplus_{j_1< ... < j_k} \CR_kF(X_{j_1}, \ldots , X_{j_k})\right)$.  By Lemma \ref{lem:xfx-contract}, in each row there is a chain contraction $h\colon B\to B$ with $sd_B+d_Bs=1_B$.  As in the proof of Lemma \ref{lem:B-is-contractible}, the map 
 \[ s\colon \bigoplus_{[n]\twoheadrightarrow [k]} F(X_k)\oplus B \to \bigoplus_{[n]\twoheadrightarrow [k]} F(X_k) \oplus B\]
 defined by 
 \[ s=\left(\begin{array}{cc} 0 & 0 \\ 0 & h \end{array}\right) \]
 is the desired chain homotopy equivalence.  Note that $s$ satisfies $s\iota=0$ since $s$ is defined to be $0$ on $\bigoplus_{[n]\twoheadrightarrow [k]} F(X_k)$.  
 
 By Theorem \ref{thm:bicomplex-htpy-equiv}, $s$ induces a chain homotopy equivalence between the total complexes $ \Tot\ M\circ F_* \circ K$ and $\Tot\ \tilde{F}$.

To complete the proof, we use the fact that $N$ and $M$ are chain homotopy equivalent. We have the following commuting square in which the left, right, and bottom arrows are chain homotopy equivalences:
\[
\begin{tikzcd}
\Tot NK\tilde F(X) \arrow[r] \arrow[d, "\simeq"'] &\Tot NF_{\ast}(KX) \arrow[d, "\simeq"]\\ 
  \Tot MK\tilde{F}(X)\arrow[r, "\simeq"]&\Tot MF_{\ast}(KX)
\end{tikzcd}
\]
This implies that the top arrow is a chain homotopy equivalence.  The result follows by noting that $NK$ is isomorphic to the identity, and $NF_\ast (KX) = \Ch(F)(X)$.  
\end{proof}

Since we can safely assume that both $F$ and $G$ are reduced, the chain complexes $D_1F$, $D_1G$ and $D_1(FG)$ are defined as comonad resolutions for the second cross effect comonad $C_2$.  For readability, we simplify our notation and write ``$C$'' for the comonad $C_2$ and ``$L\dashv R$'' for the adjoint functors of Corollary \ref{cor:comonad} defining the comonad $C=LR$. 
Recall $L$ is the diagonal functor, which has an important property:

\begin{lem}\label{lem:diagonal-commutativity} For any $G \colon \cC \kto \cB$ and $F \colon \cB \kto \cA$
\[
\begin{tikzcd}
\Fun_*(\cC^2, \cB) \arrow[r, "L"] \arrow[d, "F_*"'] & \Fun_*(\cC,\cB) \arrow[d, "F_*"] & \Fun_*(\cB^2, \cA) \arrow[r, "L"] \arrow[d, "G^*"'] & \Fun_*(\cB,\cA) \arrow[d, "G^*"] \\ \Fun_*(\cC^2,\cA) \arrow[r, "L"] & \Fun_*(\cC,\cA) & \Fun_*(\cC^2,\cA) \arrow[r, "L"] & \Fun_*(\cC,\cA)
\end{tikzcd}
\]
That is the diagonal functor commutes with both pre- and post-composition.
\end{lem}

In particular, $LR(F) \circ G \cong L(RF \circ G)$ and $F \circ LR(G) \cong L(F \circ RG)$. The left adjoint $L$ also preserves direct sums of functors as does the comonad $C=LR$, by Proposition \ref{p:cr is exact}. We are now prepared to embark upon the proof. 

\begin{proof}[Proof of Proposition \ref{lem:5.7}]
Our strategy, following the  proof of \cite[Lemma 5.7]{JM:Deriving}, is to show that the natural inclusions of $D_1(F \circ G)$ and $D_1F \circ D_1G$ into the double totalization of a first octant tricomplex $T(F,G)$ are chain homotopy equivalences.\footnote{Associativity of the chain complex monad asserts the the double totalization of a tricomplex is well-defined up to isomorphism.} Consider the tricomplex $T(F,G)$ defined by
\[ T(F,G)_{p,q,r} := (LR)^r( (LR)^pF\circ (LR)^qG)\] and its levelwise totalization
\[ A(F,G)_{n,r} := \bigoplus_{p+q=n}(LR)^r\left( (LR)^pF\circ (LR)^qG\right) \cong (LR)^r \left(\bigoplus_{p+q=n}  (LR)^pF\circ (LR)^qG\right).\]
The $n=0$ column of $A(F,G)_{n,r}$ is $D_1(FG)$ while the $r=0$ row is chain homotopy equivalent to  $D_1F \circ D_1G$ by Lemma \ref{lem:prolong}. The $r$th row is $(LR)^r(D_1F \circ D_1G)$.

Lemma \ref{lem:diagonal-commutativity} and the fact that $L$ preserves sums implies that the columns indexed by each $n >0$ are each the comonad resolution chain complexes for $LR$ applied to a functor of the form $LH$ for some functor $H$.

 Thus, our bicomplex $A(F,G)$ is of the form:
\[
\begin{tikzcd}
\vdots \arrow[d] & \vdots \arrow[d] & \vdots \arrow[d] & \vdots \arrow[d] & \\
C^3H_0 \arrow[d, "\epsilon-C\epsilon+C^2\epsilon"'] & C^3LH_1 \arrow[l, "C^3h_0"'] \arrow[d, "\epsilon-C\epsilon+C^2\epsilon"]& C^3LH_2 \arrow[l, "C^3h_1"'] \arrow[d, "\epsilon-C\epsilon+C^2\epsilon"]& C^3LH_3 \arrow[l, "C^3h_2"']\arrow[d, "\epsilon-C\epsilon+C^2\epsilon"]& \cdots \arrow[l]\\
C^2H_0 \arrow[d, "\epsilon-C\epsilon"'] & C^2LH_1 \arrow[l, "C^2h_0"] \arrow[d, "\epsilon-C\epsilon"]\arrow[u, dashed, bend left, "C^2L\eta"] & C^2LH_2 \arrow[l, "C^2h_1"] \arrow[d, "\epsilon-C\epsilon"]\arrow[u, dashed, bend left, "C^2L\eta"] & C^2LH_3 \arrow[l, "C^2h_2"]\arrow[d, "\epsilon-C\epsilon"]\arrow[u, dashed, bend left, "C^2L\eta"] & \cdots \arrow[l]\\
CH_0 \arrow[d, "\epsilon"'] & CLH_1 \arrow[l, "Ch_0"] \arrow[d, "\epsilon"] \arrow[u, dashed, bend left, "CL\eta"] &  \arrow[u, dashed, bend left, "CL\eta"]  CLH_2 \arrow[l, "Ch_1"] \arrow[d, "\epsilon"]&  \arrow[u, dashed, bend left, "CL\eta"] CLH_3 \arrow[l, "Ch_2"]\arrow[d, "\epsilon"]& \cdots \arrow[l]\\
H_0 & LH_1 \arrow[l, "h_0"] \arrow[u, dashed, bend left, "L\eta"] & LH_2 \arrow[l, "h_1"]  \arrow[u, dashed, bend left, "L\eta"] & LH_3 \arrow[l, "h_2"]  \arrow[u, dashed, bend left, "L\eta"] & \cdots \arrow[l]
\end{tikzcd}
\]
In particular, each column except for the 0th is contractible, using the unit $\eta$ of the adjoint $L \dashv R$. Applying Corollary \ref{cor:chain-htpy-equiv} it follows immediately that the canonical inclusion $D_1(FG) \hookrightarrow \Tot(A(F,G))$ is a chain homotopy equivalence.

By Corollary \ref{cor:chain-htpy-equiv}, to show that $D_1F\circ D_1G\hookrightarrow \Tot(A(F,G))$ is a chain homotopy equivalence it suffices to prove that the rows $C^r(D_1F \circ D_1G)$ of $A(F,G)$ are also contractible for each $r >0$. By exactness of $C$, it suffices to prove that the first row $C(D_1F \circ D_1G)$ is contractible.

By Proposition \ref{prop:poly-facts}\eqref{itm:degree}, the chain complexes $CD_1G(X)$ and $CD_1F(Y)$ are contractible. It follows that the inclusions
\[ D_1G(X) + D_1G(X) \hookrightarrow D_1G(X + X) \qquad D_1F(Y) + D_1F(Y) \hookrightarrow D_1F(Y+Y),\] 
which admit $CD_1G(X)$ and $CD_1F(Y)$ as direct sum complements, are chain homotopy equivalences.

By Lemma \ref{lem:che-comp}, these chain homotopy equivalences are preserved by pre- and post-composition. In particular, there is a composite chain homotopy equivalence
\[ D_1F(D_1G(X)) + D_1F(D_1G(X)) \hookrightarrow D_1F(D_1G(X) + D_1G(X)) \to D_1F(D_1G(X +X)).\]
It follows that the direct sum complement $C(D_1F\circ D_1G)(X)$ is contractible. \end{proof}

\end{document}